\documentclass{amsart}
\usepackage{amssymb,amsfonts,latexsym,amscd}

\usepackage[all]{xy}
\usepackage{verbatim}
\usepackage{hyperref}
\usepackage{mathrsfs}
\usepackage{mathtools}
\usepackage{amssymb}
\usepackage{stmaryrd}
\usepackage{tikz-cd}

\usepackage{xcolor}

\newtheorem*{rep@theorem}{\rep@title}
\newcommand{\newreptheorem}[2]{%
	\newenvironment{rep#1}[1]{%
		\def\rep@title{#2 \ref{##1}}%
		\begin{rep@theorem}}%
		{\end{rep@theorem}}}
\makeatother
\newreptheorem{theorem}{Theorem}

\numberwithin{equation}{section}

\newtheorem{theorem}{Theorem}[section]
\newtheorem{lemma}[theorem]{Lemma}
\newtheorem{proposition}[theorem]{Proposition}
\newtheorem{corollary}[theorem]{Corollary}
\theoremstyle{definition}
\newtheorem{definition}[theorem]{Definition}

\newtheorem{example}[theorem]{Example}
\newtheorem{question}[theorem]{Question}

\newtheorem{remark}[theorem]{Remark}
\newtheorem{remarks}[theorem]{Remarks}
\newtheorem{notation}[theorem]{Notation}

\newcommand{\ot}{\otimes}

\newcommand{\ben}{\begin{enumerate}}
\newcommand{\een}{\end{enumerate}}

\begin{document}

\title[Twisted unipotent groups]{Twisted unipotent groups}

\author{Ken A. Brown}
\address{School of Mathematics and Statistics,
University of Gasgow,
Gasgow G12 8QQ, United Kingdom}
\email{Ken.Brown@glasgow.ac.uk}

\author{Shlomo Gelaki}
\address{Department of Mathematics,
Iowa State University, Ames, IA 50100, USA.}
\email{gelaki@iastate.edu}

\date{\today}

\begin{abstract}
We study the algebraic structure and representation theory of the Hopf algebras ${}_J\mathcal{O}(G)_J$ when $G$ is an affine algebraic unipotent group over $\mathbb{C}$ with $\mathrm{dim}(G) = n$ and $J$ is a Hopf $2$-cocycle for $G$. The cotriangular Hopf algebras ${}_J\mathcal{O}(G)_J$ have the same coalgebra structure as $\mathcal{O}(G)$ but a deformed multiplication. We show that they are involutive $n$-step iterated Hopf Ore extensions of derivation type. The 2-cocycle $J$ has as support a closed subgroup $T$ of $G$, and ${}_J\mathcal{O}(G)_J$ is a crossed product $S \#_{\sigma}U(\mathfrak{t})$, where $\mathfrak{t}$ is the Lie algebra of $T$ and $S$ is a deformed coideal subalgebra. The simple ${}_J\mathcal{O}(G)_J$-modules are stratified by a family of factor algebras ${}_J\mathcal{O}(Z_g)_J$, parametrised by the double cosets $TgT$ of $T$ in $G$. The finite dimensional simple ${}_J\mathcal{O}(G)_J$-modules are all 1-dimensional, so form a group $\Gamma$, which we prove to be an explicitly determined closed subgroup of $G$. A selection of examples illustrate our results.
\end{abstract}

\maketitle

\tableofcontents

\section{Introduction}\label{intro} 
 
Let $G$ be an affine algebraic group over $\mathbb{C}$, and let $\mathcal{O}(G)$ be its algebra of polynomial functions, an affine commutative Hopf algebra. Let $J \in \left(\mathcal{O}(G) \otimes \mathcal{O}(G)\right)^*$ be a Hopf $2$-cocycle for $G$ (see Definition \ref{twococyc}). The twisting procedure of Drinfeld \cite{Dr2} yields a new cotriangular Hopf algebra structure on the underlying coalgebra of $\mathcal{O}(G)$ by using $J$ to deform the multiplication, and change the $R$-form $\epsilon \otimes \epsilon$ of $\mathcal{O}(G)$ to  
$J_{21}^{-1}\ast J$. See  $\S$\ref{cocycles} for the twisting procedure, and $\S$\ref{cotri1} for background on $R$-forms and cotriangular structures. Let $\mathrm{Rep}(G)$ denote the category of finite dimensional complex rational representations of $G$. Then Hopf $2$-cocycles for $G$ yield tensor structures on the forgetful functor $\mathrm{Rep}(G) \rightarrow \mathrm{Vec}$, as explained for example in \cite[Chapter 5]{EGNO}. 

The purpose of this paper is to study the algebraic structure and representation theory of the Hopf algebras ${}_J\mathcal{O}(G)_J$ when the group $G$ is unipotent. In particular, this paper extends (and corrects some of) the results of \cite{G2}. It thus forms part of a sequence of papers which began with Movshev's work \cite{Mo} on twisting for finite groups, and continued with  \cite{EG1,EG2,EG3,EG4,EG5,EG6} of Etingof and the second author, and \cite{G0,G1,G2} of the second author. As is recalled in detail in 
$\S$\ref{2cocycunip}, by results due to Drinfeld, Etingof and the second author, the equivalence classes of Hopf $2$-cocycles for a unipotent group $G$ with Lie algebra $\mathfrak{g}$ are in bijection with equivalence classes of solutions to the classical Yang-Baxter equation in $\wedge^2 \mathfrak{g}$, and to equivalence classes of pairs $(\mathfrak{t}, \omega)$, where $\mathfrak{t}\subseteq \mathfrak{g}$ is a quasi-Frobenius Lie subalgebra with nondegenerate $2$-cocycle $\omega\in \wedge^2 \mathfrak{t}^*$.

\subsection{Hopf algebraic structure}\label{subsec1.2} The following result summarises the most basic structural features of the deformations ${}_J\mathcal{O}(G)_J$ for $G$ unipotent. Early versions of parts (1), (2)  and (4) were obtained in \cite{G2}. For (6) recall that an ideal $I$ of a ring $A$ is \emph{polycentral} if it contains elements $y_1, \ldots , y_t$ such that $I = \sum_j y_j A$ with $y_1$ central in $A$ and $y_j$ central \emph{modulo} $\sum_{i < j} y_i A$ for $j = 2, \ldots , t$. 

\begin{theorem}\label{structureintro}{\rm (Theorems \ref{Ore}, \ref{augcent}, Corollaries \ref{nIHOE}, \ref{IHOEcor})} Let $G$ be a unipotent algebraic group over $\mathbb{C}$ of dimension $n$, and let $J$ be a Hopf $2$-cocycle for $G$.
\begin{enumerate}
\item $_J\mathcal{O}(G)_J$ is an $n$-step iterated Hopf Ore extension of derivation type. That is, there is a sequence 
$$ \mathbb{C} = H_0 \subset \cdots \subset {}_J(H_i)_J \subset \cdots \subset {}_J(H_n)_J ={}_J\mathcal{O}(G)_J $$
of Hopf subalgebras ${}_J(H_i)_J\subseteq {}_J\mathcal{O}(G)_J$, and derivations $\partial_i$ of ${}_J(H_i)_J$, such that ${}_J(H_{i+1})_J = {}_J(H_{i})_J[X_{i+1} ; \partial_i]$ is an Ore extension for all $0\le i \le n-1$. 
\item The formulae for the derivations $\partial_i$ can be explicitly stated in terms of the Hopf $2$-cocycle $J$ and the comultiplication $\Delta$ of $\mathcal{O}(G)$.
\item The associated graded algebra of ${}_J\mathcal{O}(G)_J$ with respect to its coradical filtration is the coordinate ring $\mathcal{O}(\widehat{G})$, the associated graded ring of $\mathcal{O}(G)$, $\widehat{G}$ being the associated graded unipotent group of $G$.
\item ${}_J\mathcal{O}(G)_J$ is a noetherian domain of Gelfand-Kirillov dimension $n$ with excellent homological properties.
\item Every prime ideal $P$ of ${}_J\mathcal{O}(G)_J$ is completely prime, i.e., ${}_J\mathcal{O}(G)_J/P$ is a domain.
\item The augmentation ideal $({}_J\mathcal{O}(G)_J)^+$ of ${}_J\mathcal{O}(G)_J$ is polycentral. \qed
\end{enumerate}
\end{theorem}

Given a cotriangular Hopf algebra $(H,R)$, we say that the $R$-form $R$ is \emph{nondegenerate}, and that $(H,R)$ is \emph{minimal}, if the \emph{radical} $$I_R := \{a \in H \mid R(a,b) = 0, \, \forall b \in H \}$$ is $\{0\}$. The radical is a Hopf ideal of $H$, and there is an induced structure of cotriangular Hopf algebra $(H/I,R)$, on which $R$ is nondegenerate, so that $(H/I,R)$ is minimal \cite[\S 2.2]{G1}; see $\S$\ref{cotri1} and Lemma \ref{radical} for details.
 
Specialise now to the case where $H$ is the coordinate ring $\mathcal{O}(G)$ of an affine algebraic group $G$, so that, as already mentioned above, $(\mathcal{O}(G),\epsilon \otimes \epsilon)$ is cotriangular. Let $J$ be a Hopf $2$-cocycle for $G$, so that, as also noted above, $({}_J\mathcal{O}(G)_J,R^J)$ is cotriangular with $R^J := J_{21}^{-1}\ast J$. Extending the terminology introduced above, we say that $J$ is \emph{minimal} if $({}_J\mathcal{O}(G)_J, R^J)$ is minimal. We can now appeal to the following result (where, for \emph{gauge equivalence}, see Remark \ref{Jrems}(3)): 

\begin{proposition}{\rm (\cite[Theorem 3.1]{G1})}
Let $G$ be an affine algebraic group, and let  
$J$ be a Hopf $2$-cocycle for $G$. Then there is a closed subgroup $T\subseteq G$ and a minimal Hopf $2$-cocycle $\widehat{J}$ for $T$, such that $J$ is gauge equivalent to $\widehat{J}$ (where $\widehat{J}$ is viewed in the obvious way as a Hopf $2$-cocycle for $G$). Hence, there is an epimorphism of cotriangular Hopf algebras from $({}_J\mathcal{O}(G)_J,R^J)$ to $({}_{\widehat{J}}\mathcal{O}(T)_{\widehat{J}},R^{\widehat{J}})$ with kernel $I_{R^{J}}$, the radical of $R^J$. \qed
\end{proposition}

Here, the closed subgroup $T\subseteq G$, which is uniquely determined by $J$ up to conjugation by $G$, is called the \emph{support} of $J$.

Assume once again that $G$ is unipotent, and let $J$ be a Hopf $2$-cocycle for $G$ with support $T$, with $R^J = J_{21}^{-1}\ast J$ the corresponding $R$-form for ${}_J\mathcal{O}(G)_J$. There is an algebra homomorphism $\Psi$  from ${}_J\mathcal{O}(G)_J$ to the vector space dual $\mathcal{O}(G)^{\ast}$, given by 
$$ \Psi (h):= R^J(-,\, h), $$
as shown in Lemma \ref{into1}. The kernel of $\Psi$ is the radical $I_{R^J}$ of $R^J$, so one sees that $\mathrm{im}(\Psi)$ is a Hopf subalgebra of $({}_J\mathcal{O}(T)_J)^{\circ,\mathrm{op}}$. In fact, $\mathrm{im}(\Psi)$ is also contained in the finite dual $\mathcal{O}(T)^{\circ}$ of $\mathcal{O}(T)$, which by the Cartier-Gabriel-Kostant theorem is isomorphic to $U(\mathfrak{t})\# \mathbb{C}[T]$, where $\mathfrak{t}$ denotes the Lie algebra of $T$, and we prove that $\mathrm{im}(\Psi) = U(\mathfrak{t})$ 
(as predicted in \cite[Theorem 3.4]{G2}; see Remark \ref{compwg22}). These results are Proposition  \ref{include} and Theorem \ref{noethdomminunip2}, which together yield Theorem \ref{minLie} and \cite[Lemma 3.7, Theorems 3.5, 3.8]{G2}, summarised as follows. Here, and throughout, for a subspace $V$ of a Hopf algebra, $V^+$ denotes the kernel of $\epsilon_{\mid V}$.

\begin{theorem}\label{structurethm} 
Let $J$ be a Hopf $2$-cocycle for a unipotent affine algebraic group $G$, let $T$ be (a representative of) the support of $J$, and let $\mathfrak{t}$ be the Lie algebra of $T$. Denote the antipode of ${}_J\mathcal{O}(G)_J$ by $S^J$.
\begin{enumerate}
\item {\rm (Theorem \ref{crossed})} ${}_J\mathcal{O}(G)_J$ decomposes as a crossed product algebra,
$$ {}_J\mathcal{O}(G)_J \cong ({}_J\mathcal{O}(G/T)_J)\#_{\sigma}{}_J\mathcal{O}(T)_J \cong ({}_J\mathcal{O}(G/T)_J)\#_{\sigma} U(\mathfrak{t})^J,$$
where ${}_J\mathcal{O}(G/T)_J={}_J\mathcal{O}(G/T)$ is the twist of the left coideal subalgebra $\mathcal{O}(G/T)$ consisiting of the functions in $\mathcal{O}(G)$ which are constant on each left coset of $T$ in $G$, where $\sigma$ is an invertible $2$-cocycle from ${}_J\mathcal{O}(T)_J$ to ${}_J\mathcal{O}(G/T)_J$, and where $U(\mathfrak{t})^J$ is the enveloping algebra $U(\mathfrak{t})$ with a twisted comultiplication.
\item {\rm (Lemma \ref{basics})} $ I_{R^J}=({}_J\mathcal{O}(G/T)_J)^{+}{}_J\mathcal{O}(G)_J= ({}_J\mathcal{O}(T\backslash G)_J)^{+}{}_J\mathcal{O}(G)_J.$
\item {\rm (Proposition \ref{centre})} The subalgebra $\mathcal{O}(T\backslash G / T)$ of $\mathcal{O}(G)$  consisting of functions which are constant on double cosets of $T$ in $G$ is contained in the centre of ${}_J\mathcal{O}(G)_J$. 
\item {\rm (Corollary \ref{payoff})} $(S^J)^2 = \mathrm{id}$, and ${}_J\mathcal{O}(G)_J$ is Calabi-Yau. \qed
\end{enumerate} 
\end{theorem}

\subsection{Representation theory}\label{reptheory} 
With regard to representation theory, our focus in this paper is on the simple ${}_J\mathcal{O}(G)_J$-modules. In view of  Theorem \ref{structurethm}(3), Schur's lemma shows that each simple ${}_J\mathcal{O}(G)_J$-module is annihilated by $\mathfrak{m}{}_J\mathcal{O}(G)_J$ for a unique maximal ideal $\mathfrak{m}$ of $\mathcal{O}(T\backslash G / T)$. Building on this observation we have the following stratification result, which one can view as generalising the fact that the simple $\mathcal{O}(G)$-modules are parametrised by the elements of $G$.

\begin{theorem}\label{stratintro} {\rm (\cite[\S 4, Proposition 4.1, Corollary 4.2, Theorem 5.1]{G2})} 
Let $G$ be a unipotent affine algebraic group, and let $J$ be a Hopf $2$-cocycle for $G$ with support $T$. Let $g \in G$, and let $Z_g:= TgT$ be the double coset of $T$ in $G$.
\begin{enumerate}
\item{\rm (Lemma \ref{doubles}(1))} 
$Z_g$ is a closed irreducible subset of $G$ with
$$\mathrm{dim}(Z_g) =  2\mathrm{dim}(T) - \mathrm{dim}(T \cap gTg^{-1}).$$
\item{\rm (Proposition \ref{homo})} 
The defining ideal $\mathcal{I}(Z_g)$ of $Z_g$ in $\mathcal{O}(G)$ is also an ideal of ${}_J\mathcal{O}(G)_J$.
\item{\rm (Corollary \ref{strat2})} Define ${}_J\mathcal{O}(Z_g)_J := {}_J\mathcal{O}(G)_J/\mathcal{I}(Z_g)$. Then ${}_J\mathcal{O}(Z_g)_J$ is an affine noetherian domain with
$$\mathrm{GKdim}({}_J\mathcal{O}(Z_g)_J )= 2 \mathrm{dim}(T) - \mathrm{dim}(T \cap gTg^{-1}).$$
\item{\rm (Theorem \ref{strata})} Each simple ${}_J\mathcal{O}(G)_J$-module factors through a unique quotient ${}_J\mathcal{O}(Z_g)_J$. \qed
\end{enumerate}
\end{theorem}

In view of Theorem \ref{stratintro}(4) it is natural to explore the representation theory and algebra structure of the algebras ${}_J\mathcal{O}(Z_g)_J$ for $g \in G$. There is a dichotomy according to whether or not $g$ belongs to the normaliser $N := N_G(T)$ of $T$. We first consider the strata $Z_g$ with $g \notin N$ - here there is incomplete information at present except when $g$ is ``very far from normalising'', (see (2)):

\begin{theorem}\label{Zstructintro} Keep the notation of Theorem \ref{stratintro}, let $N := N_G(T)$, and let $g \in G$.
\begin{enumerate}
\item {\rm (Theorem \ref{findims})} 
If $g \notin N$ then ${}_J\mathcal{O}(Z_g)_J$ has no finite dimensional modules.
\item{\rm (\cite[Remark 4.4]{G2}, Theorem \ref{caseIstruc})} If $T \cap gTg^{-1} = \{1\}$ then ${}_J\mathcal{O}(Z_g)_J$ is isomorphic to the Weyl algebra $A_{\mathrm{dim}(T)}(\mathbb{C})$. \qed
\end{enumerate}
\end{theorem} 
 
Thus finite dimensional simple modules can only occur in the strata $Z_g$ where $g \in N$. We consider now these cases. As an immediate corollary of Theorem \ref{structureintro}(5), when $G$ is unipotent every finite dimensional simple ${}_J\mathcal{O}(G)_J$-module is $1$-dimensional. The set of such modules therefore forms an algebraic group $\Gamma$, whose algebra of regular functions is the abelianisation ${}_J\mathcal{O}(G)_J/\langle [{}_J\mathcal{O}(G)_J, {}_J\mathcal{O}(G)_J]\rangle$ of ${}_J\mathcal{O}(G)_J$. The following theorem summarises what we know about the group $\Gamma$. (In particular, it corrects \cite[Theorem 5.2]{G2}.) Note first, as explained in Remark \ref{Jrems}(3), that the group $(\mathcal{O}(G)^*)^{\times}_1$ of convolution invertible elements of $\mathcal{O}(G)^*$ which map $1$ to $1$ acts on the set of Hopf $2$-cocycles for $G$. Denote this action by $J \mapsto J^{\chi}$ for $\chi \in (\mathcal{O}(G)^*)^{\times}_1$. Of course, $G$ itself, and also the group $\Gamma$, are subgroups of  $(\mathcal{O}(G)^*)^{\times}_1$.

\begin{theorem}\label{findimintro}{\rm (Theorem \ref{findims})} 
Keep the notation $G$, $J$, $T$ and $N$ of Theorem \ref{Zstructintro}, and let $F$ and $\Gamma$ be the groups of finite dimensional simple ${}_J\mathcal{O}(T)_J$ and ${}_J\mathcal{O}(G)_J$-modules, respectively. Let $[J]$ denote the orbit of $J$ under the action of the subgroup $(\mathcal{O}(T)^*)^{\times}_1$ of $(\mathcal{O}(G)^*)^{\times}_1$. Define  $$C_0 := \{ g \in N\mid J^g = J \}\,\,\, {\rm and} \,\,\,C := \{ g \in N \mid [J^g] = [J]\},$$ so that $C_0 \subseteq C$, $T \triangleleft C$, and hence $C_0 T \subseteq C$.
\begin{enumerate}
\item $\Gamma = C_0,$ and $C_0 = \{g \in G \mid J^g = J\}$.
\item $C = C_0T.$
\item $F=C_0 \cap T\subseteq T$ is a closed abelian subgroup of $T$, normal in $C_0$, with 
 $$ \mathrm{dim}(F)= \mathrm{dim}(T/[T,T]).$$
\item $C_0/F \cong C/T$, so that there is an exact sequence of unipotent groups
$$ 1 \rightarrow F \rightarrow \Gamma \rightarrow C/T \rightarrow 1.$$
\item $\mathrm{dim}(\Gamma)=\mathrm{dim}(C) - \mathrm{dim}([T,T]).$
\item The algebra ${}_J\mathcal{O}(Z_g)_J$ has nonzero finite dimensional representations if and only if $g\in C$. The distinct strata admitting nonzero finite dimensional modules are $\{{}_J\mathcal{O}(Z_g)_J \mid g \in \mathcal{C}\},$ where $\mathcal{C}$ is a set of coset representatives for $F=C_0 \cap T$ in $C_0$, or equivalently for $T$ in $C$.
\item For any $g \in \mathcal{C}$, the set of finite dimensional simple ${}_J\mathcal{O}(Z_g)_J$-modules is  
parametrized by $gF=g(C_0 \cap T)$. Specifically, the corresponding functionals have kernels the maximal 
ideals $\{\mathfrak{m}_{gx} \mid x\in C_0 \cap T \}$ of $\mathcal{O}(G)$. \qed
\end{enumerate}
\end{theorem}

\subsection{Special cases and examples}\label{examplesintro}
In $\S$\ref{examples} we consider the special case when the support $T$ of $J$ is abelian. Finally we expand \cite[Examples 6.1-6.5]{G2} to illustrate our results. 

\subsection{Notation}\label{notation} 

We work throughout over the complex numbers $\mathbb{C}$, although many of our results are valid over an arbitrary algebraically closed field of characteristic $0$. All unadorned tensor products are assumed to be over $\mathbb{C}$. The notation $(H, m_H, 1_H, \epsilon_H, \Delta_H, S_H)$ will always denote a Hopf $\mathbb{C}$-algebra and its standard operations, although we omit the suffix $H$ whenever possible. The augmentation ideal of $H$ is denoted by $H^+$, extending this usage by writing $A^+$ for $H^+ \cap A$ when $A\subset H$ is a subalgebra. We denote the finite dual of $H$ by $H^{\circ}$ \cite[$\S$9.1]{M}, and use $\ast$ to denote convolution product.

\section{Preliminaries}\label{prelims} 

\subsection{Hopf $2$-cocycles}\label{cocycles}

Let $H$ be a Hopf $\mathbb{C}$-algebra.

\begin{definition}\label{twococyc}(Doi \cite{D}, Drinfeld \cite{Dr}, Majid \cite[$\S\S$2.1,2.2,2.3]{Ma1}) A \emph{Hopf $2$-cocycle} for $H$ is a convolution-invertible bilinear map $J: H \otimes H \rightarrow \mathbb{C}$ such that
\begin{equation}\label{Jone} \sum J(a_1b_1,c)J(a_2,b_2) = \sum J(a, b_1c_1)J(b_2,c_2)
\end{equation}
and
\begin{equation}\label{Jtwo} J(a,1)= \epsilon (a)= J(1,a),
\end{equation}
for all $a,b,c \in H$. If $G$ is an affine algebraic group and $H:=\mathcal{O}(G)$ is the algebra of polynomial functions, then $J$ is often referred to as a \emph{Hopf $2$-cocycle for $G$}. \qed
\end{definition}

Let $J$ and $K$ be two Hopf $2$-cocycles for $H$. Then one can construct a new algebra $_K H_J$ as follows: as vector spaces, $_K H_J = H$, but the new multiplication $_K m_J$ is given for $a,b \in H$ by
\begin{equation}\label{mult0} _K m_J(a\otimes b)=\sum K^{-1}(a_1,b_1)a_2b_2 J(a_3,b_3).
\end{equation}
In particular there are special cases when $K$ or $J$ is $\epsilon \otimes \epsilon$, denoted respectively by $H_J$ and by $_K H$, with respective multiplications
\begin{equation}\label{mult1} m_J(a\otimes b)  =  \sum a_1b_1 J(a_2,b_2)\,\, \textit{  and  }\,\, _K m(a\otimes b)  =  \sum K^{-1}(a_1,b_1)a_2b_2.
\end{equation}
In the most important case $K = J$, $_J H_J$ is a Hopf algebra, with coalgebra structure unchanged from $H$, but with multiplication
\begin{equation}\label{mult} {}_J m_J(a\otimes b)= \sum J^{-1}(a_1,b_1)a_2b_2 J(a_3,b_3),
\end{equation}
and antipode
\begin{equation}\label{twistanti}  S^J (a)=\sum J^{-1}(a_1,S(a_2))S(a_3)J(S(a_4),a_5),
\end{equation}
for all $a,b \in H$. To avoid needless clutter we will abbreviate the notation for the multiplication in ${}_JH_J$ to $m_J$ from now on, when no confusion seems likely. We record the following basic properties, which will be needed later. 

\begin{lemma}\label{action0} Let $H$, $K$ and $J$ be as above, and let $G\subset (H^*)^{\times}$ be a subgroup.
\begin{enumerate}
\item ${}_K H_J$ is a right and left $G$-module under the respective actions defined by 
\begin{equation}\label{action}   
f\cdot g = \sum g(f_1)f_2\,\, \textit{   and   }\,\, g\cdot f  = \sum f_1 g(f_2);\,\,\,f\in {}_K H_J,\, g \in G.
\end{equation}
\item Suppose that $H = \mathcal{O}(G)$ for a group $G$. Then the first of the actions in (1) defines an action of $G$ by algebra automorphisms on $\mathcal{O}(G)_J$, and the second yields algebra automorphisms of ${}_J\mathcal{O}(G)$.
\item $\Delta$ gives an algebra embedding of $_KH_J$ into ${}_K H \otimes H_J$.
\end{enumerate}
\end{lemma}

\begin{proof}
(1) This is clear.

(2) We check the first statement, the second being similar. Use $\cdot$ to denote multiplication in $\mathcal{O}(G)_J$. As in $(1)$, use $g(-)$ to denote evaluation at $g$, and $-\cdot g$ for the right $G$-module action on the (vector space) $\mathcal{O}(G)$. For $f,h \in \mathcal{O}(G)$ and $g \in G$,
\begin{align*} 
(f\cdot h)\cdot g &=\left(\sum f_1 h_1J(f_2,h_2)\right)\cdot g \\
&= \sum (f_1  h_1)\cdot g J(f_2,h_2)
= \sum (f_1\cdot g) (h_1\cdot g) J(f_2, h_2) \\
&= \sum(g(f_1)f_2 g(h_1)h_2) J(f_3,h_3)= (f\cdot g)\cdot (h\cdot g).
\end{align*}

(3) This is \cite[Lemma 2.2]{G2}.
\end{proof}

\begin{remarks}\label{Jrems}
(1) If $J$ is a Hopf $2$-cocycle for $H$ and $A$ is a Hopf subalgebra of $H$, then $J_{|A \otimes A}$ is a Hopf $2$-cocycle for $A$.

(2) By (the dualised version of) a theorem of Schauenberg \cite{S} (see also \cite{EG5} for an alternative proof), for all finite dimensional Hopf algebras $H$ and $A$, we have $\mathrm{Corep}(H)\cong \mathrm{Corep}(A)$ as tensor categories if and only if $H$ and $A$ are related by a Hopf $2$-cocycle. This extends to all Hopf algebras, with basically the same proof.

(3) Let $J$ be a Hopf $2$-cocycle for a Hopf algebra $H$. So $J \in (H \otimes H)^{\ast}$ and there is an action of the group $(H^*)^{\times}$ of units of $H^{\ast}$ on the set of Hopf $2$-cocycles for $H$: for $\chi\in (H^*)^{\times}$ there is a Hopf $2$-cocycle $J^{\chi}$ of $H$, defined by 
$$ J^{\chi} (a,b):=\sum \chi(a_1b_1) J(a_2, b_2) \chi^{-1}(a_3)\chi^{-1}(b_3);\,\,\,a,b \in H.$$
Two Hopf $2$-cocycles for $H$ which are in the same orbit under $(H^*)^{\times}$ are said to be \emph{gauge equivalent}. If $J_1$ and $J_2$ are gauge equivalent Hopf $2$-cocycles for the finite dimensional Hopf algebra $H$, then ${}_{J_1} H_{J_1} \cong {}_{J_2}H_{J_2}$ as Hopf algebras, by the dualised version of \cite[page 115]{EGNO}. The same proof applies in the infinite dimensional setting.

(4) A Hopf $2$-cocycle $J$ for $H$ is called \emph{invariant} if $_JH_J=H$ as Hopf algebras \cite[p.2]{G2} (i.e., isomorphic via the identity map). The set of invariant Hopf $2$-cocycles for $H$ is a subgroup of $((H \otimes H)^*)^{\times}$, denoted $Z^2_{\mathrm{inv}}(H)$; in view of the previous remark, $Z^2_{\mathrm{inv}}(H)$ is normalized by $(H^*)^{\times}$. Invariant Hopf $2$-cocycles are studied in, for example, \cite{BC, EG4,S2}. \qed
\end{remarks}

\subsection{Cotriangular Hopf algebras}\label{cotri1}
A key point for us will be the preservation of the following property under cocycle deformation.

\begin{definition}\label{cotridef} (\cite[Definition 8.3.19]{EGNO}) A \emph{coquasitriangular} Hopf algebra  is a pair $(H,R)$, where $H$ is a Hopf $\mathbb{C}$-algebra and $R: H \otimes H \rightarrow \mathbb{C}$ is a convolution invertible bilinear map such that
\begin{enumerate}
\item 
$R(h, \ell g )=\sum R(h_1, g)R(h_2, \ell)\,\, \textit{  and  }\,\, R(gh, \ell) = \sum R(g, \ell_1)R(h, \ell_2);$
\item
$\sum R(h_1,g_1)h_2g_2=\sum g_1 h_1 R(h_2,g_2)$
\end{enumerate}
for $g,h, \ell \in H$. Such a functional $R$ is called an \emph{$R$-form} for $H$.

If in addition
$$ \textrm{(3)} \quad \sum R(h_1,g_1)R(g_2,h_2)=\epsilon(g)\epsilon (h), $$
that is, if $R^{-1} = R_{21}$, then $(H,R)$ is called \emph{cotriangular}. \qed
\end{definition}

There is a dual concept of a \emph{(quasi)triangular Hopf algebra} $(H,R)$, defined by Drinfeld \cite{Dr2}, \cite[Definition 8.3.1]{EGNO}, where now $R$, the \emph{universal $R$-matrix}, is an invertible element of $H \otimes H$ satisfying certain equations \cite[(8.8), page 198]{EGNO}, stemming from the braid relations. A finite dimensional Hopf algebra $(H,R)$ is easily seen to be (quasi)triangular if and only if $(H^*, R)$ is co(quasi)triangular \cite[page 201]{EGNO} (noting that $R \in H \otimes H$ can be regarded as an element of $(H^* \otimes H^*)^*$).

The following classes of examples are key for us:  

\begin{example}\label{commex}(1)(\cite[Example 8.3.4]{EGNO}) If $H$ is a commutative Hopf algebra, then $(H,\epsilon\otimes\epsilon)$ is cotriangular. Equivalently, any cocommutative Hopf algebra is triangular with universal $R$-matrix $1 \otimes 1$.

(2)(\cite[Remark 8.3.24]{EGNO}) Let $(H,R)$ be co(quasi)triangular, and let $J$ be a Hopf $2$-cocycle for $H$. Define $R^J := J_{21}^{-1}\ast R\ast J$. Then $R^J$ is an $R$-form for ${}_JH_J$, and we say that $R^J$ is obtained from $R$ by \emph{twisting by} $J$. Thus, $({}_J H_J, R^J)$ is a co(quasi)triangular Hopf algebra. \qed 
\end{example}

The following properties of the $R$-forms obtained as in Example \ref{commex}(2) will be needed later.

\begin{lemma}\label{quick} 
Let $G$ be an affine algebraic group, so $(\mathcal{O}(G),\epsilon \otimes \epsilon)$ is cotriangular. Let $J$ be a Hopf $2$-cocycle for $G$, and define $ R^J:=J_{21}^{-1}\ast J$. 
Then $R^J$ is an $R$-form for ${}_J\mathcal{O}(G)_J$ by Example \ref{commex}(2). 
\begin{enumerate}
\item For all $\beta \in \mathcal{O}(G)^+$, $R^J(\beta,1) = R^J(1,\beta) = 0$.
\item Let $p  \in \mathcal{O}(G)$ be primitive. Then for all $\alpha \in \mathcal{O}(G)$,
$$  R^J(p,\alpha)=(J - J_{21})(p,\alpha)= (J_{21}^{-1} - J^{-1})(p,\alpha) .$$
\end{enumerate}
\end{lemma}

\begin{proof} 
(1) Let $\beta \in \mathcal{O}(G)^+$. Then, by definition,  
$$R^J(\beta,1) = \sum J_{21}^{-1}(\beta_1,1)J(\beta_2,1),$$
and this equals 0 by equation (\ref{Jtwo}) in Definition \ref{twococyc}, since in every summand above, at least one of $\beta_1, \beta_2$ is in $\mathcal{O}(G)^+$.

(2) This is \cite[Lemma 2.3]{G2}.
\end{proof}

Notice that we can reformulate Definition \ref{cotridef}(2) as
$m^{\mathrm{op}}=R \ast m \ast R^{-1}$, 
thus exposing the relevance of $R$-forms for $H$ to the existence of braidings on the category $\mathrm{Comod}(H)$ of right $H$-comodules. Namely, recall the following, where (2) is the dual of \cite[Proposition 8.3.14]{EGNO}.

\begin{theorem}\label{cotri} {\rm (\cite[p. 200--202]{EGNO})} 
Let $H$ be a Hopf algebra, and let $\mathcal{C}$ be the tensor category of finite dimensional right $H$-comodules.
\begin{enumerate}
\item If $R$ is an $R$-form yielding a coquasitriangular structure on $H$, then the category $\mathcal{C}$ is a braided category, with braiding
\begin{equation}\label{braid} c_{X,Y}:= \tau \circ R^{24} \circ (\rho_X \otimes \rho_Y), 
\end{equation}
where $X,Y\in \mathcal{C}$ with coactions $\rho_X,\rho_Y$, respectively. Conversely, if $\mathcal{C}$ is braided, then this induces a natural coquasitriangular structure on $H$.

\item If $(H,R)$ is a coquasitriangular Hopf algebra and $J$ is a Hopf $2$-cocycle for $H$, then the categories of comodules over $(H,R)$ and $({}_J H_J, R^J)$ are naturally equivalent as braided categories. There is thus a bijection between braided structures on $\mathcal{C}$ and coquasitriangular structures on $H$ up to twisting by a Hopf $2$-cocycle; this bijection restricts to a bijection between symmetric structures on $\mathcal{C}$ and cotriangular structures on $H$, up to twisting by a Hopf $2$-cocycle. \qed
\end{enumerate}
\end{theorem}

The following lemma will be important in analyzing the structure of cotriangular Hopf algebras.

\begin{lemma}\label{radical}{\rm (\cite[Proposition 2.1]{G1})} 
Let $(H,R)$ be a cotriangular Hopf algebra, and let 
$I_R:=\{a \in H \mid R(b,a) = 0, \, \forall b \in H \}$ be 
the right radical of $R$. Then $I_R$ is also the left radical (with the obvious definition), and it is a Hopf ideal of $H$. \qed
\end{lemma}

\begin{definition}\label{min} A cotriangular Hopf algebra $(H,R)$ whose radical $I_R$ is $\{0\}$ is called \emph{minimal}, and we say in this case that $R$ is \emph{nondegenerate}. \qed
\end{definition}

Clearly, by Lemma \ref{radical}, a cotriangular Hopf algebra $(H,R)$ has a unique minimal cotriangular Hopf quotient, which is denoted by $(H_{\mathrm{min}}, R)$. 

\begin{lemma}\label{into1} 
Let $(H,R)$ be a cotriangular Hopf algebra. Then there are maps $\Psi_{\ell}, \Psi_r:H\to H^{\circ}$, defined respectively by $\Psi_{\ell}(a) = R(a, -)$ and $\Psi_r(a) = R(-,a)$ for $a \in H$. Each of these is a map of Hopf algebras with kernel $I_R$, mapping into (the isomorphic Hopf algebras) $H^{\circ, \, {\rm cop}}$ and  $H^{\circ,\, {\rm op}}$, respectively.
\end{lemma}

\begin{proof}
Clearly $\Psi_{\ell}$ and $\Psi_r$ are linear maps into $H^{\ast}$ with kernel $I_R$. To see that their images lie in $H^{\circ}$, let $c \in H$ and fix a finite dimensional subcoalgebra $C$ of $H$ with $c \in C$. Setting $I_C := \{h \in H\mid R(h,d) = 0,\, \forall d \in C\}$, one sees from the definition of $R$ that $I_C$ is an ideal of $H$ containing $I_R$, and that $H/I_C$ embeds into $C^{\ast}$ via $\Psi_{\ell}$. Since $\Psi_r (c)(I_C) = 0$, this proves that $\mathrm{im}(\Psi_{r}) \subset H^{\circ}$. Similarly, $\mathrm{im}(\Psi_{\ell}) \subset H^{\circ}$. 

Since for all $a,b,x \in H$,
$$ \Psi_{\ell}(ab)(x) = R(ab,x) = \sum R(a,x_1)R(b,x_2) = (\Psi_{\ell}(a)\ast\Psi_{\ell}(b))(x), $$
$\Psi_{\ell}$ is an algebra homomorphism from $H$ to $H^{\circ}$. Moreover, the first part of Definition \ref{cotridef}(1) shows that $\Psi_{\ell}$ is a bialgebra homomorphism into $H^{\circ, \, {\rm cop}}$. 

Recall that by \cite[Lemma 1.5.11]{M}, $H^{\circ, \, {\rm cop}}$ has antipode $(S^{\ast})^{-1}$, the composition inverse of the antipode $S^{\ast}$ of $H^{\circ}$. That is, 
$$(S^{\ast})^{-1}f(x)= f(S^{-1}(x)) $$
for all $f \in H^{\circ,\, {\rm cop}}$ and $x \in H$. Thus, for $h,x \in H$, 
$$ \Psi_{\ell}(S(h)) = R(S(h),x) = R(h, S^{-1}(x)) = (S^{\ast})^{-1}(\Psi_{\ell}(h)), $$
proving that $\Psi_{\ell}$ is a Hopf algebra map from $H$ into $H^{\circ, \, {\rm cop}}$.

The argument for $\Psi_r$ is similar, noting that $H^{\circ, \, {\rm op}}$ has antipode $(S^{\ast})^{-1}$. Note that the two codomain Hopf algebras are isomorphic by \cite[Corollary III.3.5]{K}.
\end{proof}

Specialise now to the case where $H$ is the coordinate ring $\mathcal{O}(G)$ of an affine algebraic group $G$, so that $(\mathcal{O}(G),\epsilon \otimes \epsilon)$ is cotriangular by Example \ref{commex}(1). Let $J$ be a Hopf $2$-cocycle for $G$, so $({}_J\mathcal{O}(G)_J,R^J)$ is cotriangular by Example \ref{commex}(2). Extending the terminology of Definition \ref{min}, we say that $J$ is \emph{minimal} if $({}_J\mathcal{O}(G)_J, R^J)$ is minimal. We then have:

\begin{proposition}\label{factor} {\rm (\cite[Theorem 3.1]{G1})} 
Let $G$ be an affine algebraic group, and let $J$ be a Hopf $2$-cocycle for $G$. Then there is a closed subgroup $T\subseteq G$ and a minimal Hopf $2$-cocycle $\widehat{J}$ for $T$, such that $J$ is gauge equivalent to $\widehat{J}$ (where $\widehat{J}$ is viewed in the obvious way as a Hopf $2$-cocycle for $G$).
In particular, there is an epimorphism of cotriangular Hopf algebras from $({}_J\mathcal{O}(G)_J,R^J)$ to $({}_{\widehat{J}}\mathcal{O}(T)_{\widehat{J}},R^{\widehat{J}})$ with kernel $I_{R^{J}}$, the radical of $R^J$. \qed
\end{proposition}

\begin{definition}\label{supp} The (conjugacy class of the) closed subgroup $T$ of Proposition \ref{factor} is called the \emph{support} of $J$. \qed
\end{definition}

\subsection{Unipotent groups}\label{unipotent} Let $G$ be a unipotent algebraic group of dimension $n$ over $\mathbb{C}$. We recall here some basic facts, and introduce some notation needed in the sequel (see \cite[2.5]{G2}). Fix a chain of closed normal subgroups $G_i$  of $G$, for $0 \leq i \leq n$, with $G_0 = \{1\}$ and $G_n = G$, and with $G_{i+1}/G_i$ central in $G/G_i$ and isomorphic to $(\mathbb{C},+)$, for $i = 0, \ldots, n-1$. This corresponds to a chain of Hopf subalgebras 
\begin{equation}\label{poly} \mathbb{C} = H_0 \subset \cdots \subset H_i \subset \cdots \subset H_n = \mathcal{O}(G)
\end{equation}
of $\mathcal{O}(G)$, with $H_i \cong \mathcal{O}(G/G_{n-i})$ and $\mathcal{O}(G)/H_i^+ \mathcal{O}(G) \cong \mathcal{O}(G_{n-i})$ for all $i$. We can fix elements $X_1, \ldots, X_n$ of $\mathcal{O}(G)$ such that
$$ H_i = \mathbb{C}[X_1, \ldots , X_i] $$
for all $i$. By construction, and by the basic theory of affine algebraic groups (or by \cite{BOZZ} or \cite{H}), for all $i = 1, \ldots , n$, we have 
\begin{equation}\label{good} 
\Delta (X_i) = X_i \otimes 1 + 1 \otimes X_i + q(X_i),
\end{equation}
where 
\begin{equation}\label{qform} q(X_i):= \sum x^i_1 \otimes x^i_2 \in H_{i-1}^+ \otimes H_{i-1}^+ 
\end{equation}
is a coalgebra $2$-cocycle. In $\S$\ref{IHOE} we will also need the notation 
$$ (\mathrm{id} \otimes \Delta)(q(X_i))= \sum x^i_1 \otimes x^i_{21} \otimes x^i_{22}.\footnote{The notation $q(X_i)= \sum X_i' \otimes X_i''$, and $(\mathrm{id} \otimes \Delta)(q(X_i))= \sum X_i' \otimes X_{i1}'' \otimes X_{i2}''$ is used in \cite[\S 2.5]{G2}.}$$

In studying cocycle twists of unipotent groups we will make use of the embedding of Lemma \ref{action0}(3) into the  tensor product of one sided twisted algebras. Here is the key result for these one-sided twists of unipotent groups. As usual, $A_n(\mathbb{C})$ denotes the $n$th Weyl $\mathbb{C}$-algebra, a simple Noetherian domain whose Gelfand-Kirillov dimension $\mathrm{GKdim}(A_n(\mathbb{C}))$ is $2n$ \cite[Proposition 8.1.15(2)]{McCR}.

\begin{theorem}\label{Weyl} {\rm (\cite[Theorem 4.7]{G1})} Let $G$ be a unipotent algebraic group of dimension $n$ with Hopf $2$-cocycle $J$, and let $T$ be the support of $J$.
\begin{enumerate}
\item As algebras,
$$ \mathcal{O}(G)_J\cong\mathcal{O}(T)_J \otimes \mathcal{O}(G/T),$$
where $G/T$ denotes the space of left cosets of $T$ in $G$, so $\mathcal{O}(G/T)$ is the commutative polynomial algebra in $\mathrm{dim}(G) - \mathrm{dim}(T)$ variables.
\item $\mathcal{O}(T)_J\cong A_{\frac{\mathrm{dim}(T)}{2}}(\mathbb{C})$ as algebras.
\item $\mathcal{O}(G)_J$ is a noetherian domain of GK-dimension $\mathrm{dim}(G)$.
\item $(\mathcal{O}(G)_J)^{\mathrm{op}} \cong {}_J\mathcal{O}(G)$. \qed
\end{enumerate}
\end{theorem}

\subsection{Hopf $2$-cocycles for unipotent groups}\label{2cocycunip}  
A \emph{quasi-Frobenius Lie algebra} is a finite dimensional Lie algebra $\mathfrak{t}$ equipped with a symplectic $2$-cocycle $\omega$. That is, $\omega:\mathfrak{t} \times \mathfrak{t}\to \mathbb{C}$ is a non-degenerate skew-symmetric bilinear form such that for all $x,y,z \in \mathfrak{t}$,
\begin{equation}\label{cycle} \omega([x,y],z) + \omega([z,x],y) + \omega([y,z],x) = 0.
\end{equation} 
Thus every quasi-Frobenius Lie algebra has even dimension.

\begin{definition}\label{CYBE} Given a Lie algebra $\mathfrak{g}$, $r \in \wedge^2\mathfrak{g}$ is a solution to the \emph{classical Yang-Baxter equation (CYBE)} if
$[r_{12},r_{13}] +[r_{12},r_{23}] + [r_{13},r_{23}]=0$. \qed
\end{definition}

The equivalence of (3) and (4) below is due to Drinfeld \cite{Dr} and applies to all complex finite dimensional Lie algebras: for a pair $(\mathfrak{t},\omega)$ as in (4) one takes $r := \omega^{-1} \in \wedge^2\mathfrak{t} \subset \wedge^2 \mathfrak{g}$. Every solution to the CYBE arises in this way.

The equivalence between (1) and (3) was proved in \cite[Theorem 3.2]{EG6}, and later in \cite[Theorem 5.1]{G1} by an alternative proof avoiding the use of Etingof-Kazhdan deformation theory. Full details and references can be found in \cite[$\S$5]{G1} (see also Remarks \ref{equivrmk}(2) below).

\begin{theorem}\label{classify}{\rm ({\rm Etingof-Gelaki}, \cite{EG6})} Let $G$ be a unipotent affine algebraic group over $\mathbb{C}$, with $\mathfrak{g}:=\mathrm{Lie}(G)$. Then the following sets are in bijection with each other.
\begin{enumerate}
\item Gauge equivalence classes of Hopf $2$-cocycles for $G$.
\item Equivalence classes of pairs $(T, \omega)$, where $T\subseteq G$ is a closed subgroup and $\omega$ is a left invariant symplectic form on $T$.
\item Equivalence classes of pairs $(\mathfrak{t},r)$, where $\mathfrak{t}\subseteq \mathfrak{g}$ is a Lie subalgebra and $r \in \wedge^2\mathfrak{t}$ is a nondegenerate solution of the CYBE.
\item Equivalence classes of pairs $(\mathfrak{t},\omega)$, where $\mathfrak{t}\subseteq \mathfrak{g}$ is a quasi-Frobenius subalgebra with symplectic $2$-cocycle $\omega$. \qed
\end{enumerate}
\end{theorem} 

\begin{remarks}\label{equivrmk}
(1) Here is the recipe, given in \cite[Theorem 5.1]{G1}, taking us from (1) to (4) in Theorem \ref{classify}. 

Recall first that given a Lie algebra $\mathfrak{p}$ and a $2$-cocycle $\omega \in \wedge^2 \mathfrak{p}^{\ast}$, one can construct a central extension $\widehat{\mathfrak{p}} := \mathfrak{p} \oplus \mathbb{C}z$ of $\mathfrak{p}$ as follows. Form the exact sequence
\begin{equation}\label{ses} 0 \rightarrow \mathbb{C}z \rightarrow \widehat{\mathfrak{p}} \rightarrow \mathfrak{p} \rightarrow 0,
\end{equation}
with the bracket on $\widehat{\mathfrak{p}}$ given by
$$ [(x,\alpha z), (y,\beta z)]:=([x,y],\omega (x,y) z) $$
for $x,y \in \mathfrak{p}$ and $\alpha,\beta \in \mathbb{C}$. We define
\begin{equation}\label{omtwist} U^{\omega}(\mathfrak{p}) := U(\widehat{\mathfrak{p}})/\langle z - 1 \rangle,
\end{equation}
a \emph{twisted enveloping algebra} of $\mathfrak{p}$ which will feature also in Question \ref{compwg2q2} below.

Now given a Hopf $2$-cocycle $J$ for $G$, let $T$ be its support and $\mathfrak{t}:={\rm Lie}(T)$. So $J$ induces a Hopf $2$-cocycle on $\mathcal{O}(T)$, which we also denote by $J$. By Lemma \ref{action0}(2), the right twisted algebra $\mathcal{O}(T)_J$ admits an action of $T$ by algebra automorphisms. This yields an action of $\mathfrak{t}$ on $\mathcal{O}(T)_J$ by derivations. But $\mathcal{O}(T)_J$ is a Weyl algebra by Theorem \ref{Weyl}(2), so all its derivations are inner \cite[Lemma 4.6.8]{Dix}. Thus, for each $x \in \mathfrak{t}$, we can choose $w_x \in \mathcal{O}(T)_J$ such that $x$ acts as $[w_x,-]$ on $\mathcal{O}(T)_J$. One then checks that 
$$ w_{[x,y]} = w_x w_y - w_y w_x + \omega(x,y);\,\,x,y \in \mathfrak{t},$$
where $\omega(x,y) \in \mathbb{C}$ and $\omega$ is a $2$-cocycle. 

It remains to show that $(\mathfrak{t}, \omega)$ is quasi-Frobenius. To this end, form the extension $\widehat{\mathfrak{t}} = \mathbb{C}z \oplus \mathfrak{t}$ as at (\ref{ses}), and observe that the map $x \mapsto w_x$ yields an isomorphism  of $T$-algebras
\begin{equation}\label{done} U^{\omega}(\mathfrak{t})= U(\widehat{\mathfrak{t}})/\langle z-1 \rangle \cong \mathcal{O}(T)_J.
\end{equation}
Now the fact that $\mathcal{O}(T)_J$ is a Weyl algebra by Theorem \ref{Weyl}(2) implies that $\omega$ is nondegenerate, so that $(\mathfrak{t}, \omega)$ is quasi-Frobenius, as claimed.

(2) The recipe taking us from (3) to (1) in Theorem \ref{classify} is given by the results of Etingof-Kazhdan \cite{EK} (see, e.g., \cite[Example 5.3, Theorem 5.3]{EG1} and \cite[\S 3.1]{EG6}). More precisely, there exists a twist $J_f(r,\hbar)\in U(\mathfrak{t})^{\otimes 2}[[\hbar ]]$ (see, e.g., \cite[Eq. (10)]{EG1}), such that $J:=J(r)$ is given by the composition
\begin{equation}\label{univforJ}
J:\mathcal{O}(G)^{\otimes 2}\twoheadrightarrow \mathcal{O}(T)^{\otimes 2}\xrightarrow{J_f(r,\hbar)}\mathcal{O}(T)^{\otimes 2}\xrightarrow{\epsilon\otimes \epsilon}\mathbb{C}.
\end{equation}

(3) Let $J$ be a Hopf $2$-cocycle for $G$ with support $T$. Then $J$ corresponds to $(\mathfrak{t},r)$ as in Theorem \ref{classify}(3). Given $t\in T$, set $J^t:=(t\ot t)*J*(t^{-1}\ot t^{-1})$ and $r^t:={\rm Ad}(t)(r)\in\wedge^2\mathfrak{t}$. Then it follows from \cite[Eq. (10)]{EG1} that for any $t\in T$, $J^t=J(r^t)$, so $J^t$ corresponds to $(\mathfrak{t},r^t)$. \qed 
\end{remarks}

\section{IHOE structure of $_J\mathcal{O}(G)_J$ for $G$ unipotent}\label{IHOEand}

\subsection{${}_J\mathcal{O}(G)_J$ is an IHOE}\label{IHOE}

The acronym IHOE stands for \emph{iterated Hopf Ore extension}. Given a positive integer $n$, an \emph{$n$-step IHOE over} $\mathbb{C}$ is a Hopf $\mathbb{C}$-algebra $H$ with a chain
\begin{equation}\label{IHOEchain} \mathbb{C} = H_0 \subset H_1 \subset \cdots \subset H_n = H 
\end{equation}
of Hopf subalgebras $H_i$ of $H$, such that for $i = 1, \ldots , n$, 
\begin{equation}\label{extend}H_i= H_{i-1}[X_i; \sigma_i, \partial_i],
\end{equation} 
with $X_i \in H_i$, $\sigma_i$ an algebra automorphism of $H_{i-1}$, and $\partial_i$ a $\sigma_i$-derivation of $H_{i-1}$. For details, see \cite{BOZZ}, where many of the basic properties of IHOEs were determined. It was shown there that the IHOE structure imposes strong constraints on $\Delta(X_i)$; definitive improvements to these constraints were obtained in \cite{H}. 

The IHOEs appearing in this paper are all of the form specified in (1) of the following definition, with the individual inclusions in the chain (\ref{IHOEchain}) as in (2).

\begin{definition}\label{dertype}\begin{enumerate}
\item An $n$-step IHOE $H$ is \emph{of derivation type} if it possesses a chain (\ref{IHOEchain}) for which each automorphism $\sigma_i$ is the identity map on $H_{i-1}$. 
\item A Hopf $\mathbb{C}$-algebra $H$ containing a Hopf subalgebra $A$ and element $X$ such that $H = A[X;\partial]$ for a $\mathbb{C}$-derivation $\partial$ of $A$ is a \emph{Hopf Ore extension of derivation type}. \qed
\end{enumerate} 
\end{definition}

As noted in $\S$\ref{unipotent}, the coordinate ring of a unipotent group of dimension $n$ is an $n$-step IHOE with all $\sigma_i$ and $\partial_i$ trivial. More precisely, throughout this section we keep the following hypotheses and notation in place:

\medskip

${\bf (H)}\,\,$  $G$ \textit{is a unipotent algebraic group over} $\mathbb{C}$, \textit{with} $ \mathcal{O}(G) = \mathbb{C}[X_1, \ldots , X_n]$ \textit{as in} $\S$\ref{unipotent}, \textit{and with further notation as specified there.}

\begin{theorem}\label{Ore} Let $G$ be as above, and let $J$ be a Hopf $2$-cocycle for $G$. In addition define $Q := J - J_{21}$ and $\overline{Q} := J^{-1} - J^{-1}_{21}$, and set $\mathfrak{m}:=\mathcal{O}(G)^+$.
\begin{enumerate}
\item The chain (\ref{poly}), viewed as a chain of vector spaces, gives a chain of Hopf subalgebras of ${}_J\mathcal{O}(G)_J$, which we write as
\begin{equation}\label{poly2} \mathbb{C} = H_0 \subset \cdots \subset \, _J(H_i)_J \subset \cdots \subset {}_J(H_n)_J ={}_J\mathcal{O}(G)_J.
\end{equation}
\item 
For all $1 \leq i,j \leq n$, we have
$$ J(X_i,X_j) + J^{-1}(X_i,X_j) + \sum J(x^i_1,x^j_1)J^{-1}(x^i_2, x^j_2) =0.$$
\item 
For all $1 \leq i,j \leq n$, we have
\begin{align*}X_i \cdot X_j &= X_iX_j + \sum J^{-1}(X_i, x^j_1)x^j_2 + \sum x^j_1J(X_i, x^j_2)\\
& +\sum x^i_1 J(x^i_2, X_j) + \sum J^{-1}(x^i_1,X_j)x^i_2 + \sum x^i_1 x^j_1J(x^i_2,x^j_2)\\
&+ \sum_{x^i_{21}x^j_{21} \in \mathfrak{m}} J^{-1}(x^i_1,x^j_1)x^i_{21}x^j_{21}J(x^i_{22},x^j_{22}).
\end{align*}
\item 
For all $1 \leq i,j \leq n$, we have 
\begin{align*}  [X_i,X_j] &= X_i\cdot X_j - X_j \cdot X_i\\
&= \quad \sum x^i_1 Q(x^i_2, X_j) + \sum x^j_1 Q(X_i, x^j_2) + \sum x^i_1 x^j_1 Q(x^i_2,x^j_2)\\
&  + \sum \left( J^{-1}(x^i_1,x^j_1)J(x^i_{22},x^j_{22}) - J_{21}^{-1}(x^i_1,x^j_1)J_{21}(x^i_{22},x^j_{22})\right)x^i_{21}x^j_{21}\\
& + \sum x^i_2\overline{Q}(x^i_1,X_j) + \sum x^j_2 \overline{Q}(X_i, x^j_1).
\end{align*}
\noindent In particular, $[X_i,X_j] \in H_{i-1}^+$ for all $1 \leq j<i \leq n$.
\item 
For all $1 \leq i \leq n$, ${}_J(H_i)_J \cong  
{}_J(H_{i-1})_J[X_i; \partial_i]$ is a Hopf Ore extension of derivation type.
\item 
$X_1$ and $X_2$ are primitive elements of ${}_J\mathcal{O}(G)_J$. 
\item 
The space $\mathcal{P}({}_J\mathcal{O}(G)_J)$ of primitive elements of $_J\mathcal{O}(G)_J$, which equals $\mathcal{P}(\mathcal{O}(G))$, generates a commutative Hopf subalgebra of $_J\mathcal{O}(G)_J$, a polynomial algebra of rank $\mathrm{dim}_k(\mathcal{P}(\mathcal{O}(G)))= \mathrm{dim}(G/[G,G])$.
\end{enumerate}
\end{theorem}

\begin{proof} 
(1) This is clear from  (\ref{mult}) and (\ref{poly}). 

(2) Omitting summation symbols for clarity where convenient, we calculate, using (\ref{good}) for $\Delta(X_i)$ and $\Delta (X_j)$, and the notation (\ref{qform}), and applying (\ref{Jtwo}),
\begin{align*} 
0 &= (\epsilon \otimes \epsilon)(X_i \otimes X_j)= (J\ast J^{-1})(X_i \otimes X_j)\\
&= J(X_i,X_j)J^{-1}(1,1) + J(X_i,1)J^{-1}(1,X_j) +J(1,X_j)J^{-1}(X_i,1)\\
& +J(1,1)J^{-1}(X_i,X_j) + J(x^i_1,X_j)J^{-1}(x^i_2,1)+J(X_i, x^j_1)J^{-1}(1,x^j_2)\\
& + J(1,x^j_1)J^{-1}(X_i,x^j_2)+J(x^i_1,1)J^{-1}(x^i_2, X_j) + J(x^i_1,x^j_1)J^{-1}(x^i_2, x^j_2)\\
&= J(X_i,X_j) + J^{-1}(X_i,X_j) + J(x^i_1,x^j_1)J^{-1}(x^i_2, x^j_2), 
\end{align*}
where the final equality holds since the omitted terms are all $0$ by (\ref{Jtwo}) and (\ref{qform}).

(3) We start from the expansions
\begin{align*} 
(\mathrm{id}\otimes \Delta)\circ\Delta (X_i) &= X_i \otimes 1 \otimes 1 + 1 \otimes X_i \otimes 1 + 1 \otimes 1 \otimes X_i\\
&  + 1 \otimes q(X_i) + \sum x^i_1 \otimes x^i_{21} \otimes x^i_{22},
\end{align*}
and a corresponding one for $X_j$. 
Thus, from (\ref{mult}) applied to the above expansions, and using (\ref{Jtwo}) to see that many terms in the product are $0$, we calculate
\begin{align*} 
X_i \cdot X_j &= J^{-1}(X_i,X_j)J(1,1) +J^{-1}(X_i,x^j_1)x^j_{21}J(1,x^j_{22})+ X_iX_j + J^{-1}(1,1)J(X_i,X_j)\\
&+J^{-1}(1,1)x^j_1 J(X_i,x^j_2)+ J^{-1}(1,1)x^i_1J(x^i_2,X_j)+ J^{-1}(1,1)x^i_1x^j_1J(x^i_2,x^j_2)\\ 
&+ J^{-1}(x^i_1,X_j)x^i_{21}J(x^i_{22},1) + J^{-1}(x^i_1,x^j_1)x^i_{21}x^j_{21}J(x^i_{22},x^j_{22})\\
&= J^{-1}(X_i, X_j) + J(X_i,X_j) + J^{-1}(X_i, x^j_1)x^j_{21}J(1,x^j_{22})+ X_iX_j + x^j_1J(X_i,x^j_2)\\
& + x^i_1 J(x^i_2,X_j)+ x^i_1x^j_1J(x^i_2,x^j_2) + J^{-1}(x^i_1,X_j)x^i_{21}J(x^i_{22},1)\\
&+ J^{-1}(x^i_1,x^j_1)x^i_{21}x^j_{21}J(x^i_{22},x^j_{22})\\
&= J^{-1}(X_i, X_j) + J(X_i,X_j) + J^{-1}(x^i_1,x^j_1)J(x^i_{22},x^j_{22})+ J^{-1}(X_i, x^j_1)x^j_{21}J(1,x^j_{22})\\
& + X_iX_j + x^j_1J(X_i,x^j_2) + x^i_1 J(x^i_2,X_j)+ x^i_1x^j_1J(x^i_2,x^j_2) + J^{-1}(x^i_1,X_j)x^i_{21}J(x^i_{22},1)\\
&+ \sum_{x^i_{21}x^j_{21} \in \mathfrak{m}} J^{-1}(x^i_1,x^j_1)x^i_{21}x^j_{21}J(x^i_{22},x^j_{22})\\
&= J^{-1}(X_i, x^j_1)x^j_{21}J(1,x^j_{22})+ X_iX_j + x^j_1J(X_i,x^j_2) + x^i_1 J(x^i_2,X_j)+ x^i_1x^j_1J(x^i_2,x^j_2) \\
&+ J^{-1}(x^i_1,X_j)x^i_{21}J(x^i_{22},1)+ \sum_{x^i_{21}x^j_{21} \in \mathfrak{m}} J^{-1}(x^i_1,x^j_1)x^i_{21}x^j_{21}J(x^i_{22},x^j_{22}),
\end{align*} 
where the sum of the first three terms in the penultimate expansion equals $0$ by (2). In the final expansion above, the summands in the first term are all 0 by (\ref{Jtwo}), except when $x^j_{22} = 1$, which occurs when $x^j_{21} = x^j_2$, so that the first term collapses to $\sum J^{-1}(X_i, x^j_1)x^j_2$. Similarly, the penultimate term collapses to $\sum J^{-1}(x^i_1, X_j)x^i_2$. 
Hence, we finally obtain 
\begin{align*}
X_i \cdot X_j &= X_iX_j + \sum J^{-1}(X_i, x^j_1)x^j_2 + \sum x^j_1J(X_i, x^j_2)\\
& +\sum x^i_1 J(x^i_2, X_j) + \sum J^{-1}(x^i_1,X_j)x^i_2 + \sum x^i_1 x^j_1J(x^i_2,x^j_2)\\
&+ \sum_{x^i_{21}x^j_{21} \in \mathfrak{m}} J^{-1}(x^i_1,x^j_1)x^i_{21}x^j_{21}J(x^i_{22},x^j_{22}).
\end{align*}

(4) Combining the calculation of $X_i \cdot X_j$ from (3) with the corresponding one for $X_j \cdot X_i$, we arrive at the formula stated in the proposition.

(5) This is immediate from (1) and (4) when we bear in mind that, as a vector space, ${}_J(H_i)_J$ has the same basis as $H_i$.

(6) That $X_1$ and $X_2$ are primitive follows from the fact that the comultiplication in ${}_J\mathcal{O}(G)_J$ is unchanged from $\mathcal{O}(G)$, since in characteristic $0$ there are no non-abelian unipotent groups of dimension at most $2$.

(7) It is clear from (4) that primitive elements commute. The rest of (7) now follows easily. 
\end{proof}

The following key corollary is easily deduced from Theorem \ref{Ore}(5):

\begin{corollary}\label{nIHOE} With $G$ and $J$ as in Theorem \ref{Ore}, ${}_J\mathcal{O}(G)_J$ is an $n$-step IHOE of derivation type.
\end{corollary}
\begin{proof} This is essentially proved already by Theorem \ref{Ore}(5). It is perhaps worth observing that the relations given by Theorem \ref{Ore}(4) ensure that ${}_J\mathcal{O}(G)_J$ is a \emph{factor} of the iterated Ore extension $\mathbb{C}[X_1][X_2; \partial_2] \cdots [X_n ; \partial_n]$. However, we know that the coalgebra structure is preserved on the same vector space by the twist operation, so the ideal we are factoring by is $\{0\}$.
\end{proof}

\begin{remarks}\label{compwg2}
(1) Here is a perspective on Corollary \ref{nIHOE}, which is useful for instance when discussing the finite dimensional simple ${}_J\mathcal{O}(G)_J$-modules in $\S$\ref{simples2}. We can view both $\mathcal{O}(G)$ and ${}_J\mathcal{O}(G)_J$ as factor algebras of the free $\mathbb{C}$-algebra on $n$ generators, $F := \mathbb{C}\langle X_1,\dots, X_n \rangle$, the former being the factor of $F$ by the ideal 
$$ D:=\langle [X_i,X_j]\mid 1 \leq i < j \leq n \rangle, $$
and the latter the factor by the ideal
$$ E:= \langle [X_i,X_j] - f_{ij}\mid 1 \leq i < j \leq n \rangle, $$
where the elements $f_{ij}\in F$ are given by Theorem \ref{Ore}(4). Let 
$$ \mathcal{B}:= \{X_1^{a_1} X_2^{a_2}\cdots X_n^{a_n}\mid a_i \in \mathbb{Z}^{\geq 0} \}  \subset F.$$ Then the cosets $$ \{ b +D \mid b \in \mathcal{B} \} \textit{   and   }\{ b +E \mid b \in \mathcal{B} \}$$
provide $\mathbb{C}$-bases of $\mathcal{O}(G)$ and ${}_J\mathcal{O}(G)_J$ respectively, \emph{with the multiplication in each case induced by the multiplication in $F$}. In particular, one sees that the sets of powers of each generator $X_i$ are the same in both $\mathcal{O}(G)$ and in ${}_J\mathcal{O}(G)_J$.

(2) Theorem \ref{Ore}(3)-(5) correct typos in \cite[Lemma 3.1]{G2}. More precisely, Theorem \ref{Ore}(3) corrects \cite[Lemma 3.1(4)]{G2} in which the second and fifth summands are missing. Theorem \ref{Ore}(4) corrects \cite[Lemma 3.1(5)]{G2} in which the fifth and sixth summands are missing. Theorem \ref{Ore}(5) corrects \cite[Lemma 3.1(6)]{G2} in which it is mistakenly stated that $X_1$ and $X_2$ are central. \qed
\end{remarks}

\subsection{Immediate consequences of IHOE structure}\label{immediate}
Here are some more or less immediate consequences of Corollary \ref{nIHOE}. Definitions of the homological terms appearing in (2) can be found in \cite{BZ}, for example. Recall for (3) that a Hopf algebra is \emph{connected} \cite[Definition 5.1.5]{M} if its coradical is just the base field $\mathbb{C}$; note for future use that this property passes to all factor Hopf algebras by \cite[Corollary 5.3.5]{M}.

\begin{corollary}\label{IHOEcor} Keep the notation and hypotheses of ${\bf(H)}$ from the start of $\S$\ref{IHOEand}.
\begin{enumerate}
\item {\rm (\cite[Corollary 3.2]{G2})} ${}_J\mathcal{O}(G)_J$ is a noetherian domain of GK dimension $n$.
\item ${}_J\mathcal{O}(G)_J$ is Artin-Schelter regular and Auslander regular with global homological dimension $n$, and is GK-Cohen-Macaulay.
\item ${}_J\mathcal{O}(G)_J$ is a connected Hopf algebra whose associated graded algebra with respect to its coradical filtration is the coordinate ring $\mathcal{O}(\widehat{G})$, the associated graded ring of $\mathcal{O}(G)$, $\widehat{G}$ being the associated graded unipotent group of $G$.
\item Every prime ideal of ${}_J\mathcal{O}(G)_J$ is completely prime.
\item Every finite dimensional simple ${}_J\mathcal{O}(G)_J$-module $V$ has $\mathrm{dim}_{\mathbb{C}}(V) = 1$.
\end{enumerate}
\end{corollary}

\begin{proof}(1), (2) These properties are valid for every $n$-step IHOE - see \cite[Theorem 3.2(2),(5)]{BOZZ}.

(3) Since the coalgebra structure of $_J\mathcal{O}(G)_J$ is the same as that of $\mathcal{O}(G)$, ${}_J\mathcal{O}(G)_J$ is a connected Hopf algebra (which also follows from the fact that it is an IHOE, \cite[Theorem 3.2(1)]{BOZZ}). Hence its coradical filtration is an algebra filtration by \cite[Lemma 5.2.8]{M}. It follows from Corollary \ref{nIHOE} and \cite[Theorem 3.2(4)]{BOZZ} that the resulting associated graded algebra of ${}_J\mathcal{O}(G)_J$ is a coradically graded commutative polynomial Hopf algebra in $n$ variables. But as already recalled, ${}_J\mathcal{O}(G)_J$ and $\mathcal{O}(G)$ are the same as coalgebras, hence they have the same associated graded Hopf algebra.

(4) This is true for all iterated Ore extensions of $\mathbb{C}$ of derivation type by \cite{L}.

(5) The annihilator $I$ of a finite dimensional simple ${}_J\mathcal{O}(G)_J$-module $V$ is a maximal ideal with $\mathrm{dim}_{\mathbb{C}}({}_J\mathcal{O}(G)_J/I) < \infty$. So, by the Artin-Wedderburn theorem and (4), ${}_J\mathcal{O}(G)_J/I \cong \mathbb{C}$. Hence, $\mathrm{dim}_{\mathbb{C}}(V) = 1$.
\end{proof}

Further to Corollary \ref{IHOEcor}(2), basic properties of IHOEs allow us to say also that $_J\mathcal{O}(G)_J$ is skew Calabi-Yau, but we do not include that property here as we will later - in Corollary \ref{payoff} - show that $_J\mathcal{O}(G)_J$ is actually Calabi-Yau. The group of finite dimensional simple ${}_J\mathcal{O}(G)_J$-modules (which exists by (5)), is studied in $\S$\ref{simples2}.

\section{Structure of twisted unipotent groups - the minimal case}\label{structure}

Retain the notation and hypotheses {\bf (H)} from the beginning of $\S$\ref{IHOEand}, so $(\mathcal{O}(G),\epsilon \otimes \epsilon)$ is cotriangular by Example \ref{commex}(1). Fix a Hopf $2$-cocycle $J$ for $G$, so $R^J:= J_{21}^{-1}\ast J$ and $({}_J\mathcal{O}(G)_J, R^J)$ is cotriangular by Example \ref{commex}(2). The aim in this section is to prove Theorem \ref{minLie} (as predicted in \cite[Theorem 3.4]{G2}; see Remark \ref{compwg22} below), describing the algebra structure of ${}_J\mathcal{O}(G)_J$ in the case where $({}_J\mathcal{O}(G)_J, R^J)$ is minimal (Definition \ref{min}). 

\subsection{The minimal case - inclusion of $\mathrm{im}(\Psi)$}\label{minimal1} 
We assume throughout this subsection that $R^J$ is nondegenerate for ${}_J\mathcal{O}(G)_J$, as in Definition \ref{min}. By Lemma \ref{into1}, there is an embedding of Hopf algebras
\begin{equation}\label{caught}     
\Psi: {}_J\mathcal{O}(G)_J \xrightarrow{1:1} ({}_J\mathcal{O}(G)_J)^{\circ, {\rm op}},\,\,\, a \mapsto R^J(-,a). 
\end{equation}
Observe that $({}_J\mathcal{O}(G)_J)^{\circ, {\rm op}}$ is a subalgebra of $({}_J\mathcal{O}(G)_J)^{*,{\rm op}}$ (where $(_J\mathcal{O}(G)_J)^*$ is the full vector space dual of ${}_J\mathcal{O}(G)_J$). Since $\mathcal{O}(G)={}_J\mathcal{O}(G)_J$ as coalgebras, their vector space duals are the same as algebras. Namely, via the identity map, as algebras,
\begin{equation}\label{same} ({}_J\mathcal{O}(G)_J)^{*,{\rm op}}= (\mathcal{O}(G)^*)^{{\rm op}}.
\end{equation}
In particular, $({}_J\mathcal{O}(G)_J)^{*,{\rm op}}$ contains as a subalgebra the (opposite algebra of the) finite dual $\mathcal{O}(G)^{\circ}$ of $\mathcal{O}(G)$. Since $\mathcal{O}(G)$ is commutative, its finite dual is cocommutative, so by the Cartier-Gabriel-Kostant structure theorem \cite[Corollary 5.6.4(3) and Theorem 5.6.5]{M} and (\ref{same}), $ ({}_J\mathcal{O}(G)_J)^{*,{\rm op}}$ contains the subalgebra $(U(\mathfrak{g})\#\mathbb{C}[G])^{{\rm op}}$, where $\mathfrak{g}:={\rm Lie}(G)$. Thus, $({}_J\mathcal{O}(G)_J)^{*,{\rm op}}$ contains the subalgebra $U(\mathfrak{g})^{{\rm op}} \cong U(\mathfrak{g})$, consisting of the distributions supported (set-theoretically) at $1$, that is vanishing on some power of the augmentation ideal $\mathfrak{m}$ of $\mathcal{O}(G)$. We shall show that 
\begin{equation}\label{image} \mathrm{im}(\Psi) = U(\mathfrak{g})^{{\rm op}}.
\end{equation}

First we use the work of Etingof and Kazhdan \cite{EK} to prove

\begin{proposition}\label{include}
With the above hypotheses and notation,
$$\mathrm{im}(\Psi) \subseteq U(\mathfrak{g})^{{\rm op}}.$$
\end{proposition}

\begin{proof} 
Let $J_f(r,\hbar)\in U(\mathfrak{g})^{\otimes 2}[[\hbar ]]$ be such that $J$ is given by the composition
\begin{equation}\label{univforJ}
J:\mathcal{O}(G)^{\otimes 2}\xrightarrow{J_f(r,\hbar)}\mathcal{O}(G)^{\otimes 2}\xrightarrow{\epsilon\otimes \epsilon}\mathbb{C}
\end{equation}
(see Remarks \ref{equivrmk}(2)). Then setting 
$$R_f(r,\hbar):=J_f(r,\hbar)_{21}^{-1}J_f(r,\hbar)\in U(\mathfrak{g})^{\otimes 2}[[\hbar ]],$$ we have that $R^J$ is given by the composition
$$R^J:\mathcal{O}(G)^{\otimes 2}\xrightarrow{R_f(r,\hbar)}\mathcal{O}(G)^{\otimes 2}\xrightarrow{\epsilon\otimes \epsilon}\mathbb{C}.$$ 
Namely, writing
$$
R_f(r,\hbar)=1+\sum_{i=1}^{\infty} (r_{i1}\otimes r_{i2})\hbar ^i\in U(\mathfrak{g})^{\otimes 2}[[\hbar ]],$$ the bilinear form $R^J$ is given by
$$
R^J(y,z)=\epsilon(yz)+\sum_{i=1}^{\infty} r_{i1}(y)r_{i2}(z);\,\,\,y,z\in \mathcal{O}(G),
$$
where the sum is well defined since it is actually finite (as we see below).

Let $z \in \mathfrak{m}:=\mathcal{O}(G)^+$. Pick $n$ such that $r_{i2}(z)=0$ for every $i>n$, which is possible since $z$ vanishes on some power of $U(\mathfrak{g})^+$, and $r_{i2}$ lies in $(U(\mathfrak{g})^+)^{2i}$ for every $i$ \cite[Example 5.3]{EG1}. Next fix $N$ such that $r_{11},\dots,r_{1n}$ vanish on 
$\mathfrak{m}^N$. Then for every $y\in\mathfrak{m}^N$, we have
$$
\Psi (z)(y)=R^J(y,z) =\sum_{i=1}^{\infty} r_{i1}(y)r_{i2}(z)=\sum_{i=1}^{N}r_{i1}(y)r_{i2}(z) = 0.
$$
Thus, $\Psi(z)$ vanishes on $\mathfrak{m}^N$, so $\Psi$ maps ${}_J\mathcal{O}(G)_{J}$ into $U(\mathfrak{g})$, as claimed.
\end{proof}

\subsection{The minimal case - surjectivity of $\Psi$}\label{minimal2}

We continue to assume throughout this subsection that $R^J$ is nondegenerate for ${}_J\mathcal{O}(G)_J$, and we retain all the notation introduced so far in $\S$\ref{structure}. The proof that $\Psi$ maps ${}_J\mathcal{O}(G)_J$ {\em onto} $U(\mathfrak{g})$ is by induction on $\mathrm{dim}(G)$, so we start by introducing the induction setup. 

Recall from $\S$\ref{2cocycunip} that $n:=\dim(G)$ is even. Given elements $x,y$ in (the vector space) $\mathcal{O}(G)$, when confusion seems possible we will use the notation $x\cdot y$ to denote multiplication in ${}_J\mathcal{O}(G)_J$ and $xy$ to denote multiplication in $\mathcal{O}(G)$.

Choose a $1$-dimensional closed central subgroup $C$ of $G$, and let $K/C$ be the support of $J$ on $\mathcal{O}(G/C)\subset \mathcal{O}(G)$. By \cite[Proposition 4.6]{G1}, $K/C $ has codimension $1$ in $G/C$, so $K$ has codimension $1$ in $G$, and hence is normal. Denote the Lie algebras of $K$ [resp. $G$]  by $\mathfrak{k}$ [resp. $\mathfrak{g}$]. Clearly, the exact sequence $0 \rightarrow \mathfrak{k} \rightarrow \mathfrak{g} \rightarrow \mathfrak{g}/\mathfrak{k}\rightarrow 0$ splits, with $\mathfrak{g} = \mathfrak{k} \oplus \mathbb{C}x$ for any element $x \in \mathfrak{g} \setminus \mathfrak{k}$. Hence, the same is true for the corresponding groups, namely, we have
\begin{equation}\label{sdirpr2}
G = K\rtimes Y \cong K\rtimes (G/K),
\end{equation}
where the subgroup $Y$ of $G$ is isomorphic to $(\mathbb{C},+)$. Note that $C  \subseteq  K$ and $C$ is central in $G$, so that 
$$ C\rtimes Y= C \times Y \subseteq G.$$

By (\ref{sdirpr2}), we can pick a primitive element $y_1\in\mathcal{O}(G)$ such that $y_1(K) = 0$ and $\mathcal{O}(Y) \cong \mathcal{O}(G/K)=\mathbb{C}[y_1]$ as Hopf algebras. Consider the chain of group epimorphisms
$$G\twoheadrightarrow G/C\twoheadrightarrow G/K.$$
By the discussion in $\S$\ref{unipotent}, this yields the chain of cotriangular Hopf algebras 
\begin{equation}\label{indcothaschain2}
\mathbb{C}[y_1] = {}_J\mathcal{O}(G/K)_J\subseteq {}_J\mathcal{O}(G/C)_J\subset {}_J\mathcal{O}(G)_J,
\end{equation}
where we will elaborate on the first equality in Lemma \ref{supportai2} below. Pick $y_{n}\in\mathcal{O}(G)$ whose restriction to $C$ generates $\mathcal{O}(C )$. Then
\begin{equation*}\label{oreextyn2}
{}_J\mathcal{O}(G)_J={}_J\mathcal{O}(G/C)_J[y_{n};\partial_{n}]
\end{equation*}
is a Hopf Ore extension, and by (\ref{qform}) we see that
\begin{equation}\label{deltayn2}
\Delta(y_n)-y_n\otimes 1-1\otimes y_n := \sum Y_{n1} \otimes Y_{n2} \in \mathcal{O}(G/C )^+\otimes \mathcal{O}(G/C )^+.
\end{equation}

It follows from (\ref{indcothaschain2}), and (again) will be further explicated in Lemma \ref{supportai2}, that $y_1$ generates the kernel of the surjective cotriangular Hopf algebra map 
\begin{equation}\label{newsurjhom2}
{}_J\mathcal{O}(G/C)_J\twoheadrightarrow {}_J\mathcal{O}(K/C)_J.
\end{equation}

We need the following easy lemma.
 
\begin{lemma}\label{kernel2} Let $(G,J)$ be a unipotent group with a Hopf $2$-cocycle, and let $p$ be a primitive element of $\mathcal{O}(G)$. Let $R^J $ be the $R$-form $J_{21}^{-1}\ast J$ as in Example \ref{commex}(2). Define
\begin{equation}\label{kill2} I_p:=\{ \alpha \in (_J\mathcal{O}(G)_J)^+\mid  R^J(\alpha, p) = 0\}. 
\end{equation}
\begin{enumerate}
\item 
$I_p$ is an ideal of ${}_J\mathcal{O}(G)_J$.
\item $I_p=\{ \alpha \in (_J\mathcal{O}(G)_J)^+ \mid R^J(p, \alpha) = 0\}.$
\end{enumerate}
\end{lemma}
 
\begin{proof}$(1)$ Clearly $I_p$ is a vector subspace. Let $\alpha \in I_p$ and $\beta \in {}_J\mathcal{O}(G)_J$. Then
$$ R^J(\alpha \cdot \beta, p) = R^J(\alpha, p)R^J(\beta,1) + R^J(\alpha,1)R^J(\beta,p) = 0 + 0 = 0,$$
where the first term is $0$ since $\alpha \in I_p$ and the second is $0$ by Lemma \ref{quick}(1). A similar calculation deals with $\beta \cdot \alpha$.

$(2)$ Since $({}_J\mathcal{O}(G)_J, R^J)$ is cotriangular, and using \cite[Lemma 2.2.2]{Ma1}, 
$$ R^J(x,p)=(R^J)^{-1}(x,S(p))= R^J(S(p),x)$$
for all $x \in {}_J\mathcal{O}(G)_J$. So (2) follows from the fact that $S(p) = -p$.
\end{proof}

Returning to our induction setup, we prove the following basic facts.

\begin{lemma}\label{supportai2}
The following hold:
\begin{enumerate}
\item
For every $x\in\mathcal{O}(G/C)$, we have 
$R^J(x,y_{1})= R^J (y_1,x) = 0$.
\item
$y_1$ is a central primitive element of ${}_J\mathcal{O}(G)_{J}$.
\item
For all $x\in\mathcal{O}(G)^+$, $R^J(y_{1}\cdot x,y_{n})=R^J(y_n, y_1 \cdot x) = 0$.
\item
Using the notation (\ref{kill2}),
\begin{align*} I_{y_{1}}  &= ({}_J\mathcal{O}(G/C)_J)^+{}_J\mathcal{O}(G)_J + y_n^2 {}_J\mathcal{O}(G)_J\\
&= ({}_J\mathcal{O}(G/C)_J)^+{}_J\mathcal{O}(G)_J + y_n\cdot y_n {}_J\mathcal{O}(G)_J.
\end{align*}
\item
We may assume that $R^J(y_{n}, y_{1})=1$.
\end{enumerate}
\end{lemma}

\begin{proof}
(1) By definition of $K$ and of $y_1$, $y_1$ lies in the radical of $R^J$ restricted to ${}_J\mathcal{O}(G/C )_J$.

(2) By (1) and Lemma \ref{quick}(2), for all $x \in \mathcal{O} (G/C)$, we have 
\begin{equation}\label{RandJ2} 0 = R^J(y_1, x) = (J - J_{21})(y_1, x) = (J_{21}^{-1} - J^{-1})(y_1,x). 
\end{equation}
The chain $1 \subset C \subseteq K \subset G$ of normal subgroups of $G$ can be refined to a chain of normal subgroups with successive subfactors all isomorphic to $(\mathbb{C},+)$. This in turn yields a chain of Hopf Ore extensions
$$ \{0\} \subset \mathcal{O}(G/K) = \mathbb{C}[y_1] \subseteq \cdots \subseteq \mathcal{O}(G/C) \subset  \mathcal{O}(G/C)[y_n] = \mathcal{O}(G),$$
with a similar chain for ${}_J\mathcal{O}(G)_J$ formed from identical subspaces, as explained in $\S$\ref{unipotent} and Theorem \ref{Ore}. Fix generators $X_2, \ldots , X_{n-1}$ for the Hopf Ore extensions from ${}_J\mathcal{O}(G/K)_J$ to ${}_J\mathcal{O}(G/C)_J$, and define $Q := J - J_{21}$ and $\overline{Q} := J^{-1} - J^{-1}_{21}$. Recall that we write, for each $i$,
\begin{equation}\label{comformula}
\Delta (X_i) - X_i \otimes 1 - 1 \otimes X_i = \sum x^i_1 \otimes x^i_2,
\end{equation}
where, as in (\ref{deltayn2}), the elements $x^i_1, x^i_2 $ lie in lower Hopf algebras, and in particular are in ${}_J\mathcal{O}(G/C)_J$. Hence, by Theorem \ref{Ore}(4), taking commutators in ${}_J\mathcal{O}(G/C)_J$, we have for $i = 2, \ldots , n-1$,
\begin{equation}\label{comm2} 
[X_i, y_1] = \sum x^i_1 Q(x^i_2,y_1) + \sum x^i_2 \overline{Q}(x^i_1,y_1) = 0,
\end{equation}
where the second equality follows from (\ref{RandJ2}).

An exactly similar argument shows that $[y_n,y_1] = 0$, in view of (\ref{deltayn2}). Since ${}_J\mathcal{O}(G)_J$ is generated by $\{y_1, X_2, \ldots , X_{n-1}, y_n\}$, (2) is proved.

(3) By (\ref{deltayn2}) and by parts (1) and (2) of the lemma, we have
\begin{align*} R^J(y_{1}\cdot x,y_{n}) &= R^J(y_1,y_n)R^J(x,1) + R^J(y_1,1)R^J(x,y_n)\\ & + \sum R^J(y_{1},Y_{n_1})R^J(x, Y_{n_2})=\sum R^J(y_{1},Y_{n_1})R^J(x, Y_{n_2})=0,
\end{align*}
where the second equality follows from Lemma \ref{quick}(1) and the third from (\ref{deltayn2}) and (1). The argument for $R^J(y_n, y_{1}\cdot x)$ is exactly similar.

(4), (5) Note first that $y_{n}^2=y_{n}\cdot y_{n}+z$ for some $z$ in $\mathcal{O}(G/C )^+$ by Theorem \ref{Ore}(3). Thus, the ideals appearing in the second and third entries of the identity in (4) are equal. Moreover, $ ({}_J\mathcal{O}(G/C)_J)^+{}_J\mathcal{O}(G)_J \subset I_{y_1}$ by Lemmas \ref{kernel2}(1) and \ref{supportai2}(1). So to prove that the ideal on the right is contained in $I_{y_1}$ it remains only to show that $y_n\cdot y_n \in I_{y_1}$. But this is immediate from Lemma \ref{quick}(1), since $y_1$ is primitive.

For the reverse inclusion, note that $R^J(-,y_1)$ is a nonzero functional on ${}_J\mathcal{O}(G)_J$ by minimality of $J$. However, by the inclusion proved above and Lemma \ref{quick}(1),
$$ R^J\left(\left({}_J\mathcal{O}(G/C)_J\right)^+{}_J\mathcal{O}(G)_J + y_n\cdot y_n {}_J\mathcal{O}(G)_J + \mathbb{C}, y_1\right) = 0.$$
Hence we must have $R^J(y_n,y_1) \neq 0$, which proves the reverse inclusion and also shows that we can rescale so that $R^J(y_n,y_1) = 1$, proving (5).
\end{proof}

We proceed now to prove Theorem \ref{noethdomminunip2}, that $\mathrm{Im}(\Psi)$ is $U(\mathfrak{g})^{{\rm op}}$, thus completing the  proof of  (\ref{image}). Recall that $\mathfrak{g}$ and $\mathfrak{k}$ are  the Lie algebras of $G$ and $K$, respectively, and let $\mathfrak{c}$ denote the Lie algebra of $C$. By (\ref{sdirpr2}), we have a chain of ideals $\mathfrak{c}\subseteq \mathfrak{k}\subset \mathfrak{g}$, 
such that $\mathfrak{c}\cong \mathbb{C}$ is central in $\mathfrak{g}$, and 
\begin{equation*}\label{sdirprlie2}
\mathfrak{g}=\mathfrak{k}\rtimes (\mathfrak{g}/\mathfrak{k})\cong  \; \mathfrak{k}\rtimes \mathbb{C}.
\end{equation*}
It follows (see, e.g., \cite[\S 1.7.10, \S 1.7.12]{McCR}) that we have algebra isomorphisms 
\begin{equation}\label{sdirprlieuniv2}
U(\mathfrak{g})\cong U(\mathfrak{k})\#U(\mathfrak{g}/\mathfrak{k}) \cong \left(U(\mathfrak{c})*U(\mathfrak{k}  /\mathfrak{c})\right)\#U(\mathfrak{g}/\mathfrak{k}),
\end{equation}
where $*$ denotes a crossed product and $\#$ a smash product. Thus, $U(\mathfrak{g})$ is generated by its subspaces $U(\mathfrak{c})$, $U(\mathfrak{k}/\mathfrak{c})$, and $U(\mathfrak{g}/\mathfrak{k})$.  
Bearing in mind that $\mathfrak{g}$ constitutes precisely the $\mathbb{C}$-derivations of $\mathcal{O}(G)$, we can specify a basis $u_1, \ldots  , u_n$ of $\mathfrak{g}$ as follows: 
\begin{align*} u_1(y_1) &= u_1(X_j) = 0, \; 2 \leq j \leq n-1, \; u_1 (y_n) = 1;\\ 
u_i(y_1) &= u_i(y_n) = 0, \; u_i (X_j) = \delta_{i,n+1-j}, \; 2 \leq i,j \leq n-1;\\
u_n(X_j) &= u_n(y_n) = 0, \; 2 \leq j \leq n-1, \; u_n(y_1) = 1.
\end{align*}
Set $\mathfrak{m}:=\mathcal{O}(G)^+$. Since the above maps are derivations,  $u_i(\mathfrak{m}^2) = u_i(1) = 0$ for all $i = 1, \ldots , n$, noting that $\mathfrak{m}^2$ is the product $\mathfrak{m}\mathfrak{m}$ using the \emph{commutative} multiplication. Recall that $U(\mathfrak{g})$ is the subalgebra of $\mathcal{O}(G)^*$ consisting of distributions supported (set-theoretically) at $1$, that is, vanishing on some power of $\mathfrak{m}$.

\begin{theorem}\label{noethdomminunip2}
Keep all the notation introduced above, and in the rest of $\S$\ref{structure}. The image of the Hopf algebra monomorphism $\Psi$ from $_J\mathcal{O}(G)_J$ defined at (\ref{caught}) is $U(\mathfrak{g})^{{\rm op}}$. 
\end{theorem}

\begin{proof}
Since $\Psi$ is an algebra map, it is enough to show that 
\begin{equation}\label{heart} u_i \in \mathrm{im}(\Psi),\; i = 1, \ldots , n.
\end{equation}

The proof of (\ref{heart}) is quite involved, requiring four lemmas, the first being:

\begin{lemma}\label{Claim I} 
After multiplying if necessary by a non-zero scalar, we have 
\begin{equation}\label{start} u_1=\Psi(y_1)\in \mathrm{im}(\Psi).
\end{equation}
\end{lemma} 

\begin{proof}
This follows from Lemma \ref{supportai2}(1),(4),(5) and the definition of $u_1$.
\end{proof}
 
We now proceed by induction on even $n\ge 0$, the case $n=0$ being trivial. So we assume that the result is true for all minimal $2$-cocycle deformations of unipotent groups whose dimension is less than $n$. In particular, by definition of $K$, the restriction of $J$ to $\mathcal{O}(G/C)$ induces a minimal Hopf $2$-cocycle $\overline{J}$ for $K/C$, with corresponding $R$-form $\overline{R}^{\overline{J}}$, and injective algebra homomorphism 
$$\overline{\Psi} : {}_{\overline{J}}\mathcal{O}(K/C)_{\overline{J}} \xrightarrow{1:1} U(\mathfrak{k}/\mathfrak{c}).$$
Thus, our induction hypothesis ensures that $\overline{\Psi}$ maps onto $U(\mathfrak{k}/\mathfrak{c})$, so there exist elements $\overline{z}_i\in {}_{\overline{J}}\mathcal{O}(K/C)_{\overline{J}}$, $2 \leq i \leq n-1$, such that we have
\begin{equation}\label{induct}  \overline{\Psi} (\overline{z_i})= \overline{u}_{i}, \quad \textit{for } i = 2, \ldots , n-1,
\end{equation}
where $\overline{\Psi}(\overline{a}) := \overline{R}^{\overline{J}}(-,\overline{a}) $ for $\overline{a} \in {}_{\overline{J}}\mathcal{O}(K/C)_{\overline{J}}$, and $\overline{u}_{i} \in \mathcal{O}(K/C)^*$ is the derivation mapping $\overline{X}_j$ to $\delta_{n+1 - i,j}$, $2 \leq i,j \leq n-1$. (Note that it makes sense to use the notation $\overline{u}_{i}$ since $u_{i} (y_1) = 0$ for $i = 2, \ldots n-1$.)

Recall from Corollary \ref{IHOEcor}(3), and the remarks preceding it, that ${}_J\mathcal{O}(G/C)_J$ is connected. This implies by a result of Schneider \cite{Sch} - see \cite[Theorem 8.4.8 and Example (2) in remarks following]{M} - that the exact sequence of Hopf algebras 
$$ 0 \rightarrow \mathcal{O}(G/K)  \rightarrow {}_J\mathcal{O}(G/C)_J  \rightarrow {}_{\overline{J}}\mathcal{O}(K/C)_{\overline{J}} \rightarrow 0 $$
is cleft. By definition, this means that there is an invertible cleaving map $\gamma$ of ${}_{\overline{J}}\mathcal{O}(K/C)_{\overline{J}}-$comodules, 
$$\gamma:{}_{\overline{J}}\mathcal{O}(K/C)_{\overline{J}}\xrightarrow{1:1}{}_J\mathcal{O}(G/C)_{J}.$$
Thus, we can use $\gamma$ to view ${}_{\overline{J}}\mathcal{O}(K/C)_{\overline{J}}$ as a subspace of ${}_J\mathcal{O}(G/C)_{J}$, so that there is a crossed product decomposition
$${}_J\mathcal{O}(G/C)_{J}=\mathcal{O}(G/K)\#_{\sigma}{}_{\overline{J}}\mathcal{O}(K/C)_{\overline{J}}\subset {}_J\mathcal{O}(G)_J $$ 
by \cite[Theorem 7.2.2]{M}.

Let $\iota:K/C\hookrightarrow G/C$ denote the inclusion morphism, so we have a surjective Hopf algebra map $\iota^* : \mathcal{O}(G/C) \twoheadrightarrow \mathcal{O}(K/C)$, and
$$\iota^*\circ \gamma={\rm id}_{|{}_{\overline{J}}\mathcal{O}(K/C)_{\overline{J}}}.$$
Denote the canonical surjection from $U(\mathfrak{k})$ onto $U(\mathfrak{k}/\mathfrak{c})$ by $\pi$. Then, by definition, for every $\overline{x}\in {}_{\overline{J}}\mathcal{O}(K/C)_{\overline{J}}$, we have 
$$ \pi \circ \Psi(\gamma(\overline{x}))=\overline{\Psi}(\iota^*\gamma(\overline{x}))=\overline{\Psi}(\overline{x}).$$
Thus, lifting back along $\gamma$ we obtain elements 
$$z_i := \gamma (\overline{z}_i)\in {}_J\mathcal{O}(G/C)_J,\,\,\,2\le i \le n-1,$$ such that 
$$ \Psi(z_i)\left(y_1{}_J\mathcal{O}(G/C)_J\right)=0,\,\,\,2\le i \le n-1,$$
since $y_1$ is in the radical of $R^J_{|{}_J\mathcal{O}(G/C)_J}$. Then the induced map on ${}_J\mathcal{O}(G/C)_J$ is
\begin{equation}\label{capture} 
\Psi(z_i)_{|{}_J\mathcal{O}(G/C)_J}=\overline{\Psi}(\iota^*\circ \gamma (\overline{z}_i)) \circ \iota^* = \overline{\Psi}(\overline{z}_i)\circ \iota^* =  \overline{u}_{i}\circ \iota^*  = u_{i\, |{}_J\mathcal{O}(G/C)_J}.
\end{equation}

The second and third lemmas together will prove that $u_2, \ldots, u_{n-1}$ are in $\mathrm{im}(\Psi)$.

\begin{lemma}\label{ClaimII} 
For each $i = 2, \ldots , n-1$, after replacing $z_i$ by $\widehat{z_i} \in {}_J\mathcal{O}(G/C)_J$, 
$$\Psi(\widehat{z_i})=u_i\in \mathrm{im}(\Psi),\,\,\,2 \leq i \leq n-1.$$
\end{lemma}

\begin{proof} 
It suffices to prove that after suitable adjustments changing the elements $z_i$ to $\widehat{z_i} \in {}_J\mathcal{O}(G/C)_J$ for every $2 \leq i \leq n-1$, we have
\begin{align*} \Psi(\widehat{z_2}), \ldots ,&\Psi(\widehat{z_{n-1}}) \textit{ are linearly independent elements of }\mathfrak{g},\\
&\textit{still satisfying (\ref{capture}), and with }  \Psi(\widehat{z_i})(y_n) = 0.
\end{align*}

First, for $i = 2,\ldots , n-1$, define $\tau_i := R^J(y_n, z_i) \in \mathbb{C}$,  and set
$$ \widetilde{z_i}:= z_i - \tau_i y_1 .$$
Then for every $w \in {}_J\mathcal{O}(G/C)_J$, and $2\le i\le n-1$, we have 
\begin{equation}\label{new}
\Psi(\widetilde{z_i})(w)=R^J(w, \widetilde{z_i})= R^J(w, z_i)
\end{equation}
since $y_1$ lies in the radical of $({}_J\mathcal{O}(G/C)_J, R^J)$ by Lemma \ref{supportai2}(1). Moreover, 
\begin{equation}\label{new2} R^J(y_n, \widetilde{z_i})=R^J(y_n,z_i) - \tau_i R^J(y_n, y_1)= \tau_i - \tau_i= 0
\end{equation}
for every $i = 2, \ldots ,n-1$, by Lemma \ref{supportai2}(5).

To complete the proof it remains to show that for every $2\le i\le n-1$ we can adjust $\widetilde{z_i}$ to an element $\widehat{z_i} \in {}_J\mathcal{O}(G/C)_J$, without damaging the versions for $\widehat{z_i}$ of (\ref{new}) and (\ref{new2}), but such that $\Psi(\widehat{z_i})\in \mathfrak{g}$; that is, that 
\begin{equation}\label{sort} \Psi(\widehat{z_i})(\mathfrak{m}^2)=0. 
\end{equation}

Fix $2 \leq i \leq n-1$. We already know by Proposition \ref{include}, that there exists $\ell > 1$ such that
\begin{equation}\label{limit} \Psi(\widetilde{z_i})(\mathfrak{m}^{\ell})= 0.
\end{equation}
Also, by (\ref{capture}) and (\ref{new}), $\Psi(\widetilde{z_i})_{|{}_J\mathcal{O}(G/C)_J} = u_{i\, |{}_J\mathcal{O}(G/C)_J}$, and $u_{i\, |{}_J\mathcal{O}(G/C)_J}$ is a derivation of $\mathcal{O}(G/C)$, hence we have
\begin{equation}\label{safe} \Psi(\widetilde{z_i})\left(\mathfrak{m}^2 \cap {}_J\mathcal{O}(G/C)_J\right)=0.
\end{equation}

For each $n$-tuple ${\bf b}=(b_1, b_2,\dots,b_n) \in (\mathbb{Z}_{\geq 0})^n$ define, using the usual commutative product in $\mathcal{O}(G)$,
\begin{equation}\label{ydef} {\bf y}^{{\bf b}}:= y_1^{b_1}X_2^{b_2}\cdots X_{n-1}^{b_{n-1}}y_n^{b_n}.
\end{equation}
In view of (\ref{limit}) and (\ref{safe}) we must show that for each $i = 2, \ldots , n-1$, $\widetilde{z_i}$ can be modified to $\widehat{z_i}$  so that $\Psi(\widehat{z_i})$ maps to $0$ every monomial in the set 
\begin{equation}\label{setofmons}
\mathcal{M}:= \left\{ {\bf y}^{{\bf b}}\, \bigg |\, 2 \leq \sum_{i = 1}^n b_i < \ell, \, b_n \neq 0 \right\}\subset \mathfrak{m}^2.
\end{equation}

Denote by $\mathcal{N}$ the set of $(n-1)$-tuples ${\bf a}=(a_2,\dots,a_n)\in (\mathbb{Z}_{\geq 0})^{n-1}$ such that 
$$2 \leq \sum_{i = 2}^n a_i < \ell\,\,\,{\rm and}\,\,\,a_n \neq 0.$$ 
Using the usual product in $\mathcal{O}(G)$, for each ${\bf a} \in \mathcal{N}$ define the element 
$$ {\bf z}_{{\bf a}}:=\frac{1}{a_n!\cdots a_2!}y_1^{a_n}\widetilde{z_2}^{a_{n-1}}\cdots \widetilde{z_{n-1}}^{a_2}\in {}_J\mathcal{O}(G)_J.$$
Set $ \mathcal{A}:=\{{\bf z}_{\bf a} \mid {\bf a} \in \mathcal{N} \}.$

To complete the proof of Lemma \ref{ClaimII}, we therefore need:

\begin{lemma}\label{newlemma}
\begin{enumerate}
\item Let ${\bf b} = (b_1, \ldots , b_n) \in (\mathbb{Z}_{\geq 0})^{n}$ with $b_1 > 0$. For every ${\bf z}_{{\bf a}} \in \mathcal{A}$,
$$ R^J({\bf y}^{{\bf b}}, {\bf z}_{{\bf a}})=0.$$
\item For every ${\bf a}\in \mathcal{A}$, and every $n$-tuple ${\bf b} = (0,b_2, \ldots , b_n)\in (\mathbb{Z}_{\geq 0})^{n}$, we have  
$R^J({\bf y}^{{\bf b}},{\bf z}_{{\bf a}}) = 1$ when ${\bf b} = (0, {\bf a})$ and $0$ otherwise.
\end{enumerate}
\end{lemma}

\begin{proof}
(1) By hypothesis, ${\bf y}^{{\bf b}} = y_1 \cdot {\bf z}$ for some monomial ${\bf z} \in {}_J\mathcal{O}(G)_J$.
Thus
$$R^J({\bf y}^{{\bf b}}, {\bf z}_{{\bf a}})=R^J(y_1 \cdot {\bf z}, {\bf z}_{{\bf a}})= \sum R^J( y_1, \left({\bf z}_{{\bf a}}\right)_{1})R^J({\bf z}, \left({\bf z}_{{\bf a}}\right)_{2})=0,
$$
where the second equality follows from Definition \ref{cotridef}(2), and the final equality holds because all the terms  $R^J( y_1, ({\bf z}_{{\bf a}})_{1})$ are $0$ by Lemma \ref{supportai2}(1).

(2) Fix ${\bf a} \in \mathcal{A}$. Then for any ${\bf b} = (0,b_2, \ldots , b_n)\in (\mathbb{Z}_{\geq 0})^{n}$, using the centrality of $y_1$ in ${}_J\mathcal{O}(G)_J$ (Lemma \ref{supportai2}(2)) for the second equality, and Definition \ref{cotridef}(1) for the third, we have (omitting the summation symbol)
\begin{eqnarray*}
\lefteqn{a_n!\cdots a_2!R^J({\bf y}^{{\bf b}},{\bf z}_{{\bf a}})}\\
& = & R^J\left(X_2^{b_2}\cdots X_{n-1}^{b_{n-1}}y_n^{b_n},y_1^{a_n}\widetilde{z_2}^{a_{n-1}}\cdots \widetilde{z_{n-1}}^{a_2}\right)\\
& = & R^J\left(X_2^{b_2}\cdots X_{n-1}^{b_{n-1}}y_n^{b_n}, \widetilde{z_2}^{a_{n-1}}\cdots \widetilde{z_{n-1}}^{a_2}\cdot y_1^{a_n}\right)\\
& = & R^J\left(\left( X_2^{b_2}\cdots X_{n-1}^{b_{n-1}}y_n^{b_n}\right)_1,y_1^{a_n}\right)R^J\left(\left( X_2^{b_2}\cdots X_{n-1}^{b_{n-1}}y_n^{b_n}\right)_2, \widetilde{z_2}^{a_{n-1}}\cdots \widetilde{z_{n-1}}^{a_2}\right)\\
& = & \Psi \left(y_1\right)^{a_n}\left(\left( X_2^{b_2}\cdots X_{n-1}^{b_{n-1}}y_n^{b_n}\right)_1\right)\Psi\left(\widetilde{z_2}^{a_{n-1}}\cdots \widetilde{z_{n-1}}^{a_2}\right)\left(\left( X_2^{b_2}\cdots X_{n-1}^{b_{n-1}}y_n^{b_n}\right)_2\right)\\
& = & u_1^{a_n}\left(\left( X_2^{b_2}\cdots X_{n-1}^{b_{n-1}}y_n^{b_n}\right)_1\right)\Psi\left(\widetilde{z_2}^{a_{n-1}}\cdots \widetilde{z_{n-1}}^{a_2}\right)\left(\left( X_2^{b_2}\cdots X_{n-1}^{b_{n-1}}y_n^{b_n}\right)_2\right),
\end{eqnarray*}
where the last equality follows from Lemma \ref{Claim I}. Using (\ref{comformula}) and the definition of $u_1$ to calculate the terms $u_1^{a_n}\left(\left( X_2^{b_2}\cdots X_{n-1}^{b_{n-1}}y_n^{b_n}\right)_1\right)$, it is straightforward to verify that if $R^J({\bf y}^{{\bf b}},{\bf z}_{{\bf a}})\ne 0$, then $b_n=a_n$ and  
$$a_n!\cdots a_2!R^J({\bf y}^{{\bf b}},{\bf z}_{{\bf a}})=a_n! \Psi \left(\widetilde{z_2}^{a_{n-1}}\cdots \widetilde{z_{n-1}}^{a_2}\right)\left(X_2^{b_2}\cdots X_{n-1}^{b_{n-1}}\right)\ne 0.
$$
Therefore, by (\ref{capture}) and (\ref{new}), we have
$$a_n!\cdots a_2!R^J({\bf y}^{{\bf b}},{\bf z}_{{\bf a}})=a_n!u_2^{a_{n-1}}\cdots u_{n-1}^{a_2}\left(X_2^{b_2}\cdots X_{n-1}^{b_{n-1}}\right)\ne 0.
$$
This implies that if $R^J({\bf y}^{{\bf b}},{\bf z}_{{\bf a}})\ne 0$, then ${\bf a}={\bf b}$ and $R^J({\bf y}^{{\bf b}},{\bf z}_{{\bf a}})=1$, as claimed.
\end{proof}

Return now to the proof of Lemma \ref{ClaimII}. For $i = 2, \ldots , n-1$, set 
$$\widetilde{z_i}':=\sum_{{\bf a}\in \mathcal{A}}R^J({\bf y}^{{\bf a}},\widetilde{z_i}){\bf z}_{{\bf a}}\in {}_J\mathcal{O}(G/C )_J.$$ 
By Lemma \ref{newlemma}, for any monomial ${\bf y}^{{\bf b}}$ in the set $\mathcal{M}$ defined in (\ref{setofmons}), we have
\begin{eqnarray*}
\lefteqn{\Psi(\widetilde{z_i}-\widetilde{z_i}')({\bf y}^{{\bf b}}) =\Psi(\widetilde{z_i})({\bf y}^{{\bf b}})-\Psi(\widetilde{z_i}')({\bf y}^{{\bf b}})}\\
& = &
R^J({\bf y}^{{\bf b}},\widetilde{z_i})-\sum_{{\bf a}\in \mathcal{A}}R^J({\bf y}^{{\bf a}},\widetilde{z_i})R^J({\bf y}^{{\bf b}},{\bf z}_{{\bf a}}) =R^J({\bf y}^{{\bf b}},\widetilde{z_i})- R^J({\bf y}^{{\bf b}},\widetilde{z_i})=0.
\end{eqnarray*}
Thus, replacing $\widetilde{z_i}$ by $\widehat{z_i} := \widetilde{z_i}-\widetilde{z_i}'$ for each $i = 2, \ldots , n-1$ completes the proof of Lemma \ref{ClaimII}, since now $\Psi (\widehat{z_i})$ coincides with $u_i$ on every monomial in $\mathcal{O}(G)$.
\end{proof}

For Theorem \ref{noethdomminunip2} it thus remains only to show that 

\begin{lemma}\label{Claim III}
$u_n \in \mathrm{im}(\Psi)$.
\end{lemma} 

\begin{proof}
By Proposition \ref{include}, $\Psi(y_n)$ lies in $U(\mathfrak{g})$, so we can fix an integer ${\ell}\ge 2$ such that $\Psi(y_n)$ vanishes on $\mathfrak{m}^{\ell}$.

Retain the notation (\ref{ydef}), and for each ${\bf b} \in (\mathbb{Z}_{\geq 0})^n$, let ${\bf y}^{{\bf b*}}$ denote the linear form on $\mathcal{O}(G)$ which is $1$ on ${\bf y}^{{\bf b}}$ and $0$ on every monomial ${\bf y}^{{\bf i}}$ such that ${\bf i}\ne {\bf b}$. Let ${\bf b}=(0,b_2,\dots,b_n) \in (\mathbb{Z}_{\geq 0})^n$. Then ${\bf y}^{{\bf b*}}$ vanishes on some power of the ideal generated by $y_1,X_2, \ldots,X_{n-1},y_n$, so that
\begin{equation}\label{contain}
{\bf y}^{{\bf b*}}\in U(\mathfrak{k})\subset U(\mathfrak{g}).
\end{equation}
Hence, by Lemmas \ref{Claim I} and \ref{ClaimII}, for each ${\bf b}=(0,i_2,\dots,i_n) \in (\mathbb{Z}_{\geq 0})^n$, there exists an element ${\bf z}_{{\bf b}}\in {}_J\mathcal{O}(G/C)_{J}$ such that 
\begin{equation}\label{shrew}{\bf y}^{{\bf b*}}=\Psi({\bf z}_{{\bf b}}).
\end{equation}

For each ${\bf b}=(b_1,\ldots,b_n) \in (\mathbb{Z}_{\geq 0})^n$, define $|{\bf b}|:=\sum_{i=1}^n b_i$. Set
$$z:=\sum_{{\bf b}=(0,b_2,\dots,b_n) \in (\mathbb{Z}_{\geq 0})^n,\,2\le |{\bf b}|<{\ell}}R^J({\bf y}^{{\bf b}},y_n){\bf z}_{{\bf b}}\in {}_J\mathcal{O}(G/C )_J.$$ 
Then by (\ref{contain}) and (\ref{shrew}), we have
$$\Psi(z)=\sum_{{\bf b}=(0,b_2,\dots,b_n) \in (\mathbb{Z}_{\geq 0})^n,\,2\le |{\bf b}|<{\ell}}R^J({\bf y}^{{\bf b}},y_n)\Psi({\bf z}_{{\bf b}})\in U(\mathfrak{k}).$$
We claim that $\Psi(z-y_n)$ vanishes on $\mathfrak{m}^2$; that is,  
$$\Psi(z-y_n)({\bf y}^{{\bf d}})=0$$ 
for every ${\bf d} \in (\mathbb{Z}_{\geq 0})^n$ with $|{\bf d}|\geq 2$. 
Indeed, if ${\bf d}=(0,d_2,\dots,d_n) \in (\mathbb{Z}_{\geq 0})^n$ such that $2\le |{\bf d}|<\ell$, then 
\begin{eqnarray*}
\lefteqn{\Psi(z)({\bf y}^{{\bf d}})=\sum_{{\bf b}=(0,b_2,\dots,b_n) \in (\mathbb{Z}_{\geq 0})^n,\,2\le |{\bf b}|<{\ell}}R^J({\bf y}^{{\bf b}},y_n)\Psi({\bf z}_{{\bf b}})({\bf y}^{{\bf d}})}\\
& = & \sum_{{\bf b}=(0,b_2,\dots,b_n) \in (\mathbb{Z}_{\geq 0})^n,\,2\le |{\bf b}|<{\ell}}R^J({\bf y}^{{\bf b}},y_n)\langle {\bf y}^{{\bf b*}},{\bf y}^{{\bf d}}\rangle
=R^J({\bf y}^{{\bf d}},y_n)=\Psi(y_n)(y^{{\bf d}}).
\end{eqnarray*}

Next if ${\bf b} =(b_1, \ldots , b_n) \in (\mathbb{Z}_{\geq 0})^n$ such that $2\le |{\bf b}|<\ell$ and $b_1\ge 1$, then by Lemma \ref{supportai2}(3), 
$$\Psi(z)({\bf y}^{{\bf b}})=0=\Psi(y_n)({\bf y}^{{\bf b}}).$$
Finally, it is clear that if $|{\bf b}|\geq \ell$, then  
$$\Psi(z)({\bf y}^{{\bf b}})=0=\Psi(y_n)({\bf y}^{{\bf b}}).$$
Thus, $\Psi(z-y_n)$ vanishes on $\mathfrak{m}^2$, as claimed. 

Finally, since ${}_J\mathcal{O}(G)_J$ is cotriangular and $y_1$ is primitive, \cite[Lemma 2.2.2]{Ma1} shows that
$$ R^J(y_n,y_1)=(R^J)^{-1}(y_1,y_n)= R^J(S(y_1), y_n)= -R^J(y_1,y_n).$$
Hence, by Lemma \ref{supportai2}(5),
\begin{equation}\label{value} R^J(y_1, y_n)= -1.
\end{equation}
Thus, $\Psi(z-y_n)(y_1)=R^J(y_1, z)-R^J(y_1, y_n)=1$, so $\Psi(z-y_n)=u_n$, as desired.
\end{proof}
\medskip

The theorem follows now from Lemmas \ref{Claim I}, \ref{ClaimII}, and \ref{Claim III}.
\end{proof}

Combining Proposition \ref{include} and Theorem \ref{noethdomminunip2} yields
 
\begin{theorem}\label{minLie} 
Let $G$ be a unipotent affine algebraic group, and let $J$ be a Hopf $2$-cocycle for $G$ such that the $R$-form $R^J := J_{21}^{-1}\ast J$ is nondegenerate. Then the map $\Psi$ defined at (\ref{caught}) yields an isomorphism of Hopf algebras
$${}_J\mathcal{O}(G)_J\cong\left(U(\mathfrak{g})^J\right)^{{\rm cop}},$$
where $U(\mathfrak{g})^J=U(\mathfrak{g})$ as algebras, with coproduct given by $\Delta^J(\cdot)=J^{-1}*\Delta(\cdot)*J$, where $\Delta$ is the standard cocommutative coproduct of $U(\mathfrak{g})$. \qed
\end{theorem}

\begin{remark}\label{compwg22}
Theorem \ref{minLie} is \cite[Theorem 3.4]{G2}. However, the proof given in \cite{G2} is incomplete. More precisely, the proof in \cite{G2} assumes a positive answer to Question \ref{compwg22q} below, which we are unable to prove. \qed
\end{remark}

\begin{question}\label{compwg22q}
Set $\mathfrak{m}:=\mathcal{O}(G)^+$. Is it true that every power of $\mathfrak{m}$ in the algebra ${}_J\mathcal{O}(G)_J$ contains a power of $\mathfrak{m}$ in the algebra $\mathcal{O}(G)$? \qed
\end{question}

\begin{corollary}\label{teneqvrepcat}
Under the hypotheses of Theorem \ref{minLie}, there are equivalences of tensor categories 
$${\rm Rep}({}_J\mathcal{O}(G)_J)\xrightarrow{\cong}{\rm Rep}(U(\mathfrak{g}))\xrightarrow{\cong}{\rm Rep}(\mathfrak{g}).$$
\end{corollary}

\begin{proof}
By Theorem \ref{minLie}, it suffices to prove that ${\rm Rep}(U(\mathfrak{g}))\cong{\rm Rep}(U(\mathfrak{g})^J)$ as tensor categories.
To this end, note that the functor 
$$\left({\rm Id},{\rm J}\right):{\rm Rep}(U(\mathfrak{g}))\xrightarrow{\cong}{\rm Rep}(U(\mathfrak{g})^J),$$
where ${\rm J}_{V,W}:=J_f(r,\hbar)_{\mid V\otimes W}:V\otimes W\xrightarrow{\cong}V\otimes W$ (\ref{univforJ}), is a well defined tensor equivalence.
\end{proof}

\section{Structure of twisted unipotent groups - the general case}\label{structure2} We assume throughout this section that $G$ is a unipotent algebraic group over $\mathbb{C}$ of dimension $n$, and that $J$ is a Hopf $2$-cocycle for $G$ with corresponding $R$-form $R^J := J_{21}^{-1}*J$ for ${}_J\mathcal{O}(G)_J$. Let $T$ be a chosen representative of the conjugacy class of supports of $J$, and $(\mathfrak{t}, \omega)$ the Lie algebra of $T$, so $(\mathfrak{t},\omega)$ is a quasi-Frobenius Lie algebra as explained in $\S$\ref{2cocycunip}.

\subsection{Crossed product decomposition}\label{crossedsec}  Recall that since $T$ is a closed subgroup of $G$, $\mathcal{O}(G/T)$ and $\mathcal{O}(T\backslash G)$ are left and right coideal subalgebras of $\mathcal{O}(G)$ respectively, and hence are polynomial subalgebras in $\mathrm{dim}(G) - \mathrm{dim}(T)$ variables; see for example \cite{BG}.

Following Proposition \ref{factor}, there is an epimorphism $\pi$ of Hopf algebras from ${}_J\mathcal{O}(G)_J$ onto ${}_J\mathcal{O}(T)_J$ with kernel $I_{R^J}$, the radical of $R^J$. The notation (\ref{mult1}) is used below for the algebras ${}_J\mathcal{O}(G)$ and $\mathcal{O}(G)_J$, and extended in the obvious way to subalgebras of these algebras.
 Finally, recall that $S^J$ denotes the antipode of ${}_J \mathcal{O}(G)_J$.
 
\begin{lemma}\label{basics} 
With the above notation, we have the following:
\begin{enumerate}
\item The algebra ${}_J\mathcal{O}(G/T)_J={}_J\mathcal{O}(G/T)$ is a left coideal subalgebra of $_J\mathcal{O}(G)_J$, and a subalgebra of ${}_J\mathcal{O}(G)$. Correspondingly, $_J\mathcal{O}(T\backslash G)_J =\mathcal{O}(T\backslash G)_J$ is a right coideal subalgebra of ${}_J\mathcal{O}(G)_J$, and a subalgebra of $\mathcal{O}(G)_J$.
\item $I_{R^J}=(\mathcal{O}(T\backslash G)_J)^+ {}_J\mathcal{O}(G)_J=({}_J\mathcal{O}(G/T))^+ {}_J\mathcal{O}(G)_J.$
\item $S^J(\mathcal{O}(T\backslash G)_J)={}_J\mathcal{O}(G/T)$, and vice versa.
\end{enumerate}
\end{lemma}

\begin{proof} 
(1) The maps $\pi:\mathcal{O}(G)\twoheadrightarrow \mathcal{O}(T)$ and $\Delta$ are the same when applied to $_J\mathcal{O}(G)$. Therefore, since ${}_J\mathcal{O}(G/T)_J$ is defined using the coalgebra structure, we see that, as vector subspaces of $\mathcal{O}(G)$, we have
\begin{equation}\label{coinv} {}_J\mathcal{O}(G/T)_J = \mathcal{O}(G/T) = \{f \in \mathcal{O}(G)\mid (\mathrm{id} \otimes \pi)\circ \Delta (f) = f \otimes \overline{1}\}.
\end{equation}
In particular, ${}_J\mathcal{O}(G/T)_J$ is a left coideal subalgebra of ${}_J\mathcal{O}(G)_J$ since $\mathcal{O}(G/T)$ is a left coideal subalgebra of $\mathcal{O}(G)$. Hence, for 
$a,b \in {}_J\mathcal{O}(G/T)_J$, we have
\begin{align*} m_J(a \otimes b) &= \sum J^{-1}(a_1,b_1)a_2b_2J(a_3,b_3)\\
 &=  \sum J^{-1}(a_1,b_1)a_2b_2\epsilon(a_3)\epsilon(b_3)\\
&= \sum J^{-1}(a_1,b_1)a_2b_2.
\end{align*}
That is, ${}_J\mathcal{O}(G/T)_J={}_J\mathcal{O}(G/T)$. This proves the first half of (1); the second half is exactly similar.

(2), (3) These follow from (\ref{coinv}) by standard properties of Hopf ideals in connected Hopf algebras due to Masuoka, see e.g. \cite[Proposition 2.3]{BG}.
\end{proof}

The Hopf algebra $U(\mathfrak{t})^J$ featuring in the next theorem is $U(\mathfrak{t})$ with twisted coproduct, as defined in Theorem \ref{minLie}.

\begin{theorem}{\rm (\cite[Theorem 3.8]{G2})}\label{crossed} 
With the notation and hypotheses given at the start of $\S$\ref{structure2}, ${}_J\mathcal{O}(G)_J$ decomposes as a crossed product,
\begin{align*} {}_J\mathcal{O}(G)_J &\cong {}_J\mathcal{O}(G/T)\#_{\sigma} {}_J\mathcal{O}(T)_J\cong \mathcal{O}(T\backslash G)_J \#_{\sigma'} {}_J\mathcal{O}(T)_J\\
&\cong {}_J\mathcal{O}(G/T)\#_{\sigma} U(\mathfrak{t})^J \cong  \mathcal{O}(T\backslash G)_J \#_{\sigma'} U(\mathfrak{t})^J,
\end{align*}
where $\sigma$ and $\sigma'$ are invertible $2$-cocycles from $({}_J\mathcal{O}(T)_J)^{\otimes2}$ to ${}_J\mathcal{O}(G/T)$ and $\mathcal{O}(T\backslash G)_J$, respectively. \qed
\end{theorem}

\begin{remarks}
(1) Suppose in the above that $T$ is normal in $G$. Then of course $\mathcal{O}(G/T)$ is a Hopf subalgebra of $\mathcal{O}(G)$, and from the proof of Lemma \ref{basics} we see that ${}_J\mathcal{O}(G/T)_J  = \mathcal{O}(G/T)$; in fact by Proposition \ref{centre} below, it is a central (polynomial) Hopf subalgebra of ${}_J\mathcal{O}(G)_J$. But in general, ${}_J\mathcal{O}(G/T)$ is not commutative, as will be seen in Example \ref{ex5}.

(2) We do not know whether the $2$-cocycle $\sigma$ in Theorem \ref{crossed} can be non-trivial. But the action of ${}_J \mathcal{O}(T)_J$ on ${}_J\mathcal{O}(G/T)$ can be non-trivial, as is shown by Examples \ref{ex5} and \ref{ex6} below. \qed 
\end{remarks}

\subsection{Consequences of Theorem \ref{crossed}}\label{crossedconseq} 
Although we are unable to give a precise description of the centre of ${}_J\mathcal{O}(G)_J$ for unipotent affine algebraic groups $G$, we can at least identify an important central subalgebra:

\begin{proposition}{\rm (\cite[Lemma 3.7]{G2})}\label{centre} 
If $G$ is a unipotent affine algebraic group with a Hopf $2$-cocycle $J$ whose support is $T$, then the algebra $\mathcal{O}(T\backslash G / T)$ of functions which are constant on the double cosets of $T$ in $G$ is contained in the centre of ${}_J\mathcal{O}(G)_J$.
\end{proposition}

\begin{proof} 
By \cite[Theorem 3.5]{G1} (without needing the hypothesis of unipotence), there are algebra embeddings 
$$\mathcal{O}(G/T) = \mathcal{O}(G/T)_J \subseteq Z(\mathcal{O}(G)_J)\,\,\,{\rm and}\,\,\,\mathcal{O}(T \backslash G) = {}_J\mathcal{O}(T \backslash G)  \subseteq Z({}_J \mathcal{O}(G)),$$ where the equalities emphasise that the multiplication is the same in the linked algebras. Hence, we have
\begin{equation}\label{hum} 
\mathcal{O}(T\backslash G / T) \subseteq \mathcal{O}(G/T) = \mathcal{O}(G/T)_J \subseteq Z(\mathcal{O}(G)_J) 
\end{equation}
and
\begin{equation}\label{haw} \mathcal{O}(T\backslash G / T) \subseteq  \mathcal{O}(T \backslash G) ={}_J\mathcal{O}(T \backslash G) \subseteq Z(_J \mathcal{O}(G)). 
\end{equation}
Moreover (\ref{hum}) and (\ref{haw}) imply that
\begin{equation}\label{catch}  
\Delta (\mathcal{O}(T\backslash G / T)) \subseteq \mathcal{O}(T \backslash G) \otimes \mathcal{O}(G/T) \subseteq Z({}_J\mathcal{O}(G) \otimes \mathcal{O}(G)_J)),
\end{equation}
since $\mathcal{O}(T\backslash G)$ and $\mathcal{O}(G/T)$ are right and  left coideal subalgebras of $\mathcal{O}(G)$, respectively. Thus, the result follows from (\ref{catch}), and the fact that by Lemma \ref{action0}(3), $\Delta$ embeds ${}_J\mathcal{O}(G)_J$ in ${}_J\mathcal{O}(G) \otimes \mathcal{O}(G)_J$. 
\end{proof}

Recall that an ideal $I$ of a ring $A$ is called \emph{polycentral} if there is a finite subset $\{y_1, \ldots , y_t\}\subset I$, which generate $I$ as a right ideal, such that for all $i = 1, \ldots , t$,
$$ y_i + \left(\sum_{j < i} y_j A\right) \textit{ is central in } A/\left(\sum_{j < i}y_j A\right). $$
The ideal $I$ and its corresponding generating set are called \emph{regular polycentral} if in addition each $y_i$ is not a zero divisor in $A/(\sum_{j<i} y_j A)$.

\begin{theorem}\label{augcent} 
Let $G$ be a unipotent affine algebraic group of dimension $n$, and $J$ a Hopf $2$-cocycle for $G$. Then the augmentation ideal ${}_J\mathcal{O}(G)_J^+$ of ${}_J\mathcal{O}(G)_J$ has a regular polycentral generating set of $n$ elements.
\end{theorem}

\begin{proof} By Theorem \ref{crossed}, with $T$ as usual denoting the support of $J$, and $\mathfrak{t}$ the Lie algebra of $T$, we have
\begin{equation}\label{product} {}_J\mathcal{O}(G)_J \cong {}_J\mathcal{O}(G/T)_J \#_{\sigma} U(\mathfrak{t})^J.
\end{equation}
If $T = G$ then ${}_J\mathcal{O}(G)_J \cong U(\mathfrak{t})$ as an algebra, and the result follows by taking as generators of ${}_J\mathcal{O}(G)_J^+$ a basis of $\mathfrak{t}$ obtained from its upper central series. So we may suppose without loss of generality that $T \subsetneq G$. By standard properties of unipotent groups in characteristic $0$, the normaliser $N_G(T)$ is a closed subgroup of $G$ strictly containing $T$. Thus, $N_G(T)/T$ is a non-trivial unipotent group, so we can find a chain of subgroups
$$ 1 \subset T_1/T \subset T_2/T \subset \cdots \subset T_{\ell}/T = N_G(T)/T,$$
with each subgroup normal in the next, and successive factors isomorphic to $(\mathbb{C},+)$. Repeating with $N_G(T)$ in place of $T$, we eventually obtain a normal subgroup $N$ of $G$ with $T \subseteq N$ and $G/N \cong (\mathbb{C},+)$.

By the discussion at the start of $\S$\ref{minimal2}, the factor $G/N$ of $G$ corresponds to a Hopf subalgebra $H_1 = \mathbb{C}[X]$ of $\mathcal{O}(G)$. Here $X$ is primitive, and for all $g \in G$, $X$ is a constant function on the coset $gN = Ng$. Moreover, for all $g \in G$, we have 
$$ TgT \subset NgN = gN,$$
so that $X$ is constant on the double coset $TgT$. In particular, by Proposition \ref{centre}, $X$ is central in $_J\mathcal{O}(G)_J$. Since $X \in {}_J\mathcal{O}(G/T)^+$, we can factor by the Hopf ideal $X {}_J\mathcal{O}(G)_J$ of ${}_J\mathcal{O}(G)_J$. The ideal $X {}_J\mathcal{O}(G)_J$ is, by construction, contained in $({}_J\mathcal{O}(G/T))^+_J\mathcal{O}(G)_J$. Hence, Lemma \ref{basics}(2) applies, and we obtain a quotient Hopf algebra of GK-dimension $n-1$, namely ${}_J\mathcal{O}(N)_J$. Now one inducts on $n$ to obtain the theorem. 
\end{proof}

\begin{corollary}\label{payoff} 
Let $G$ be a unipotent affine algebraic group, and $J$ a Hopf $2$-cocycle for $G$.
\begin{enumerate}
\item $(S^J)^2 = \mathrm{id}$.
\item ${}_J\mathcal{O}(G)_J$ is Calabi-Yau; that is, its Nakayama automorphism is trivial.
\end{enumerate}
\end{corollary}

\begin{proof} 
Suppose that $\mathrm{dim}(G) =n$, and for brevity, denote ${}_J\mathcal{O}(G)_J$ by $H$. Since $H^+$ has a regular polycentral sequence of $n$ generators by Theorem \ref{augcent}, it follows from the change of rings theorem 
\cite[Lemma, $\S$6.5]{BZ} that, with $\mathbb{C}$ denoting the trivial $H$-module, we have
\begin{equation}\label{triv} \mathrm{Ext}^i_H(\mathbb{C},H)\cong  \delta_{in}\mathbb{C}, 
\end{equation}
where the isomorphism is of $(H,H)$-bimodules. 

Therefore, since ${}_J\mathcal{O}(G)_J$ is Artin-Schelter regular by Corollary \ref{IHOEcor}(2), it follows from \cite[Proposition 4.5]{BZ} that the Nakayama automorphism $\nu$ of $H$ is given by
\begin{equation}\label{coup} \nu = (S^J)^2.
\end{equation}
On the other hand, since $H$ is an IHOE by Proposition \ref{Ore}(5), its only units are the nonzero scalars $\mathbb{C}^{\times}$, so by (\ref{triv}) and \cite[Corollary 4.6]{BZ}, we have
\begin{equation}\label{four} (S^{J})^4 = \mathrm{id}. 
\end{equation} 
However, we know from \cite[Theorem 3.2(viii)]{BOZZ} that the antipode of an IHOE over $\mathbb{C}$ is either of infinite order or has order $2$. The first option is ruled out by (\ref{four}), so (1) is proved, and then (2) follows from (\ref{coup}).   
\end{proof}

\begin{remark}\label{ssqnot1}
If $G$ is not unipotent, $S^J$ may not be involutive, as demonstrated in \cite[Example 5.2]{EG1}. \qed
\end{remark}

\subsection{Stratification of representations by double cosets}\label{repstrat}

We keep the notation as specified at the start of $\S$\ref{structure2}. When $J$ is trivial, so that ${}_J\mathcal{O}(G)_J = \mathcal{O}(G)$ is commutative and $T = \{1\}$, the simple $\mathcal{O}(G)$-modules are of course parametrized by the points of $G$. We generalize this to arbitrary $J$ here.

The starting point is Proposition \ref{centre}. By \cite[Proposition 9.1.7]{McCR}, the endomorphism algebra of each simple ${}_J\mathcal{O}(G)_J$-module $V$ is $\mathbb{C}$, so $V$ is killed by a maximal ideal of the central subalgebra ${}_J\mathcal{O}(T\backslash G / T)_J$. Our target is to make this observation more precise. For $g \in G$, let $Z_g := TgT$ be the double coset of $T$ in $G$ containing $g$, so that $Z_g$ is the orbit containing $g$ under the left action of $T \times T$ on $G$ given by
\begin{equation}\label{act} 
(r,s)\cdot g := rgs^{-1};\,\,\,r,s \in T. 
\end{equation}

\begin{lemma}{\rm (\cite[\S 4]{G2})}\label{doubles} 
Fix $g \in G$, and set $T_g := T \cap gTg^{-1}$.
\begin{enumerate}
\item $Z_g$ is an irreducible closed subset of $G$ of dimension
$$ 2\mathrm{dim}(T) - \mathrm{dim}(T_g ).$$
\item Define the embedding $ \theta: T_g \rightarrow T\times T,\,\,\,a \mapsto (a,g^{-1}ag).$ Then $\theta(T_g)$ acts on $T \times T$ on the right, and 
$$ \mathcal{O}(Z_g) \cong  (\mathcal{O}(T) \otimes \mathcal{O}(T))^{\theta(T_g)}. $$
\item Let $\mathcal{I}(Z_g)$ be the defining ideal of $Z_g$ in $\mathcal{O}(G)$, and let $i:T\hookrightarrow G$ be the inclusion morphism. Then there is an algebra epimorphism
$$ m_g^{\ast} := (i^{\ast} \otimes i^{\ast})\circ(\mathrm{id} \otimes \lambda_{g^{-1}})\circ \Delta : \mathcal{O}(G) \twoheadrightarrow \left(\mathcal{O}(T) \otimes \mathcal{O}(T)\right)^{\theta(T_g)}$$
with kernel $\mathcal{I}(Z_g)$. \qed
\end{enumerate}
\end{lemma}

\begin{proof} 
(1), (2)  Since $T \times T$ is unipotent its orbits on $G$ under the action (\ref{act}) are closed by the theorem of Kostant, Rosenberg and Rosenlicht, \cite[Theorem 17.64]{Mi}, so $Z_g$ is an irreducible closed subset of $G$. The stabilizer of $g$ in $T \times T$ is $\{(a,g^{-1}ag)\mid a \in T \cap gTg^{-1}\}$, so that $Z_g \cong (T \times T)/(T \cap gTg^{-1})$. More precisely, the closed subgroup $\theta(T_g)$ acts on $T \times T$ on the right by
$$ (x,y)\cdot \theta(a) = (xa, g^{-1}a^{-1}gy),$$
for $x,y \in T$ and  $a \in T_g$, yielding an isomorphism of affine varieties
$$ (T \times T)/\theta(T_g) \rightarrow Z_g, \quad \overline{(x,y)}\mapsto xgy.$$
This proves (1) and (2).

(3) The above right action of $\theta(T_g)$ induces a left action of $\theta(T_g)$ on $\mathcal{O}(T)\otimes \mathcal{O}(T)$, given, for $a \in T_g$, $x,y \in T$, and $\alpha, \beta \in \mathcal{O}(T)$, by 
$$ \theta(a)(\alpha \otimes \beta)(x,y)=\alpha(xa)\beta(g^{-1}a^{-1}gy). $$ 
That is, $\theta(a)$ acts via $\rho_a \otimes \lambda_{g^{-1}ag}$, where $\rho$ and $\lambda$ are the right and left regular actions of $G$ on $\mathcal{O}(G)$ as noted in (\ref{action}). So there is an algebra isomorphism
$$ \mathcal{O}(Z_g) \cong (\mathcal{O}(T) \otimes \mathcal{O}(T))^{\theta(T_g)}.$$ 
Now (3) is just a reformulation of the above.
\end{proof}

The value of Lemma \ref{doubles} lies in the following result:

\begin{proposition}{\rm (\cite[Proposition 4.1]{G2})}\label{homo}
Retain the above notation. The map $m_g^{\ast}$ is also a surjective algebra homomorphism from $_J\mathcal{O}(G)_J$ onto 
$({}_J\mathcal{O}(T) \otimes \mathcal{O}(T)_J)^{\theta(T_g)}$. 
Hence, $\mathcal{I}(Z_g)$ is also a two-sided ideal of ${}_J\mathcal{O}(G)_J$. \qed
\end{proposition}

\begin{proof} 
Since $i^{\ast} :{}_J\mathcal{O}(G)_J \rightarrow {}_J\mathcal{O}(T)_J$ is an algebra epimorphism, by Proposition \ref{factor}, it is  enough by Lemma \ref{doubles}(3) to show that $(\mathrm{id} \otimes \lambda_{g^{-1}})\circ \Delta$ is an algebra homomorphism from ${}_J\mathcal{O}(G)_J$ to ${}_J\mathcal{O}(G)_J \otimes \,{}_J\mathcal{O}(G)_J$. This is a routine check.
\end{proof}

In view of Proposition \ref{homo}, we follow \cite{G2} and make the following 
\begin{definition}\label{strat1} For $g \in G$, set ${}_J\mathcal{O}(Z_g)_J := {}_J\mathcal{O}(G)_J/\mathcal{I}(Z_g)$. \qed
\end{definition}

\begin{corollary}\cite[Corollary 4.2]{G2}\label{strat2} 
For any $g \in G$, ${}_J\mathcal{O}(Z_g)_J$ is a noetherian domain, with
$ \mathrm{GKdim}({}_J\mathcal{O}(Z_g)_J)=2\mathrm{dim}(T) - \mathrm{dim}(T_g)$.
\end{corollary}

\begin{proof} 
Since ${}_J\mathcal{O}(Z_g)_J$ is a factor of ${}_J\mathcal{O}(G)_J$, it is affine Noetherian by Corollary \ref{IHOEcor}(1). It is a domain because Proposition \ref{homo} shows that it is a subalgebra of ${}_J\mathcal{O}(T) \otimes \mathcal{O}(T)_J$, which is a Weyl algebra by Theorem \ref{Weyl}. The associated (commutative) graded algebras of $\mathcal{O}(G)$ and of ${}_J\mathcal{O}(G)_J$ with respect to the filtrations induced by their coradical filtrations are identical, since the coradical filtration of these algebras is the same. The same is thus true of the factor algebras $\mathcal{O}(G)/\mathcal{I}(Z_g)$ and ${}_J\mathcal{O}(G)_J/\mathcal{I}(Z_g)$. Denoting this commutative affine algebra by $A$, we have
$$ \mathrm{GKdim}(\mathcal{O}(G)/\mathcal{I}(Z_g))=\mathrm{GKdim}(A)=\mathrm{GKdim}({}_J\mathcal{O}(G)_J/\mathcal{I}(Z_g)) $$
by two applications of \cite[Proposition 6.6]{KL}. The value of this common GK-dimension is given by Lemma \ref{doubles}(1).
\end{proof}

Keep the notation from the start of $\S$\ref{structure2} and from earlier in $\S$\ref{repstrat}. We now provide an elaborated version of the proof of \cite[Theorem 5.1]{G2} for the reader's convenience.

\begin{theorem}{\rm (\cite[Theorem 5.1]{G2})}\label{strata} 
Every simple ${}_J\mathcal{O}(G)_J$-module factors through a unique quotient ${}_J\mathcal{O}(Z_g)_J$.
\end{theorem}

\begin{proof} 
Let $V$ be a simple (left) ${}_J\mathcal{O}(G)_J$-module. By Proposition \ref{centre} and Schur's lemma \cite[Proposition 9.1.7]{McCR}, $V$ is annihilated by $K{}_J\mathcal{O}(G)_J$ for a maximal ideal $K$ of $\mathcal{O}(T\backslash G \slash T)$. Now, $K{}_J\mathcal{O}(G)_J = K\mathcal{O}(G)$ is also an ideal of $\mathcal{O}(G)$, and is clearly $(T \times T)$-invariant under the action of the left and right winding automorphisms of $T$ on $\mathcal{O}(G)$ (\ref{windaut}). Therefore, the same invariance is true of $\sqrt{ K\mathcal{O}(G)}$, where $\sqrt{ K\mathcal{O}(G)}/K\mathcal{O}(G)$ denotes the nilradical of $\mathcal{O}(G)/K\mathcal{O}(G)$. Thus, the variety $\mathcal{V}(\sqrt{ K\mathcal{O}(G)})$ is a disjoint union of orbits $TgT$, all of them closed by Lemma \ref{doubles}(1), as $g$ runs through those elements of $G$ whose maximal ideals $\mathfrak{m}_g$ contain $K\mathcal{O}(G)$. Choose a set of representatives $\mathcal{S} \subset G$ for these orbits. Thus, by the commutative Nullstellensatz, we have 
\begin{equation}\label{split} \bigcap_{g \in \mathcal{S}} \mathcal{I}(Z_g)=\sqrt{ K\mathcal{O}(G)}.
\end{equation}
Note that (\ref{split}) is valid whether viewed in $\mathcal{O}(G)$ or in ${}_J\mathcal{O}(G)_J$, and that for $g \in \mathcal{S}$, the ideals $\mathcal{I}(Z_g)$ of ${}_J\mathcal{O}(G)_J$ are comaximal, because they are by definition comaximal as ideals of $\mathcal{O}(G)$, and since addition of ideals involves only the vector space structure. Hence, $V$ is annihilated by at most one of the ideals $\mathcal{I}(Z_g)$ for $g \in \mathcal{S}$. The theorem is thus equivalent to the claim that
\begin{equation}\label{crux} V \textit{ is killed by } \mathcal{I}(Z_g) \textit{ for some } g \in \mathcal{S}.
\end{equation}

To prove (\ref{crux}), we first determine the $(T\times T)$-invariants of the strata: namely,
\begin{equation}\label{invar} 
{}^T\left({}_J\mathcal{O}(Z_g)_J\right)^T=\,{}^T\left({}\mathcal{O}(Z_g)\right)^T=\mathbb{C}.
\end{equation}
The first equality in (\ref{invar}) holds because the space of invariants is independent of the multiplication on the algebra, and the second equality holds because the maximal ideals of $\mathcal{O}(Z_g)$ form a single $(T \times T)$-orbit intersecting in ${0}$, so that every nonzero $(T \times T)$-invariant in $\mathcal{O}(Z_g)$ must be a unit, while the only units of $\mathcal{O}(Z_g)$ are the scalars $\mathbb{C}^{\times}$ by Lemma \ref{doubles}(2), noting that $\mathcal{O}(T)^{\otimes 2}$ is a polynomial algebra.   

Now observe that by (\ref{split}), there is a canonical embedding of algebras
\begin{equation}\label{into1A} 
\varphi : {}_J\mathcal{O}(G)_J/\sqrt{K\mathcal{O}(G)} \hookrightarrow \Pi_{s \in \mathcal{S}}\,{}_J\mathcal{O}(Z_s)_J.
\end{equation}
Moreover, $\varphi$ is clearly $(T\times T)$-equivariant, so the restriction of $\varphi$ gives an embedding
\begin{equation}\label{into2} 
\varphi:{}^T\left({}_J\mathcal{O}(G)_J/\sqrt{K\mathcal{O}(G)}\right)^T \hookrightarrow \Pi_{s \in \mathcal{S}}\, {}^T({}_J\mathcal{O}(Z_s)_J)^T= \Pi_{s \in \mathcal{S}}\mathbb{C}, 
\end{equation}
where the final equality follows from (\ref{invar}). From (\ref{into1A}) and (\ref{into2}) we see that
\begin{equation}\label{caught}  
{}^T\left({}_J\mathcal{O}(G)_J/\sqrt{K\mathcal{O}(G)}\right)^T \subset Z\left({}_J\mathcal{O}(G)_J/\sqrt{K\mathcal{O}(G)}\right).
\end{equation}

If $|\mathcal{S}| = 1$, then ${}_J\mathcal{O}(Z_g)_J={}_J\mathcal{O}(G)_J/\sqrt{K\mathcal{O}(G)}$, and (\ref{crux}) is immediate by definition of $K$. So suppose that $|\mathcal{S}| > 1$, and let $TxT$ be a double coset with $x \in \mathcal{S}$. Then $T x T$ is preserved by $T \times T$, and so therefore is $\mathcal{I}(Z_x)$. Therefore, $T \times T$ acts rationally on the non-zero space $\mathcal{I}(Z_x)/\sqrt{K\mathcal{O}(G)}$. Since $T \times T$ is unipotent, its only rational irreducible module is the trivial module, hence there exists $0 \neq \alpha \in \mathcal{I}(Z_x) \setminus \sqrt{K\mathcal{O}(G)}$ with $\alpha$ being $(T \times T)$-invariant. 

By (\ref{caught}), $\alpha \in Z\left({}_J\mathcal{O}(G)_J/\sqrt{K\mathcal{O}(G)}\right)$, but $\alpha$ is not constant on $\mathcal{V}(\sqrt{K\mathcal{O}(G)})$ since $\alpha(TxT) = 0$ whereas $TxT \subsetneq \mathcal{V}(\sqrt{K\mathcal{O}(G)})$ and $\alpha \notin \sqrt{K\mathcal{O}(G)}$. Hence, by the noncommutative Nullstellensatz \cite[Proposition 9.1.7]{McCR}, there exists $\lambda \in \mathbb{C}$ such that $(\alpha - \lambda)V = 0$, and $\alpha - \lambda \notin \sqrt{K\mathcal{O}(G)}$. Define 
$$ I_1:= \sqrt{K\mathcal{O}(G)} + (\alpha - \lambda)\mathcal{O}(G), $$
so that $I_1$ is a $(T \times T)$-invariant ideal of both $\mathcal{O}(G)$ and of ${}_J\mathcal{O}(G)_J$ with 
$$ \sqrt{K\mathcal{O}(G)} \subsetneq I_1 \textit{  and  } I_1V = 0. $$
Now $\mathcal{V}_1 := \mathcal{V}(\sqrt{I_1})$ is a proper subvariety of $\mathcal{V}(\sqrt{K\mathcal{O}(G)})$ and is a union of $(T \times T)$-orbits labelled by a proper subset $\mathcal{S}_1$ of $\mathcal{S}$. If $|\mathcal{S}_1| = 1$, say $\mathcal{S}_1 = \{y \}$, then $\sqrt{I_1} V = 0$ with $\sqrt{I_1} = \mathcal{I}(Z_{y}) $, proving (\ref{crux}).

If on the other hand $|\mathcal{S}_1| > 1$, then we can repeat the above argument with $\sqrt{I_1}$ replacing $\sqrt{K\mathcal{O}(G)}$. Since $\mathcal{O}(G)$, or equivalently, ${}_J\mathcal{O}(G)_J$ are noetherian, this process must terminate after finitely many repeats, so that (\ref{crux}) follows, completing the proof of the theorem.   
\end{proof}

\subsection{On the structure of the algebras ${}_J\mathcal{O}(Z_g)_J$}\label{stratstruc}

Retain the notation and hypotheses from the start of $\S$\ref{structure2}, the definitions of the algebras ${}_J\mathcal{O}(Z_g)_J$ from Definition \ref{strat1}, and the ideals $\mathcal{I}(Z_g)$ (of $\mathcal{O}(G)$ and of ${}_J\mathcal{O}(G)_J$) from Lemma \ref{doubles}(3).  Let $g \in G$, and recall from Lemma \ref{doubles} that we denote $T \cap gTg^{-1}$ by $T_g$. We divide the discussion of ${}_J\mathcal{O}(Z_g)_J$ into the following three cases, recalling that this algebra depends only on the double coset $TgT$ of $T$ in $G$ to which $g$ belongs. Here, and in $\S$\ref{simples2}, we denote the normaliser $N_G(T)$ of $T$ in $G$ by $N$.
\medskip

\noindent {\bf Case I:} $T_g = \{1\}$.
\medskip

\noindent {\bf Case II:} $T_g = T$; that is, $g \in N$.
\medskip

\noindent {\bf Case III:} $\{1\} \subsetneq T_g \subsetneq T$.
\medskip

\noindent We treat {\bf Cases I} and {\bf II} separately in the following results, keeping the above notation $G$, $J$ and $T$ in force throughout.

\begin{theorem}{\rm (\cite[Remark 4.4]{G2})}\label{caseIstruc}
Let $g \in G$ be such that {\bf Case I} holds; that is $T_g = \{1\}$. Then ${}_J\mathcal{O}(Z_g)_J$ is a Weyl algebra; specifically, 
$$ {}_J\mathcal{O}(Z_g)_J\cong A_{\mathrm{dim}(T)}(\mathbb{C}). $$
\end{theorem}

\begin{proof} 
In this case $\theta (T_g) = \{1\}$, so ${}_J\mathcal{O}(Z_g)_J \cong {}_J\mathcal{O}(T) \otimes \mathcal{O}(T)_J$ by Proposition \ref{homo}. The result follows from Theorem \ref{Weyl}(2),(4).
\end{proof}
 
In contrast to Case I, the Case II strata are no longer in general mutually isomorphic algebras. The following is a refinement and elaboration of Theorem \ref{crossed}.

\begin{theorem}\label{caseIIstruc}
The following hold: 
\begin{enumerate}
\item The Hopf ideal $\mathcal{O}(G/N)^+\mathcal{O}(G)$ of $\mathcal{O}(G)$ is also a Hopf ideal of ${}_J\mathcal{O}(G)_J$, so $\mathcal{O}(G/N)^+\mathcal{O}(G) = \mathcal{O}(G/N)^+{}_J\mathcal{O}(G)_J = ({}_J\mathcal{O}(G/N)_J)^+{}_J\mathcal{O}(G)_J $, and  
$${}_J\mathcal{O}(G)_J/({}_J\mathcal{O}(G/N)_J)^+{}_J\mathcal{O}(G)_J \cong{}_J\mathcal{O}(N)_J$$
as Hopf algebras.
\item ${}_J\mathcal{O}(G/T)_J/({}_J\mathcal{O}(G/N)_J)^+{}_J\mathcal{O}(G/T)_J$ is isomorphic to the central (polynomial) Hopf subalgebra $\mathcal{O}(N/T)$ of the quotient Hopf algebra ${}_J\mathcal{O}(N)_J$ of ${}_J\mathcal{O}(G)_J$, with 
\begin{equation}\label{twister} {}_J\mathcal{O}(N)_J\cong \mathcal{O}(N/T)\#_{\sigma}\left(U(\mathfrak{t})^J\right),
\end{equation}
where $\sigma$ is an invertible 2-cocycle from $\left(U(\mathfrak{t})^J\right) \otimes \left(U(\mathfrak{t})^J\right)$ to $\mathcal{O}(N/T)$.
\item Let $g$ be such that {\bf Case II} holds; that is $g \in N$. Then $\mathcal{I}(Z_g)$ is generated by its intersection with ${}_J\mathcal{O}(G/T)_J$. That is, we have
$$\mathcal{I}(Z_g)\cap {}_J\mathcal{O}(G/T)_J = \mathfrak{m}_{gT},$$
where $\mathfrak{m}_{gT}$ denotes the functions in $\mathcal{O}(G/T)$ which are 0 on $gT$. Thus, $\mathfrak{m}_{gT}$ is an ideal of ${}_J\mathcal{O}(G/T)_J$, which contains $({}_J\mathcal{O}(G/N)_J)^+{}_J\mathcal{O}(G/T)_J$, with ${}_J\mathcal{O}(G/T)_J/\mathfrak{m}_{gT} \cong\mathbb{C}$. Moreover,
$\mathcal{I}(Z_g) =\mathfrak{m}_{gT}{}_J\mathcal{O}(G)_J$.
\end{enumerate} 
\end{theorem}

\begin{proof}
(1) Working first in $\mathcal{O}(G)$, observe that 
$$\mathcal{O}(G/N)^+ \mathcal{O}(G)=\bigcap_{x \in N}\mathfrak{m}_x,$$
where as usual, for $x \in G$, $\mathfrak{m}_x = \{f \in \mathcal{O}(G)\mid f(x) = 0 \}$. Since for  each $g \in N$,
$$ \mathcal{I}(Z_g) = \bigcap_{x \in gT} \mathfrak{m}_x,$$
by the definition of $\mathcal{I}(Z_g)$ in Lemma \ref{doubles}(3) and its proof, it follows that 
\begin{equation}\label{collect}
\mathcal{O}(G/N)^+ \mathcal{O}(G)=\bigcap_{g \in N}\mathcal{I}(Z_g).
\end{equation}
Since each $\mathcal{I}(Z_g)$ is an ideal of ${}_J\mathcal{O}(G)_J$ by Proposition \ref{homo}, we see from (\ref{collect}) that $\mathcal{O}(G/N)^+ \mathcal{O}(G)$ is an ideal of ${}_J\mathcal{O}(G)_J$. Since the comultiplication is unchanged in the switch from $\mathcal{O}(G)$ to ${}_J\mathcal{O}(G)_J$, $\mathcal{O}(G/N)^+ \mathcal{O}(G)$ is a Hopf ideal of ${}_J\mathcal{O}(G)_J$. The displayed isomorphism now follows because $\mathcal{O}(G/N)^+\mathcal{O}(G)$ is contained in the radical of $R^J$ by Lemma \ref{basics}(2), so that ${}_J\mathcal{O}(N)_J$ is well-defined, with an epimorphism to it from ${}_J\mathcal{O}(G)_J$ as in Proposition \ref{factor}.

(2) Working in ${}_J\mathcal{O}(N)_J$, we see that since $T \triangleleft N$, the first claim is a special case of Proposition \ref{centre}. Moreover, $\mathcal{O}(N/T)$ is a polynomial algebra because $N/T$ is a unipotent group. Finally, (\ref{twister}) is a special case of Theorem \ref{crossed}, where in this case the action is trivial thanks to the centrality of $\mathcal{O}(N/T)$.

(3) The displayed equality is proved by first noting that the equivalent statement is obviously true in $\mathcal{O}(N)$, hence also true in ${}_J\mathcal{O}(N)_J$ since $\mathcal{O}(N/T)$ is central in ${}_J\mathcal{O}(N)_J$. Finally, we can lift the statements back to ${}_J\mathcal{O}(G)_J$ using the displayed isomorphism in (1). 
\end{proof}

For any $g\in G$, let 
\begin{equation}\label{windaut}
\begin{split}
\tau^{\ell}_g :\mathcal{O}(G)\rightarrow \mathcal{O}(G),\,\,\,\alpha \mapsto \sum \alpha_1(g)\alpha_2,\\
\tau^{r}_g :\mathcal{O}(G)\rightarrow \mathcal{O}(G),\,\,\,\alpha \mapsto \sum \alpha_1\alpha_2(g),
\end{split}
\end{equation}
be the associated winding automorphisms of $\mathcal{O}(G)$.

\begin{lemma}\label{stratIIfinish}
Let $g\in G$. The following hold:
\begin{enumerate}
\item 
{\rm (\cite[Lemma 3.16]{G2})} $\tau^{\ell}_g :{}_J\mathcal{O}(G)_J\rightarrow {}_{J^g}\mathcal{O}(G)_J$ 
is an algebra isomorphism, where ${}_{J^g} \mathcal{O}(G)_J $ is defined at (\ref{mult0}).

\item 
Assume that $g \in N$. With ${}_{J^g} \mathcal{O}(T)_J $ defined as at (\ref{mult0}), there is an algebra isomorphism 
${}_J\mathcal{O}(Z_g)_J\cong {}_{J^g} \mathcal{O}(T)_J$.
\end{enumerate}
\end{lemma}

\begin{proof}
(1) Note that $\tau^{\ell}_g$ is at least a linear isomorphism of the underlying vector space, and let $\alpha, \beta \in {}_J\mathcal{O}(G)_J$. Use $\cdot$ [resp.$\ast$] to denote multiplication in ${}_J\mathcal{O}(G)_J$ [resp. in ${}_{J^g}\mathcal{O}(G)_J$], with juxtaposition for multiplication in $\mathcal{O}(G)$. Then
\begin{align*} \tau^{\ell}_g (\alpha \cdot\beta) &= \sum \tau^{\ell}_g\left(J^{-1} (\alpha_1 ,\beta_1) \alpha_2\beta_2 J(\alpha_3, \beta_3)\right)\\
&= \sum J^{-1} (\alpha_1 ,\beta_1)\tau^{\ell}_g(\alpha_2)\tau^{\ell}_g(\beta_2) J(\alpha_3, \beta_3)\\
&= \sum J^{-1} (\alpha_1 ,\beta_1)\alpha_2(g)\beta_2(g) \alpha_3\beta_3 J(\alpha_4, \beta_4).
\end{align*}
On the other hand, 
\begin{align*} 
\tau^{\ell}_g(\alpha )\ast \tau^{\ell}_g(\beta) &=   \left(\sum \alpha_1(g)\alpha_2\right) \ast \left(\sum \beta_1(g)\beta_2\right)= \sum \alpha_1(g)\beta_1(g)\alpha_2 \ast \beta_2\\
&= \sum \alpha_1(g)\beta_1(g) J^{-1}\left(\alpha_2(g^{-1})\alpha_3 \alpha_4(g),\beta_2(g^{-1})\beta_3\beta_4(g)\right)\alpha_5\beta_5J(\alpha_6,\beta_6)\\
&= \sum\epsilon(g)\epsilon(g)J^{-1}(\alpha_1,\beta_1)\alpha_2(g)\beta_2(g)\alpha_3\beta_3J(\alpha_4,\beta_4)\\
&= \sum J^{-1}(\alpha_1, \beta_1) \alpha_2(g)\beta_2(g)\alpha_3\beta_3J(\alpha_4,\beta_4).
\end{align*}
This proves (1).

(2) The isomorphism $\tau^{\ell}_g$ from (1) maps $\mathfrak{m}_{gT}\mathcal{O}(G)$ onto $\mathfrak{m}_{T}\mathcal{O}(G)$. Since $g \in N$, Theorem \ref{caseIIstruc}(3) shows that 
$${}_J\mathcal{O}(Z_g)_J = {}_J\mathcal{O}(G)_J/\mathcal{I}(gT)\cong  {}_{J^g}\mathcal{O}(G)_J/\mathcal{I}(T) = {}_{J^g}\mathcal{O}(T)_J,$$
as claimed.
\end{proof}

\begin{question}\label{compwg2q2}
Let $g\in N$. Is it true that ${}_J\mathcal{O}(Z_g)_J\cong U^{\omega_g}(\mathfrak{t})$ for a $2$-cocycle $\omega_g$ of $\mathfrak{t}$? If yes, is it true that $\omega_g=\omega - \omega \circ (\mathrm{Ad}g \otimes \mathrm{Ad}g)$, where $\omega$ is the $2$-cocycle of $\mathfrak{t}$ corresponding to $J$? (In \cite[Theorem 4.5]{G2} it is claimed without a proof that the answer to the second question is positive.) \qed
\end{question}

Although we do not currently have detailed information on the structure of the {\bf Case III} strata, we note here that if $g \in G$ is such that {\bf Case III} or {\bf Case I} holds that is, if $g \notin N$, then ${}_J\mathcal{O}(Z_g)_J$ has no nonzero finite dimensional modules, by Theorem \ref{caseIstruc}, and Theorem \ref{findims} below.

\section{Finite dimensional simple ${}_J\mathcal{O}(G)_J$-modules}\label{simples2}

In this section we study the group of finite dimensional simple ${}_J\mathcal{O}(G)_J$-modules. (In particular, we correct \cite[Theorem 5.2]{G2}.) Since the finite dimensional simple ${}_J\mathcal{O}(G)_J$-modules are $1$-dimensional by Corollary \ref{IHOEcor}(5), the set of these modules forms a group, which we denote in this section by $\Gamma$. We show in $\S$\ref{subgp} that $\Gamma$ is a closed subgroup of $G$, and then we give a precise description of $\Gamma$ in $\S$\ref{gamma}, proving Theorem \ref{findims} (= Theorem \ref{findimintro}).

\subsection{$\Gamma$ is a subgroup of $G$}\label{subgp} We start with the following elementary observation.

\begin{lemma}\label{free} 
Let $k$ be a field, $n \geq 1$, and $F = k\langle X_1,\dots, X_n \rangle$ the free $k$-algebra of rank $n$. Let $D :=\langle [X_i,X_j]\mid 1\leq i < j \leq n \rangle$ be the commutator ideal of $F$, so $ F/D \cong k[X_1,\dots, X_n]$. Let $S := F/E$ be any factor algebra of $F$. Then the commutator ideal $\langle [S,S] \rangle$ of $S$ is $I := (D + E)/E$. 
\end{lemma}

\begin{proof} 
Since $D \subseteq D + E$, $S/I$ is commutative, so $\langle [S,S] \rangle \subseteq I$. On the other hand, the image in $S$ of every generator $[X_i,X_j]$ of $D$ lies in $\langle [S,S]\rangle$, so $I \subseteq  \langle [S,S] \rangle $.
\end{proof}

Now we apply Lemma \ref{free} in the setting of Remarks \ref{compwg2}(1). That is, in the lemma we take $k = \mathbb{C}$ and $n = \mathrm{dim}(G)$, so $\mathcal{O}(G) \xrightarrow{\pi_D}F/D$, $X_i\mapsto X_i+D$, is an algebra isomorphism, and we define the ideal $E$ of $F$ to be the ideal $E$ of Remarks \ref{compwg2}(1), so that ${}_J\mathcal{O}(G)_J \xrightarrow{\pi_E}F/E$, $X_i\mapsto X_i+E$, is an algebra isomorphism. Then by Lemma \ref{free}, $(D +E)/E$ is the commutator ideal of $F/E$. 

Now let $I:=\langle [{}_J\mathcal{O}(G)_J, {}_J\mathcal{O}(G)_J]\rangle$ be the commutator ideal of ${}_J\mathcal{O}(G)_J$. Since the commutator ideal of any Hopf algebra is a Hopf ideal, $I$ is a Hopf ideal of ${}_J\mathcal{O}(G)_J$, and $\pi_E$ induces a canonical Hopf algebra isomorphism  
$${}_J\mathcal{O}(G)_J/I\xrightarrow{\overline{\pi_E}} (F/E)/((D + E)/E),\,\,\,X_i+I\mapsto \pi_E(X_i)+(D + E)/E.$$
In particular, since the algebra $(F/E)/((D + E)/E)$ is canonically isomorphic to the algebra $F/(D + E)$, we see that the algebra $F/(D + E)$ is a factor Hopf algebra of ${}_J\mathcal{O}(G)_J$ in a canonical way. 

Now set $\mathscr{I}:=\pi_D^{-1}\left ((D + E)/D\right )$; it is an ideal of $\mathcal{O}(G)$, and $\pi_D$ induces a canonical algebra isomorphism  
$$\mathcal{O}(G)/\mathscr{I}\xrightarrow{\overline{\pi_D}} (F/D)/((D + E)/D),\,\,\,X_i+\mathscr{I}\mapsto \pi_D(X_i)+(D + E)/D.$$
Since the algebra $(F/D)/((D + E)/D)$ is canonically isomorphic to the algebra $F/(D + E)$, it follows from the above that $\mathscr{I}$ is a Hopf ideal of $\mathcal{O}(G)$, so the algebra $F/(D + E)$ is also a factor Hopf algebra of $\mathcal{O}(G)$ in a canonical way.

It follows that the Hopf algebras $\mathcal{O}(G)/\mathscr{I}$ and ${}_J\mathcal{O}(G)_J/I$ are canonically isomorphic via the ``identity map" 
$$\mathcal{O}(G)/\mathscr{I}\xrightarrow{{\rm id}} {}_J\mathcal{O}(G)_J/I,\,\,\,X_i+\mathscr{I}\mapsto X_i+I,$$
hence $I=\mathscr{I}$, so $I$ is also a Hopf ideal of $\mathcal{O}(G)$.

Thus, combining the above with Remarks \ref{compwg2}(1), we have obtained the following.

\begin{theorem}\label{abel} 
Let $G$ be a unipotent affine algebraic group, $J$ a Hopf $2$-cocycle for $G$, $I:=\langle [{}_J\mathcal{O}(G)_J, {}_J\mathcal{O}(G)_J]\rangle$ the commutator ideal of ${}_J\mathcal{O}(G)_J$, and $\Gamma$ the group of finite dimensional simple ${}_J\mathcal{O}(G)_J$-modules. The following hold:
\begin{enumerate}
\item ${}_J\mathcal{O}(G)_J/I \cong \mathcal{O}(\Gamma)$, so $\Gamma$ is a unipotent affine algebraic group.
\item
$I$ is a Hopf ideal of both ${}_J\mathcal{O}(G)_J$ and $\mathcal{O}(G)$.
\item
The identity map $\mathcal{O}(G)/I\xrightarrow{{\rm id}} {}_J\mathcal{O}(G)_J/I$ is a Hopf algebra isomorphism.
\item
Every maximal ideal of ${}_J\mathcal{O}(G)_J$ of codimension $1$ is also a maximal ideal of $\mathcal{O}(G)$. 
\item
$\Gamma$ is the closed subgroup of $G$ whose defining ideal is $I\subseteq \mathcal{O}(G)$. \qed
\end{enumerate}
\end{theorem}

\subsection{The determination of $\Gamma$}\label{gamma} 
In this subsection we obtain an explicit description of the subgroup $\Gamma$ of $G$. We deal first in Theorem \ref{strfindimsF} with the case where $J$ is minimal, then address the general case from Lemma \ref{alghomIJ} onwards. The main result of this subsection is Theorem \ref{findims}.

Retain the notation of $\S$\ref{stratstruc} and Lemma \ref{stratIIfinish}. 
By \cite[Proposition 3.3 and Definition 3.2]{ES}, $(\mathfrak{t}, [\,,\,],r,\delta)$ is a \emph{nondegenerate coboundary Lie bialgebra} with cobracket 
\begin{equation}\label{cobracket}
\delta:\mathfrak{t}\to \wedge^2\mathfrak{t},\,\, x \mapsto [x\otimes 1+1\otimes x,r];
\end{equation} 
here, as in Theorem \ref{classify}, $r:=\omega^{-1}\in \wedge^2\mathfrak{t}$ is the associated nondegenerate solution to the classical Yang-Baxter equation, and we can restate (\ref{cobracket}) as $\delta = \partial_1(r)$. By \cite[Definitions 3.4]{ES}, this means that  $(\mathfrak{t}, [\,,\,],r,\partial_1(r))$ is a \emph{triangular} Lie bialgebra.

Recall also from \cite[Proposition 2.2]{ES} that $\mathfrak{t}^*$ is a Lie bialgebra with bracket $\delta^*$ and cobracket $[\,,\,]^*$, and the map 
\begin{equation}\label{liealgselfdual}
F_r:\mathfrak{t}^*\xrightarrow{\cong}\mathfrak{t},\,\,\alpha \mapsto (\alpha\otimes {\rm id})(r),
\end{equation}
is an isomorphism of Lie bialgebras, where \cite[Proposition 2.1.4]{CP} shows that the map is a homomorphism of Lie bialgebras, with bijectivity following from nondegeneracy combined with \cite[Proposition 3.2]{ES}. 
We consider first the group $\Gamma \cap T$ of finite dimensional simple representations of ${}_J\mathcal{O}(T)_J$. We show in (1) and (2) of the following result that $\Gamma \cap T$ is isomorphic to a particular closed abelian subgroup $F$ of $T$. Later, in Theorem \ref{findims}(3), we shall see that this isomorphism is in fact an equality.

\begin{theorem}\label{strfindimsF} 
Retain the above notation. The following hold:
\begin{enumerate}
\item
There is a group isomorphism $\Gamma \cap T \cong T/[T,T]$, so in particular $\Gamma \cap T$ is a unipotent abelian affine algebraic group. 
\item
$\Gamma \cap T$ is canonically isomorphic to a closed abelian subgroup $F$ of $T$.
\item 
$\mathrm{Lie}(F)= \mathrm{ker}(\delta) = \{x\in \mathfrak{t}\mid [x\otimes 1 +1\otimes x,r]=0\}$. 
\item
$F =\{t\in T \mid r^t = r\}=\{t \in T \mid J^t = J \}$.
\end{enumerate}
\end{theorem} 

\begin{proof}
(1) This follows from Corollary \ref{teneqvrepcat}.

(2) It is sufficient to show that ${\rm Lie}(T/[T,T]) = \mathfrak{t}/[\mathfrak{t},\mathfrak{t}]$ is isomorphic to a Lie subalgebra of $\mathfrak{t}$. To this end, using the discussion and notation at the start of this subsection, note that  
the canonical projection $\pi:\mathfrak{t}\twoheadrightarrow \mathfrak{t}/[\mathfrak{t},\mathfrak{t}]$ is a Lie bialgebra map by (\ref{cobracket}). Since $\mathfrak{t}/[\mathfrak{t},\mathfrak{t}]\xrightarrow{\cong} \left(\mathfrak{t}/[\mathfrak{t},\mathfrak{t}]\right)^*$ as Lie bialgebras, we obtain an injective map of Lie bialgebras 
$$\mathfrak{t}/[\mathfrak{t},\mathfrak{t}]\xrightarrow{\cong} \left(\mathfrak{t}/[\mathfrak{t},\mathfrak{t}]\right)^*\xrightarrow{\pi^*} \mathfrak{t}^*\xrightarrow{F_r} \mathfrak{t},$$ 
as desired. Denote the Lie subalgebra $F_r \circ \pi^*((\mathfrak{t}/[\mathfrak{t},\mathfrak{t}])^*)$ of $\mathfrak{t}$ by $\mathfrak{f}$. Thus, $\mathfrak{f}$ is an abelian Lie subalgebra of $\mathfrak{t}$ of dimension $\mathrm{dim}(T/[T,T])$, and (2) follows with $F$ being the closed abelian subgroup of $T$ with $\mathrm{Lie}(F) = \mathfrak{f}$.

(3) Since $\delta$ is a $1$-cocycle, $\mathrm{ker}(\delta)$ is a Lie subalgebra of $\mathfrak{t}$. By the proof of Theorem \ref{strfindimsF}, 
$\delta(F_r(\pi^*(\alpha)))=0$ for every $\alpha\in \left(\mathfrak{t}/[\mathfrak{t},\mathfrak{t}]\right)^*$. Namely, $r\in \left(\wedge^2 \mathfrak{t}\right)^{\mathfrak{f}}$ is invariant under the adjoint action of $\mathfrak{f}$ on $\wedge^2 \mathfrak{t}$. Thus, $\mathfrak{f}\subseteq \mathrm{ker}(\delta)$.

Now since $\mathfrak{t}^*\xrightarrow{\cong}\mathfrak{t}$ as Lie bialgebras, it follows that ${\rm dim}(\mathrm{ker}(\delta))={\rm dim}(\mathrm{ker}([\,,\,]^*)$. But clearly, $\mathrm{ker}([\,,\,]^*)=\left(\mathfrak{t}/[\mathfrak{t},\mathfrak{t}]\right)^*$. Hence,
by Theorem \ref{strfindimsF}, we have 
$${\rm dim}(\mathrm{ker}(\delta))={\rm dim}(\left(\mathfrak{t}/[\mathfrak{t},\mathfrak{t}]\right)^*)={\rm dim}(\mathfrak{f}),$$ which implies the statement.

(4) By the well known correspondence between closed subgroups of $T$ and Lie subalgebras of $\mathfrak{t}$, we have $F=\{t\in T \mid r^t = r\}$, since by (3), both closed subgroups have the same Lie algebra. Also, by Remarks \ref{equivrmk}(3), $J^t$ corresponds to $r^t$ for any $t\in T$. Thus, for any $t\in T$, $J^t=J$ if and only if $r^t=r$. This implies that $\{t\in T \mid r^t = r\}=\{t \in T \mid J^t = J \}$, as claimed. 
\end{proof}

\begin{remark}\label{FRemarks}
Let $\mathfrak{t}$ have basis $\{a,b,c,d\}$ with nonzero brackets $[a,b]=c$, $[c,b]=d$, and let $r:=a\wedge c + d\wedge b$ (see Example \ref{ex4} below). Then $\mathfrak{f}= {\rm sp}_{\mathbb{C}}\{c,d\}$ is an abelian ideal of $\mathfrak{t}$, which is strictly contained in the abelian ideal ${\rm sp}_{\mathbb{C}}\{a,c,d\}$. Thus, $\mathfrak{f}$ is not always a maximal abelian subalgebra of $\mathfrak{t}$. \qed
\end{remark}

By Theorem \ref{strata}, every simple ${}_J\mathcal{O}(G)_J$-module is a simple ${}_J\mathcal{O}(Z_g)_J$-module for a unique stratum $Z_g$, where for $g,h \in G$, 
${}_J\mathcal{O}(Z_g)_J = {}_J\mathcal{O}(Z_h)_J$ if and only if $TgT = ThT$. Thus, we can analyse the group $\Gamma$ of finite dimensional simple ${}_J\mathcal{O}(G)_J$-modules by considering the finite dimensional simple ${}_J\mathcal{O}(Z_g)_J$-modules for each $g \in G$.

Let $(\mathcal{O}(G)^*)_1^{\times}$ denote the group of invertible elements $\chi\in \mathcal{O}(G)^*$ with $\epsilon(\chi)=1$, and note that $G \subseteq (\mathcal{O}(G)^*)_1^{\times}$. 
Recall from Remark \ref{Jrems}(3) that $(\mathcal{O}(G)^*)_1^{\times}$ acts on the set of Hopf $2$-cocycles $J$ for $G$ via 
$$J^{\chi}:=(\chi\circ m)*J*\left(\chi^{-1}\otimes \chi^{-1}\right);\,\,\,\chi\in \left(\mathcal{O}(G)^*\right)_1^{\times},$$ 
where $m$ is the multiplication map. Let $[J]$ denote the orbit of $J$ under the action of the subgroup $(\mathcal{O}(T)^*)_1^{\times}$ of $(\mathcal{O}(G)^*)_1^{\times}$.

\begin{lemma}\label{alghomIJ}
Let $L$ and $J$ be Hopf $2$-cocycles for $G$, and let $\chi \in (\mathcal{O}(G)^*)_1^{\times}$. 
\begin{enumerate}
\item $\chi:{}_L \mathcal{O}(G)_J\to \mathbb{C}$ is an algebra homomorphism if and only if $L=J^{\chi}$. 
\item $\chi$ yields a representation of ${}_J\mathcal{O}(G)_J$ if and only if $J = J^{\chi}$.
\item Suppose that $L$ and $J$ both have support $T$. Then the algebra ${}_L\mathcal{O}(T)_J$ has $1$-dimensional representations if and only if $[L] = [J]$.
\end{enumerate}
\end{lemma}

\begin{proof} 
(1) Let $\cdot$ denote the multiplication in ${}_L\mathcal{O}(G)_J$, as in (\ref{mult0}). Then a linear map $\chi:{}_L\mathcal{O}(G)_J\to \mathbb{C}$ is an algebra homomorphism if and only if we have
\begin{eqnarray*}
\lefteqn{(\chi\otimes \chi)(x\otimes y)=\chi(x)\chi(y)=\chi(x\cdot y)}\\
& = & L^{-1}(x_1\otimes y_1)\chi(x_2y_2)J(x_3\otimes y_3)=L^{-1}(x_1\otimes y_1)(\chi\circ m)(x_2\otimes y_2)J(x_3\otimes y_3)\\
& = & \left(L^{-1}*(\chi\circ m)*J\right)(x\otimes y)
\end{eqnarray*}
for all $x,y \in {}_L\mathcal{O}(G)_J$. But this is equivalent to $\chi\otimes \chi=L^{-1}*(\chi\circ m)*J$, that is to 
$$L=(\chi\circ m)*J*(\chi^{-1}\otimes \chi^{-1})=J^{\chi},$$ as claimed.

(2) This follows immediately from (1).

(3) Apply the special case $G = T$ of (1). 
\end{proof}

\begin{notation}\label{groups} 
Viewing $J$ as a Hopf $2$-cocycle for $T$, define the following closed subgroups of $N = N_G(T)$, where as above, $[J]$ denotes the orbit of $J$ under the action of $(\mathcal{O}(T)^*)_1^{\times}$:
\begin{equation}\label{smallgroup}
C_0:=\{ g \in N \mid J^g = J \}\,\,\,{\rm and}\,\,\, 
C :=\{ g \in N \mid [J^g] = [J] \}.
\end{equation}
\end{notation}

\begin{remarks}\label{grouprems} 
(1) Clearly, $C_0 \subseteq C$ is a closed subgroup, and $T \subseteq C$ is a closed normal subgroup (since $C \subseteq N = N_G(T)$ and, by definition, $[J^t] = [J]$ for $t\in T$).

(2) $\Gamma$ is the group of grouplike elements of $\left({}_J\mathcal{O}(G)_J\right)^{\circ}$, so is a subgroup of $\left(\mathcal{O}(G)^*\right)_1^{\times}$. More precisely, by Theorem \ref{abel} and Lemma \ref{alghomIJ}(2),
\begin{equation}\label{Gammagrplels}
\Gamma =\{\chi\in (\mathcal{O}(G)^*)_1^{\times} \mid J^{\chi}=J\} \subseteq G.
\end{equation}

(3) In view of (\ref{Gammagrplels}, $Z(G) \subseteq \Gamma$. \qed
\end{remarks}

We can now state and prove our main result on the group $\Gamma$ of 1-dimensional representations of ${}_J\mathcal{O}(G)_J$.

\begin{theorem}\label{findims} 
Retain the notation used throughout this section, so $G$ is an affine algebraic unipotent group, $J$ is a Hopf $2$-cocycle for $G$ with support $T$, $N = N_G(T)$, $C_0 \subseteq C$ are the subgroups of $N$ defined at (\ref{smallgroup}), and $F$ is the closed subgroup of $T$ defined at Theorem \ref{strfindimsF}(3),(4). 
\begin{enumerate}
\item $\Gamma = C_0,$ and $C_0 = \{g \in G \mid J^g = J\}$.
\item $C = C_0T.$
\item $F=C_0 \cap T$ is a closed abelian subgroup of $T$, normal in $C_0$, with 
 $$ \mathrm{dim}(F)= \mathrm{dim}(T/[T,T]).$$
\item 
$C_0/F \cong C/T$, so that there is an exact sequence of algebraic groups
$$ 1 \rightarrow F \rightarrow \Gamma \rightarrow C/T \rightarrow 1.$$
\item 
$\mathrm{dim}(\Gamma)=\mathrm{dim}(C) - \mathrm{dim}([T,T]).$
\item 
The algebra ${}_J\mathcal{O}(Z_g)_J$ has nonzero finite dimensional representations if and only if $g\in C$. The distinct strata admitting nonzero finite dimensional modules are $\{{}_J\mathcal{O}(Z_g)_J \mid g \in \mathcal{C}\},$ where $\mathcal{C}$ is a set of coset representatives for $F=C_0 \cap T$ in $C_0$, or equivalently for $T$ in $C$.
\item 
For any $g \in \mathcal{C}$, the set of finite dimensional simple ${}_J\mathcal{O}(Z_g)_J$-modules is parametrized by $gF=g(C_0 \cap T)$. Specifically, the corresponding functionals have kernels the maximal ideals $\{\mathfrak{m}_{gx} \mid x\in C_0 \cap T \}$ of $\mathcal{O}(G)$.
\end{enumerate}
\end{theorem}

\begin{proof} 
(1) Let $\chi \in \Gamma$. By (\ref{Gammagrplels}), $\chi$ is a grouplike element of $\left({}_J\mathcal{O}(G)_J\right)^{\circ}$, with 
\begin{equation}\label{preprevious}
\chi \in G \textit{  and  } J=J^{\chi}=(\chi\circ m)*J*(\chi^{-1}\otimes \chi^{-1}).
\end{equation}
It follows that 
$R^J=(\chi\otimes \chi)*R^J*(\chi^{-1}\otimes \chi^{-1})\in \left(\mathcal{O}(T)\otimes \mathcal{O}(T)\right)^*$. Note that $\left({}_J\mathcal{O}(T)_J,R^J\right)$ is the minimal cotriangular quotient Hopf algebra of $\left({}_J\mathcal{O}(G)_J,R^J\right)$ by Lemma \ref{radical} (quoted from \cite[Section 2.2]{G1}), and that since $J^{\chi}=J$, the map
$${}_J\mathcal{O}(G)_J\to {}_J\mathcal{O}(G)_J,\,\,x\mapsto 
\sum \chi(x_1)x_2\chi^{-1}(x_3),$$
is a Hopf algebra automorphism by Lemma \ref{alghomIJ}. Therefore, it follows that 
\begin{equation}\label{addon}
(\chi\otimes \chi)*\left(\mathcal{O}(T)\otimes \mathcal{O}(T)\right)^* *(\chi^{-1}\otimes \chi^{-1})=\left(\mathcal{O}(T)\otimes \mathcal{O}(T)\right)^*
\end{equation}
inside the algebra $\left(\mathcal{O}(G)\otimes \mathcal{O}(G)\right)^*$. Hence, $\chi \in N = N_G(T)$ and so $\Gamma \subseteq C_0$ by definition of the latter group. Conversely, if $g \in C_0$ then $g \in \Gamma$ by Lemma \ref{alghomIJ}(2).

The second equality in (1) is also proved in the argument given above.

(2), (6) Suppose that $g \in G$ is such that the stratum ${}_J\mathcal{O}(Z_g)_J$ admits a $1$-dimensional simple module $\chi$. Then, by (1),  $\chi \in C_0 \subseteq N$. Hence, by definition of $Z_g$, $\chi = gt$ for some element $t \in T$. Thus $g = \chi t^{-1} \in C_0 T$. 

Conversely, suppose $g \in C$, so that $[J^g] = [J]$ by definition of $C$. That is, there exists $t_0 \in T$ such that $J^g = J^{t_0}$. Hence
$$ gt_0^{-1} \in \{ x \in N \mid J^x = J \} = C_0.$$
Therefore $C = C_0 T$ and $\chi := gt_0^{-1}$ is a 1-dimensional representation of ${}_J\mathcal{O}(Z_g)_J$. 

The final sentence of (6) is a consequence of the first together with Theorem \ref{strata}, since the latter guarantees that every finite dimensional simple ${}_J\mathcal{O}(G)_J$-module factors through ${}_J\mathcal{O}(Z_g)_J$ for a unique $g\in \mathcal{C}$.

(3) That $F = C_0 \cap T$ is immediate from Theorem \ref{strfindimsF}(4) and the definition of $C_0$ at (\ref{smallgroup}). The dimension of $F$ was also determined at Theorem \ref{smallgroup}. Since $C_0 \subseteq N = N_G(T)$ it is clear that $C_0 \cap T$ is normal in $C_0$.

(4) The isomorphism holds because $C_0 \cap T = F \triangleleft C_0$ by (1), and $C_0 T = C$ by (2). The exact sequence follows at once, since $C_0 = \Gamma$ by (1).

(5) This is a straightforward consequence of (3) and (4).

(7) Let $g \in C_0$. There is an algebra isomorphism from ${}_J\mathcal{O}(T)_J = {}_J\mathcal{O}(Z_1)_J$ to ${}_J\mathcal{O}(Z_g)_J$ given by the automorphism $\tau^{\ell}_g$ of ${}_J\mathcal{O}(G)_J$, as shown by Lemma \ref{stratIIfinish}. Clearly, this map carries the group $F=C_0 \cap T$ of 1-dimensional representations of ${}_J\mathcal{O}(T)_J$ to the set $gF=g(C_0 \cap T)$ of all such representations of ${}_J\mathcal{O}(Z_g)_J$.
\end{proof}

\section{Special cases and examples}\label{examples} 

\subsection{Abelian support}\label{abelian} Let $G$ be a unipotent affine algebraic group, and $J$ a Hopf $2$-cocycle for $G$. When is ${}_J\mathcal{O}(G)_J$ commutative? Notice that by Theorem \ref{nIHOE}, this is the case if and only if all the derivations appearing in the presentation of ${}_J\mathcal{O}(G)_J$ as an IHOE are trivial, and hence holds  if and only if ${}_J\mathcal{O}(G)_J = \mathcal{O}(G)$ via the identity map. By the definition (\ref{mult}) of the multiplication in ${}_J\mathcal{O}(G)_J$, this happens if and only if 
\begin{equation}\label{commute} m \ast J = J \ast m,
\end{equation}
where $m$ denotes the multiplication in $\mathcal{O}(G)$. By definition, see \cite[$\S$2.2]{EG4}, (\ref{commute}) holds if and only if $J$ belongs to the set $Z^2_{\mathrm{inv}}(G)$ of \emph{invariant} Hopf $2$-cocycles. Invariant Hopf $2$-cocycles were defined and studied in \cite{S2}; see also \cite{BC} where the adjective \emph{lazy} is used for them.  In \cite[Theorem 4.2]{EG4}, it is shown (for \emph{any} affine algebraic group over $\mathbb{C}$) that if $J \in  Z^2_{\mathrm{inv}}(G)$ then the support of $J$ is a closed normal abelian subgroup of $G$. The outcome is the following result:

\begin{proposition}\label{nochange} Let $G$ be an affine algebraic group, and $J$ a Hopf $2$-cocycle for $G$. Then the following are equivalent:
\begin{enumerate}
\item ${}_J\mathcal{O}(G)_J$ is commutative.
\item ${}_J\mathcal{O}(G)_J \cong \mathcal{O}(G)$ as Hopf algebras.
\item ${}_J\mathcal{O}(G)_J = \mathcal{O}(G)$; that is, the identity map is an isomorphism of Hopf algebras.
\item $J \in Z^2_{\mathrm{inv}}(G)$.
\item The support of $J$ is an abelian normal subgroup of $G$.
\end{enumerate}
\end{proposition}

\begin{proof} The equivalence of (1)-(4), and the proof of $(4)\Rightarrow (5)$ are explained before the proposition. That $(5)\Rightarrow (4)$ follows from \cite[Theorem 6.1 and the first three lines of its proof]{EG4}.
\end{proof}

For the following corollary recall from Proposition \ref{factor} and Definition \ref{supp} that the support of the Hopf $2$-cocycle $J$ of $G$ is only defined up to conjugacy in $G$.

\begin{corollary}\label{abeliansupp} 
Let $G$ be an affine algebraic unipotent group, $J$ a Hopf $2$-cocycle for $G$ with abelian support $T$, and $N:=N_G(T)\subseteq G$ the normaliser of $T$.
\begin{enumerate}
\item 
We have $ \langle [ {}_J\mathcal{O}(G)_J, {}_J\mathcal{O}(G)_J] \rangle =  \mathcal{O}(G/N)^+{}_J\mathcal{O}(G)_J$, 
and $${}_J\mathcal{O}(G)_J/\langle [ {}_J\mathcal{O}(G)_J, {}_J\mathcal{O}(G)_J] \rangle = \mathcal{O}(N).$$
\item 
The group $\Gamma$ of finite dimensional simple ${}_J\mathcal{O}(G)_J$-modules is $N$. 
\end{enumerate} 
\end{corollary}

\begin{proof}
(1) By Theorem \ref{caseIIstruc}(1), the right hand side of the displayed equality is a Hopf ideal of ${}_J\mathcal{O}(G)_J$, and the factor of ${}_J\mathcal{O}(G)_J$ by this ideal is isomorphic as a Hopf algebra to ${}_J\mathcal{O}(N)_J$. In particular, this factor is commutative and equal to $\mathcal{O}(N)$, by Proposition \ref{nochange} applied to it. It follows that 
\begin{equation}\label{star}
\langle [ {}_J\mathcal{O}(G)_J, {}_J\mathcal{O}(G)_J] \rangle \subseteq \mathcal{O}(G/N)^+{}_J\mathcal{O}(G)_J.
\end{equation} 

Suppose now that the inclusion in (\ref{star}) is strict. Recall that 
$\langle [{}_J\mathcal{O}(G)_J, {}_J\mathcal{O}(G)_J] \rangle $ is always a Hopf ideal (in any Hopf algebra), so the factor by it, being commutative and in characteristic 0, is reduced  
by \cite[\S 11.4]{W}. Hence, the strict inclusion in (\ref{star}), combined with the commutative Nullstellensatz, shows that there is a $1$-dimensional ${}_J\mathcal{O}(G)_J$-module $V$ which is not killed by $\mathcal{O}(G/N)^+{}_J\mathcal{O}(G)_J$. But by Theorem \ref{strata}, $V$ factors through ${}_J\mathcal{O}(Z_g)_J$ for a unique $g \in G$, and by Theorem \ref{findims}(5), $g \in C \subseteq N$, so that $\mathcal{O}(G/N)^+{}_J\mathcal{O}(G)_J \subset \mathcal{I}(Z_g)$, with $\mathcal{I}(Z_g)V = 0$. This is a contradiction, and hence the both parts of (1) are proved.

(2) This is immediate from (1). 
\end{proof}

\subsection{Examples}\label{examples2} In the following examples $E_{ij}$ will denote the $n \times n$ matrix with $1$ in the $(i,j)^{\mathrm{th}}$ entry and zeros elsewhere; $n$ will be clear from the context. The first example is trivial since we already know from Proposition \ref{Ore}(7), or from Proposition \ref{nochange}, that when $G$ is abelian there are \emph{no} non-trivial twists of $\mathcal{O}(G)$. Nevertheless the calculation here will be used in some later examples. 

\begin{example}(\cite[\S 5]{G1}, \cite[\S 6]{G2})\label{ex1} 
Let $G = (\mathbb{C},+)^2$, so that $\mathcal{O}(G) = \mathbb{C}[X,V]$ is a polynomial Hopf algebra. Let $\mathfrak{g} = \mathrm{Lie}(G) = \mathbb{C}\frac{\partial}{\partial{X}} + \mathbb{C}\frac{\partial}{\partial{V}}$, and take $ r:=\frac{\partial}{\partial{X}} \wedge \frac{\partial}{\partial{V}}$. Then $J:=  (\epsilon \otimes \epsilon)\circ e^{r/2}\in \left(\mathcal{O}(G) \otimes \mathcal{O}(G)\right)^*$ is a minimal Hopf $2$-cocycle for $G$,  
\begin{equation}\label{find}
\begin{split} 
& J(X,X) = J(V,V) = 0,\, J(X, V)= J^{-1}(V,X) = 1/2,\\
& \hspace{2cm} J^{-1}(X,V)=J(V,X)=- 1/2,
\end{split} 
\end{equation}
and ${}_J\mathcal{O}(G)_J = \mathcal{O}(G)$. \qed
\end{example}

\begin{example}(\cite[Example 6.1]{G2})\label{ex2} 
Let $G$ be the Heisenberg group 
$$ G =\{g(x,v,y) := 1 + xE_{12} + vE_{13} + yE_{23} \mid x,v,y \in \mathbb{C} \} \subseteq {\rm U}(3, \mathbb{C}). $$
We will employ here, and in later examples, the following obvious notation: for $g = g(x,v,y)\in G$, define $X \in\mathcal{O}(G)$ by $X(g) = x$, and analogously for $V$ and $Y$, so $\mathcal{O}(G)  =  \mathbb{C}[X,V,Y]$ with $X$ and $Y$ primitive and $\Delta (V) =  V \otimes 1 + 1 \otimes V + X \otimes Y$. 

Take for $T$ the normal abelian subgroup $ \{1 + xE_{12} + vE_{13} \mid x,v \in \mathbb{C} \}$ of $G$. Thus, $\mathrm{Lie}(G) = \mathfrak{g} =  \mathbb{C}a + \mathbb{C}b + \mathbb{C}c$, where $a := \frac{\partial}{\partial{X}},$ $ b := \frac{\partial}{\partial{Y}}$ and $c := \frac{\partial}{\partial{V}}$, so $[a,b] = c$. So, $\mathfrak{t} := \mathrm{Lie}(T)$ is the abelian Lie subalgebra of $\mathfrak{g}$ spanned by $a$ and $c$. As in Example \ref{ex1}, $ r  :=  a\wedge c\in \wedge^2\mathfrak{t}$ is a non-degenerate $\mathfrak{g}$-invariant solution to the CYBE.  Hence $J:=(\epsilon \otimes \epsilon) \circ e^{r/2}$ is a minimal Hopf $2$-cocycle for $T$ and, by pullback along the surjection from $\mathcal{O}(G)$ onto $\mathcal{O}(T)$, a (non-minimal) Hopf $2$-cocycle for $G$. 

By Proposition \ref{nochange}, ${}_J\mathcal{O}(G)_J = \mathcal{O}(G)$. Of course it is a straightforward exercise to check this directly using Proposition \ref{Ore}. \qed
\end{example}

\begin{example}(\cite[Example 6.2]{G2})\label{ex3} 
Let $G$ be the $4$-dimensional unipotent group
$$ G = \{1 + xE_{12} + vE_{13} + wE_{14} + y(E_{23}+E_{34}) + (y^2/2)E_{24} \mid x,v,w,y \in \mathbb{C} \} \subseteq {\rm U}(4,\mathbb{C}).$$
With similar notation to Example \ref{ex2}, $\mathcal{O}(G)  = \mathbb{C}[X,V,W,Y],$ with 
$$ \Delta (V)  =  V \otimes 1 + 1 \otimes V + X \otimes Y, \; \Delta (W)  = W \otimes 1 + 1 \otimes W + V \otimes Y + X \otimes (Y^2/2),$$
and $X$, $Y$ primitive. Set $\mathfrak{m} := \mathcal{O}(G)^+$ and $ a := \frac{\partial}{\partial X} , \, b :=\frac{\partial}{\partial Y} , \, c := \frac{\partial}{\partial V}, \, d := \frac{\partial}{\partial W}$ in $(\mathfrak{m}/\mathfrak{m}^2)^{\ast}$, so that $\mathfrak{g} := \mathrm{Lie}(G)$ has basis $\{a,b,c,d\}$ with non-zero brackets  $[a,b] = c$ and $[c,b] = d$.

Choose as support $T:= \{1 + xE_{12} + vE_{13} \mid x,v \in \mathbb{C}\},$ so that $T \cong (\mathbb{C},+)^2$ and $\mathcal{O}(T) = \mathbb{C}[X,V]$. Choose $r := a\wedge c \in \wedge^2 \mathfrak{t}$, with $J := (\epsilon \otimes \epsilon)\circ e^{r/2}$ as in Example \ref{ex1}. Thus, as Hopf algebras, $_J\mathcal{O}(T)_J = \mathcal{O}(T).$
Pull $J$ back to $\mathcal{O}(G)$ along the map sending $Y$ and $W$ to $0$. As in (\ref{find}),
$$ J(X,V) = J^{-1}(V,X) = 1/2,\,\, J(V,X) = J^{-1}(X,V) = -1/2,$$
and it is clear from the definitions that the value taken by $J$ and $J^{-1}$ on all other generator pairs is 0. It is easy to check that $[X,Y] = [X,V] = [Y,V] = 0$
in ${}_J\mathcal{O}(G)_J$, as the coproducts of $X$, $Y$ and $V$ are as in Example \ref{ex2}. Similarly, a simple check shows that $[Y,W] = 0$ in ${}_J\mathcal{O}(G)_J$. 

Consider now $[W,X]$ in ${}_J\mathcal{O}(G)_J$. We find using Theorem \ref{Ore}(3) that 
$$W \cdot X = WX + \frac{1}{2} Y\,\,\,\,{\rm and}\,\,\,\, X \cdot W =  XW - \frac{1}{2}Y,$$ hence $ [W,X] = Y$. Similarly, we have
$$W \cdot V = WV + \frac{1}{4} Y^2\,\,\,\,{\rm and}\,\,\,\, V \cdot W = VW - \frac{1}{4} Y^2,$$ hence $ [W,V] = \frac{1}{2} Y^2 $ in ${}_J\mathcal{O}(G)_J$. Thus, we have
\begin{equation}\label{ex3present}{}_J \mathcal{O}(G)_J=\mathbb{C}\langle X,V,W,Y \mid [W,X]=Y,\,[W,V] = Y^2/2, \textit{all other brackets $0$}\rangle.
\end{equation}

Since $T$ is abelian and  
$$ N := N_G(T) = \{ g(x,v,w,0)\mid x,v,w \in \mathbb{C} \} \lhd G, $$
with $G/N \cong (\mathbb{C},+)$ and $\mathcal{O}(G/N) = \mathbb{C}[Y]$, Corollary \ref{abeliansupp} implies (with the obvious notation) that
$$ {}_J\mathcal{O}(G)_J/[\langle {}_J\mathcal{O}(G)_J,{}_J\mathcal{O}(G)_J \rangle] = \mathbb{C}[\overline{X},\overline{V},\overline{W}] \cong \mathcal{O}((\mathbb{C},+)^3),$$
with $\langle [{}_J\mathcal{O}(G)_J,{}_J\mathcal{O}(G)_J ]\rangle = Y{}_J\mathcal{O}(G)_J$. All this can be checked directly using (\ref{ex3present}). And we can easily confirm that, as predicted by Corollary \ref{abeliansupp},
$$ \Gamma = N = C = C_0.$$

It is easy to check that the only central primitive elements are the scalar multiples of $Y$. More precisely, a straightforward calculation shows that 
$$ \textit{the centre  of }{}_J\mathcal{O}(G)_J \textit{ is }\mathbb{C}[Y] = \mathcal{O}(G\backslash T/G).$$

Consider next the stratification of the representation theory of ${}_J \mathcal{O}(G)_J$ by the double cosets $Z_g :=TgT$ of $T$ in $G$, using  the notation introduced in $\S$\ref{repstrat}. 
\medskip

\noindent{\bf Case I:} $g = g(x_0,v_0,w_0,0) \in N$. Here,  
$$ Z_g=TgT=Tg = \{g(x,v,w_0,0) \mid x,v \in \mathbb{C} \}. $$
Therefore the defining ideal of $Z_g$ in $\mathcal{O}(G)$ is $ \mathcal{I}(Z_g) = \langle Y, W - w_0 \rangle$. By Proposition \ref{homo}, $\mathcal{I}(Z_g)$ is also an ideal of ${}_J \mathcal{O}(G)_J$ (as is clear), and 
$${}_J \mathcal{O}(G)_J/\mathcal{I}(Z_g):= {}_J \mathcal{O}(Z_g)_J\cong \mathbb{C}[X,V] \cong  U(\mathfrak{t}), $$
as predicted by Theorem \ref{caseIIstruc}.
\medskip

\noindent{\bf Case II:} $g = g(x_0,v_0, w_0,y_0) \in G, \, y_0 \neq 0$. Here,
$$ Z_g = TgT = \{g(x,v,w,y_0)\mid x,v,w \in \mathbb{C} \}.$$
Hence $ \mathcal{I}(Z_g)=\langle Y - y_0 \rangle$, so that
$${}_J \mathcal{O}(Z_g)_J= \mathbb{C}[\overline{V} - \frac{1}{2}y_0\overline{X}] \otimes \mathbb{C}\langle \overline{W}, \overline{X} \rangle\cong \mathbb{C}[t] \otimes A_1(\mathbb{C}).\quad\quad\quad\quad \quad\quad\quad\quad\quad\quad\qed$$
\end{example}

\begin{example}(\cite[Example 6.3]{G2})\label{ex4} Let $G$ be as in Example \ref{ex3}, keep the same notation for $\mathcal{O}(G)$ and for $\mathfrak{g}$, but now take 
$$ r:= a \wedge c + d \wedge b .$$ 
Then $r$ is a non-degenerate solution of the CYBE in $\wedge^2 \mathfrak{g}$, so it corresponds to a minimal Hopf $2$-cocycle $J$ for $G$. Similar calculations to Example \ref{ex3} yield the following formulae for $J$:
\begin{align*} J(X,V) &= J^{-1}(V,X) = J(W,Y) = J^{-1}(Y,W) = 1/2,\\
J(V,X) &= J^{-1}(X,V) = J(Y,W) = J^{-1}(W,Y) = -1/2,
\end{align*}
with $J$ taking the value 0 on all other pairs of generators. Calculating as before to obtain relations for $_J \mathcal{O}(G)_J$, we find that ${}_J \mathcal{O}(G)_J = \mathbb{C}\langle X,V,W,Y \rangle$ with relations $[W,X] = Y$, $[W,V] = Y^2/2 + X$, with all other pairs of generators commuting. Replacing $X$ by $X' := X + Y^2/2$ gives 
\begin{equation}\label{ex4display} {}_J \mathcal{O}(G)_J = \mathbb{C}\langle X',V,W,Y \mid [W,X'] = Y,\, [W,V] = X', \textit{other brackets }0 \rangle.
\end{equation}  
Comparing the above relations  with the ones for $\mathfrak{g}$ at the start of Example \ref{ex3}, we see that as an  algebra
$$ {}_J \mathcal{O}(G)_J \cong U(\mathfrak{g}),$$
as predicted by Theorem \ref{minLie}.

From (\ref{ex4display}) we easily show that 
$$\langle [{}_J\mathcal{O}(G)_J, {}_J\mathcal{O}(G)_J]\rangle = {}_J\mathcal{O}(G)_J Y + {}_J\mathcal{O}(G)_J X, $$
so that 
$$ {}_J\mathcal{O}(G)_J/\langle [{}_J\mathcal{O}(G)_J, {}_J\mathcal{O}(G)_J] \rangle =  \mathbb{C}[\overline{V}, \overline{W}]. $$
That is, the group of finite dimensional simple ${}_J\mathcal{O}(G)_J$-modules is
\begin{equation}\label{findimex4} \Gamma = \{ I + vE_{13} + wE_{14} \mid v,w \in \mathbb{C}\}.
\end{equation}
One can check (for example using Theorem \ref{strfindimsF}(3) that $\Gamma=F = C_0 \cap T = C_0$. \qed
\end{example}

\begin{example}(\cite[Example 6.4]{G2})\label{ex5}  
Let $G={\rm U}(4,\mathbb{C})$. 
We have $\mathrm{dim}(G) = 6$, and 
$$\mathcal{O}(G) = \mathbb{C}[F_{12},F_{23},F_{34}, F_{13}, F_{24}, F_{14}],$$ where $F_{ij}(A) = a_{ij}$ ($A = (a_{k\ell}) \in G$). So the first three generators are primitive, and
\begin{align*} \Delta (F_{13}) = F_{13} \otimes 1 + &1 \otimes F_{13} + F_{12} \otimes F_{23},\,\,\Delta(F_{24}) =  F_{24} \otimes 1 + 1 \otimes F_{24} + F_{23} \otimes F_{34},\\
\Delta(F_{14}) &= F_{14} \otimes 1 + 1 \otimes F_{14} + F_{13} \otimes F_{34} + F_{12} \otimes F_{24}.
\end{align*}
Take as support the subgroup $T:=\{ xE_{12} + uE_{34} \mid x,u \in \mathbb{C} \}  \cong (\mathbb{C},+)^2 \subseteq G$, so $\mathcal{O}(T) = \mathbb{C}[F_{12}, F_{34}]$. Write $a := \partial/\partial F_{12}, \, c := \partial/\partial F_{34}$, and define $r:= a \wedge c$ and $J:=(\epsilon \otimes \epsilon)\circ e^{r/2}$. Pull $J$ back through the obvious Hopf epimorphism $\pi$ from $\mathcal{O}(G)$ to $\mathcal{O}(T)$, to define a Hopf $2$-cocycle on $G$. 
\medskip

\noindent{\bf Algebra structure:} Just as in Example \ref{ex1}, 
$$ J(F_{12},F_{34})= J^{-1}(F_{34},F_{12})=1/2,\, \,J^{-1}(F_{12},F_{34}) =J(F_{34},F_{12})=-1/2,$$
with the value of $J$ being $0$ on all other generator pairs. From this one calculates that the commutator relations between the generators of ${}_J\mathcal{O}(G)_J$ are as follows:
\begin{equation}\label{twice}
[F_{12},F_{24}] = F_{23}, \; [F_{12},F_{14}] = F_{13},\;[F_{34},F_{13}] = F_{23},\; [F_{34},F_{14}] = F_{24},
\end{equation}
with the other brackets being $0$. It follows from Theorem \ref{Ore} that
$$ {}_J\mathcal{O}(G)_J \textit{ is the } \mathbb{C}\textit{-algebra generated by the }F_{ij} \textit{ with relations } (\ref{twice}).$$
In particular, defining $\mathfrak{n} := \sum_{i,j}\mathbb{C}F_{ij}$, a $6$-dimensional nilpotent Lie subalgebra of ${}_J\mathcal{O}(G)_J$, we see that
$$ \textit{ as an algebra, }{}_J\mathcal{O}(G)_J \cong U(\mathfrak{n}), \textit{ the enveloping algebra of a nilpotent Lie algebra.}$$
\medskip
\noindent {\bf The group $\Gamma$ of finite dimensional simple modules:} Since the support $T$ of $J$ is abelian, Corollary \ref{abeliansupp} tells us that
\begin{equation}\label{normal} \Gamma = N_G(T) = \{I + xE_{12} + uE_{34} + zE_{14} \mid x,u,z \in \mathbb{C}\} \cong (\mathbb{C},+)^3. 
\end{equation}
We can easily check this by direct computation. Namely, from (\ref{twice}), we have  
$$ \langle [{}_J\mathcal{O}(G)_J,{}_J\mathcal{O}(G)_J]\rangle = \langle F_{23}, F_{13}, F_{24} \rangle,$$ 
where the ideal in $\mathcal{O}(G)$, or in ${}_J\mathcal{O}(G)_J$, generated by the three elements above is the same vector subspace, namely the defining Hopf ideal of $N_G(T)$ in $\mathcal{O}(G)$.
\medskip

\noindent{\bf Crossed product decomposition (as in Theorem \ref{crossed}):} By Lemma \ref{radical}, the kernel $I$ of $\pi$ is a Hopf ideal  both of $\mathcal{O}(G)$ and of $_J\mathcal{O}(G)_J$, and is generated (as a right or left ideal in both cases) by the elements $F_{23}$, $F_{13}$, $F_{24}$, and  $F_{14}$. It follows that the  same vector subspace $\mathcal{O}(G/T)$ of $\mathcal{O}(G)$ constitutes the left coinvariants of $\pi$ in both algebras. That is,
 $$\mathcal{O}(G/T) = {}^{\mathrm{co} \pi}\mathcal{O}(G)= {}^{\mathrm{co} \pi}({}_J\mathcal{O}(G)_J).$$
It is easy to calculate that the functions
\begin{equation}\label{enough} 
F_{23},\, F_{13}, \,Y := F_{24} - F_{23}F_{34}, \, V := F_{14} - F_{13}F_{34}
\end{equation}
are constant on the left cosets of $T$ in $G$, and so are contained in $\mathcal{O}(G/T)$. On the other hand these $4$ elements generate a left coideal subalgebra $I$ of $\mathcal{O}(G)$, and also generate $I$ as an ideal of $\mathcal{O}(G)$ and as a left or right ideal of ${}_J\mathcal{O}(G)_J$, so it follows from the bijection between left coideal subalgebras and Hopf ideals given by Masuoka's theorem for connected Hopf algebras (see \cite[Proposition 2.3, Theorem 3.2]{BG}) that $\mathcal{O}(G/T)$ in $\mathcal{O}(G)$, and ${}_J\mathcal{O}(G/T)_J$ in ${}_J\mathcal{O}(G)_J$ are both generated as subalgebras by the elements listed in (\ref{enough}).

The non-zero commutators of the generators (\ref{enough}) of ${}_J\mathcal{O}(G/T)_J$ are 
\begin{equation}\label{once} 
[Y,F_{13}] = F_{23}^2, \; [F_{13},V] = F_{23}F_{13}, \; [Y,V] =F_{23}Y.
\end{equation}
The algebra $U$ generated by the elements (\ref{enough}) \emph{modulo} the relations (\ref{once}) is a \emph{PBW-extension} of its commutative polynomial subalgebra $A := \mathbb{C}[F_{23},F_{13}]$, in the terminology of \cite[Definition, page 163]{KL}. Hence, by Matczuk's theorem \cite[Theorem A]{Mat} and \cite[Theorem 12.3.1]{KL}, we have
\begin{equation}\label{stun} \mathrm{GKdim}(U)= \mathrm{GKdim}(A) + 2  = 4.
\end{equation}
\noindent By construction of $U$, ${}_J\mathcal{O}(G/T)_J$ is a quotient of $U$. We claim that
\begin{equation}\label{cost} U = {}_J\mathcal{O}(G/T)_J.
\end{equation} 
If ${}_J\mathcal{O}(G/T)_J$ were a proper factor of $U$, then $\mathrm{GKdim}({}_J\mathcal{O}(G/T)_J) < 4$ by (\ref{stun}) and \cite[Proposition 3.15]{KL}. But in that case we would immediately arrive at a contradiction  from the smash product decomposition 
\begin{equation}\label{block} {}_J\mathcal{O}(G)_J \cong {}_J\mathcal{O}(G/T)_J \# U(\mathfrak{t}) = \mathbb{C}\langle F_{23},F_{13}, V,Y\rangle \# \mathbb{C}[F_{12}, F_{34}]
\end{equation}
of Theorem \ref{crossed}, via the equalities and inequalities 
$$ 6 = \mathrm{GKdim}({}_J\mathcal{O}(G)_J) =\mathrm{GKdim}({}_J\mathcal{O}(G/T)_J) + 2 <  4 + 2,$$
where the first equality follows from Corollary \ref{IHOEcor}(1) and the second from a second application of Matczuk's theorem, this time with $A = {}_J\mathcal{O}(G/T)_J$. This proves (\ref{cost}).

We note that the crossed product decomposition (\ref{block}) guaranteed by Theorem \ref{crossed} is in fact a non-trivial smash product: it is a smash product because the generators $F_{12}$ and $F_{34}$ of $\mathfrak{t}$ commute in ${}_J\mathcal{O}(G)_J$, while the action of $U(\mathfrak{t})$ on ${}_J\mathcal{O}(G/T)_J$ is non-trivial in view of the non-zero commutators
$$ [F_{12},V] = F_{13}, \, [F_{12}, Y] = F_{23}, \, [F_{34}, F_{13}] = F_{23},\, [F_{34},V] = Y.$$
\medskip

\noindent {\bf The strata ${}_J\mathcal{O}(Z_g)_J$:} Recall that for $g \in G$, the algebra $\mathcal{O}(Z_g)$ denotes the functions constant on the double coset $TgT$. Let
\begin{equation}\label{choice} 
g := g(x,y,z,a,b,u) = I + xE_{12} + yE_{13} + zE_{14} + aE_{23} + bE_{24} + uE_{34} 
\end{equation}
be a generic element of $G$.
\medskip

\noindent {\bf Case I:} $g \notin N_G(T)$. For {\bf Case I} we require by (\ref{normal}) that not all of $y, a$ and $b$ are $0$. One finds by routine calculations that there are two types of double coset with $g(x,y,z,a,b,u) \notin N_G(T)$, depending on whether or not $a= 0$.
\medskip

\noindent {\bf Case I(1):} $a= 0$. Let $g$ be as in (\ref{choice}), but with $a= 0$ and either $y$ or $b$ nonzero. One finds
$$ TgT = I + \mathbb{C}E_{12} +yE_{13} + \mathbb{C}E_{14} + bE_{24} + \mathbb{C}E_{34}.$$
Thus the defining equations for the double coset $TgT$ are
$$ F_{23} = 0,\, F_{13} = y,\, F_{24} = b.$$
Hence, abusing notation, we see that 
$$ {}_J\mathcal{O}(Z_g)_J  =  \mathbb{C}\langle F_{12}, F_{14}, F_{34} \rangle,$$
with defining relations
$$[F_{12},F_{14}] = y, \, [F_{34}, F_{14}] = b, \, [F_{12}, F_{34}] = 0.$$
Thus, $bF_{12} - yF_{34}$ is central in ${}_J\mathcal{O}(Z_g)_J$, and so
\begin{equation}\label{mixed} 
{}_J\mathcal{O}(Z_g)_J \cong \mathbb{C}[t] \otimes A_1 (\mathbb{C}).
\end{equation}

\noindent {\bf Case I(2):} $a\neq 0$. We can take as a double coset representative in this case 
$$ g(a):= I + aE_{23} $$
for $a\in \mathbb{C}^{\times}$. A calculation yields that the defining relations for the closed set $Tg(a) T$ in $G$ are 
$$ F_{23} - a, \,  F_{13}F_{24} - a F_{14}. $$
Thus, on the one hand, $Tg(a_0)T$ is a copy of affine $4$-space; secondly, according to Proposition \ref{homo} the right ideal of ${}_J\mathcal{O}(G)_J$ generated by the above two elements should be a two-sided ideal, namely $\mathcal{I}(Z_{g(a)})$. Well, from (\ref{twice}) we see that $F_{23} - a$ is central, and then an easy calculation shows that $F_{13}F_{24} - aF_{14}$ is central \emph{modulo} $\langle F_{23} - a\rangle$, so that  $\mathcal{I}(Z_{g(a)})$ is indeed a polycentral two-sided ideal of ${}_J\mathcal{O}(G)_J$.

To determine the structure of ${}_J\mathcal{O}(Z_{g(a)})_J$ consider first the algebra 
$$ R_{a}:={}_J\mathcal{O}(G)_J/\langle F_{23} - a \rangle. $$
Routine calculations which are left to the reader show that
$$R_{a}=\mathbb{C}\langle \overline{F_{12}},\overline{F_{24}} \rangle \otimes \mathbb{C}\langle \overline{F_{13}},\overline{F_{34}} \rangle \otimes \mathbb{C}[\overline{F_{13}F_{24} - a F_{14}}]\cong A_1(\mathbb{C}) \otimes A_1(\mathbb{C}) \otimes \mathbb{C}[t].$$
Finally, 
$${}_J\mathcal{O}(Z_{g(a)})_J\cong R_{a}/\langle \overline{F_{13}F_{24} - aF_{14}}\rangle, $$
so that, for all $a \in \mathbb{C}^{\times}$, we have
$${}_J\mathcal{O}(Z_{g(a)})_J \cong  A_2(\mathbb{C}). $$
As one easily checks, $ T \cap T^{g(a)}=\{1\}$ so that we are in {\bf Case 1} of $\S$\ref{stratstruc}, with the structure of ${}_J\mathcal{O}(Z_{g(a)})_J$ is as predicted there, namely it is $A_{\mathrm{dim}(T)}(\mathbb{C})$.
\medskip

\noindent {\bf Case II:} $g \in N_G(T)$.  In view of (\ref{normal}) we can take for $z \in \mathbb{C}$, $ g(z):= I + z E_{14}$ as a coset representative. So 
$$Tg(z)T =  Tg(z) =  T + z E_{14},$$
and it follows that $ \mathcal{I}(Tg(z_0)T)$  is the right ideal generated (both in $\mathcal{O}(G)$ and in ${}_J\mathcal{O}(G)_J$) by  $F_{23}, F_{13}, F_{24}$, and $F_{14} - z $, this being a two-sided ideal of both algebras by Proposition \ref{homo}. Thus,
\begin{equation}\label{shard} 
{}_J\mathcal{O}(Z_{g(z)})_J = {}_J\mathcal{O}(G)_J/\mathcal{I}(Tg(z)) = \mathbb{C}[ \overline{F_{12}}, \overline{F_{34}}], 
\end{equation}
a (commutative) polynomial algebra in $2$ variables, isomorphic to $U(\mathfrak{t})$ as predicted in Case II of $\S$\ref{stratstruc}, given the centrality of $g(b)$ in $G$. \qed
\end{example}

\begin{example}(\cite[Example 6.5]{G2})\label{ex6} 
Let $G={\rm U}(4,\mathbb{C})$ be as in Example \ref{ex5}, and retain the notation $\mathcal{O}(G) = \mathbb{C}[F_{12}, F_{13}, F_{14}, F_{23}, F_{24}, F_{34}]$ from there. But now we take $T$ to be the $4$-dimensional subgroup which was denoted by $G$ in Examples \ref{ex3} and \ref{ex4}, retaining the notation $\mathcal{O}(T) = \mathbb{C}[X,Y,V,W]$ and $\mathfrak{t} = \mathbb{C}a +\mathbb{C}b +\mathbb{C}c + \mathbb{C}d$ from these two examples. Take 
$$ r:= a\wedge c + b \wedge d \in\wedge^2 \mathfrak{t} $$
as in Example \ref{ex4}, so $J$ is a minimal Hopf $2$-cocycle for $T$, with values on the generator pairs of $\mathcal{O}(T)$ as displayed at the beginning of Example \ref{ex4}. The inclusion of $T$ in $G$ yields the Hopf algebra surjection $\pi: \mathcal{O}(G) \twoheadrightarrow \mathcal{O}(T)$ with
$$ F_{12}\mapsto X,\; F_{13}\mapsto V, \;  F_{14}\mapsto W,\; F_{23},F_{34}\mapsto Y,\; F_{24}\mapsto Y^2/2.$$

\noindent{\bf Algebra structure:} Pulling $J$ back along $\pi$ produces a non-minimal Hopf $2$-cocycle for $G$ whose non-zero values on generator pairs are
\begin{align*} 
&J(F_{12},F_{13}) = J^{-1}(F_{13},F_{12}) = J(F_{14},F_{23})= J^{-1}(F_{23},F_{14}) = \\
&J(F_{14},F_{34}) = J^{-1}(F_{34},F_{14}) = 1/2;\\
&J(F_{13},F_{12})  = J^{-1}(F_{12},F_{13})  = J(F_{23},F_{14})= J^{-1}(F_{14},F_{23})  =\\ &J(F_{34},J_{14})  = J^{-1}(F_{14},F_{34})  =  - 1/2.
\end{align*}

\noindent From Theorem \ref{Ore}(3),(4) we calculate that ${}{}_J\mathcal{O}(G)_J$ is an IHOE on the generators $\{F_{ij} \mid 1 \leq i \leq 4, i < j \leq 4 \}$, whose non-zero commutators are 
\begin{equation}\label{full} 
[F_{14}, F_{12}] =  F_{34}, \;  [F_{14}, F_{13}]= F_{12} + F_{23}F_{24} - F_{24}, \;  [F_{14}, F_{24}] = F_{23} - F_{34}. 
\end{equation}
Thus, ${}_J\mathcal{O}(G)_J$ can be presented as a Hopf Ore extension of derivation type, namely
$$ {}_J\mathcal{O}(G)_J\cong\mathbb{C}[F_{12},F_{13},F_{23},F_{24}, F_{34}][F_{14}; \partial].$$
Here, the coefficient algebra $\mathbb{C}[F_{12},F_{13},F_{23},F_{24}, F_{34}]$ is the coordinate ring of the factor $G/Z$, where $Z = \{I + wE_{14}\mid w \in \mathbb{C}\}$ is the center of $G$, and the definition of the derivation $\partial$ is immediate from the relations displayed above.
\medskip

\noindent {\bf Crossed product decomposition:} Routine calculations confirm that the elements $F_{34} - F_{23}$ and $2F_{24} - F_{23}^2$ are in $\mathcal{O}(G/T)$. It is easy to check that they generate a left coideal subalgebra of $\mathcal{O}(G)$, and together generate the radical $I_{R^J} = \mathrm{ker}(\pi)$. So, applying Masuoka's bijection as in Example \ref{ex5}, we find that 
$$ {}_J\mathcal{O}(G/T)_J=\mathcal{O}(G/T)=\mathbb{C}[F_{34} - F_{23}, 2F_{24} - F_{23}^2]. $$
Either by direct calculation or by using the antipode, we obtain similarly 
$${}_J\mathcal{O}(T\backslash G)_J =\mathcal{O}(T\backslash G) = \mathbb{C}[F_{34} - F_{23}, 2F_{23}F_{34} - 2F_{24} - F_{23}^2].$$
By Theorem \ref{crossed}, using the notation $\overline{X}$ for images of elements $X$ of ${}_J\mathcal{O}(G)_J$ under $\pi$, we obtain
\begin{align*} 
{}_J\mathcal{O}(G)_J  &\cong  {}_J\mathcal{O}(G/T)_J \#_{\sigma} {}_J\mathcal{O}(T)_J\\
&\cong  \mathbb{C}[F_{34} - F_{23}, 2F_{24} - F_{23}^2] \#_{\sigma} \mathbb{C}\langle \overline{F_{12}}, \overline{F_{13}}, \overline{F_{23}},\overline{F_{14}}\rangle.
\end{align*}
We have already seen in Example \ref{ex4} that ${}_J\mathcal{O}(T)_J \cong U(\mathfrak{t})$, as required by Theorem \ref{minLie}.
\medskip

\noindent {\bf Finite dimensional simple modules:} From the defining relations (\ref{full}) of ${}_J\mathcal{O}(G)_J$ it is immediate that its commutator ideal $\langle [{}_J\mathcal{O}(G)_J, {}_J\mathcal{O}(G)_J]\rangle$ contains the elements $A:= \{F_{34}, F_{23}, F_{12} - F_{24}\}$. Let $I$ be the ideal of ${}_J\mathcal{O}(G)_J$ generated by $A$, and observe that the elements of $A$ in the given order form a polycentral set of generators for $I$ (in the terminology of $\S$\ref{crossedconseq}). This implies that $A$ generates the same ideal $I$ in ${}_J\mathcal{O}(G)_J$ and in $\mathcal{O}(G)$, necessarily a Hopf ideal, with
$${}_J\mathcal{O}(G)_J/I = \mathcal{O}(G)/I =  \mathbb{C}[\overline{F_{12}},\overline{F_{13}}, \overline{F_{14}}],$$ 
where the first equality (of Hopf algebras) is seen to hold by checking that the commutators of all the generators are $0$ in ${}_J\mathcal{O}(G)_J/I$.  Thus, $I = \langle [{}_J\mathcal{O}(G)_J, {}_J\mathcal{O}(G)_J]\rangle$, and so, reading off from the generating set $A$, the group $\Gamma$ of $1$-dimensional ${}_J\mathcal{O}(G)_J$-modules is
$$ \Gamma := \{I +  x(E_{12} + E_{24}) +  vE_{13} + wE_{14}\mid x, v,w \in \mathbb{C}\} \cong (\mathbb{C},+)^3.$$
We check that $\Gamma = C_0$ and $\Gamma \cap T = \{I  + vE_{13} + wE_{14}\mid  v,w \in \mathbb{C}\} = F$, and 
$$ N :=  N_G(T) = \{ T + zE_{24}\mid  z \in \mathbb{C} \}, $$
so that $C = \Gamma T = N$. 
\medskip

\noindent {\bf The strata ${}_J\mathcal{O}(Z_g)_J$:} There are two cases to consider.

\noindent{\bf Case I:} $g \in N:=N_G(T)$. We may choose as a set of coset representatives for $N/T$ the elements 
$$ \{\gamma(x) := I + x(E_{12} + E_{24})\mid x \in \mathbb{C} \} \subset 
C_0 = \Gamma.$$
Since $\gamma(x) \in \Gamma$, the winding automorphism $\tau^{\ell}_{\gamma(x)}: \alpha \mapsto \sum \alpha_1\gamma(x)(\alpha_2)$, $\alpha \in \mathcal{O}(G)$, is an algebra automorphism of ${}_J\mathcal{O}(G)_J$, as well as of $\mathcal{O}(G)$, and maps the ideal $\mathcal{I}(\gamma(0)) = \langle F_{34} - F_{23}, 2F_{24} - F_{23}^2 \rangle$ to $\mathcal{I}(\gamma(x))$, hence showing that, as algebras,
$${}_J\mathcal{O}(Z_{\gamma(x)})_J \cong {}_J\mathcal{O}(Z_{\gamma(0)})_J= {}_J\mathcal{O}(T)_J \cong U(\mathfrak{t}).$$

\noindent{\bf Case II:} $g \notin N_G(T)$. Consider the elements
$$ \{\beta(u) := I + uE_{34}\mid  u \in \mathbb{C}^{\times} \} \subset G \setminus N.$$
One calculates that $T\beta(u)T  = N\beta(u)$, so that, together with the Case I cosets $T\gamma(x)$, these yield all the double cosets of $T$ in $G$. Since the defining ideal of $N$ in $\mathcal{O}(G)$ is $\langle F_{34} - F_{23} \rangle$, we see that 
$$\mathcal{I}(Z_{\beta(u)}) 
= 
\langle F_{34} - F_{23} - u \rangle. $$
The generator $F_{34} - F_{23} - u$ is central in ${}_J\mathcal{O}(G)_J$, so that 
$${}_J\mathcal{O}(Z_{\beta(u)})_J = {}_J\mathcal{O}(G)_J/(F_{34} - F_{23} - u){}_J\mathcal{O}(G)_J =  \mathbb{C}\langle \overline{F_{12}}, \overline{F_{13}}, \overline{F_{34}}, \overline{F_{24}}, \overline{F_{14}}\rangle.$$
Define new generators for ${}_J\mathcal{O}(Z_{\beta(u)})_J$ as follows:
$$ 
P := \overline{F_{24}} - \overline{F_{34}}(\overline{F_{34}} - u),\,\, Q:= \overline{F_{14}},\,\,
Z := \overline{F_{34}}, \,\, X :=  -\overline{F_{12}},$$ 
$$Y :=  \overline{F_{13}} - (P - \overline{F_{12}})P + \frac{1}{2}(\overline{F_{34}} + 1)P^2.$$
One can now easily check that
$$ {}_J\mathcal{O}(Z_{\beta(u)})_J =  \mathbb{C}[X,Y,Z] \otimes \mathbb{C}\langle P,Q\rangle \cong \mathbb{C}[X,Y,Z] \otimes A_1(\mathbb{C}).\quad\quad\quad\quad\quad\quad\quad\quad\quad \qed$$
\end{example}

\section{Acknowledgements} 

We thank Pavel Etingof for helpful correspondence. The work of K.B. was supported by Leverhulme Emeritus Fellowship EM 2017-081. The work of S.G. was supported by Simons Foundation Award 963288.

\end{document}